\documentclass{article}
\usepackage[utf8]{inputenc}
\usepackage[english]{babel}

\parskip \smallskipamount

\RequirePackage{etex}
\usepackage{blindtext}
\usepackage{amssymb}
\usepackage{hyperref}
\hypersetup{
    colorlinks=true,
    linkcolor=blue,
    filecolor=magenta,      
    urlcolor=cyan,
}
 
\usepackage{makeidx}
\usepackage{amsmath,amsthm,amssymb}
\usepackage{mathtools}
\usepackage{amsfonts,bbm}
\usepackage{comment}
\usepackage{hyperref}
\usepackage{mathtools}
\usepackage{graphicx}
\usepackage{enumerate}
\usepackage[english]{babel}
\theoremstyle{plain}
\usepackage{tikz}
\newcommand{\itref}[1]{\textit{(\ref{#1})}}
\newcommand{\N}{\mathbb{N}}
\newcommand{\Z}{\mathbb{Z}}

\newcommand{\s}{\mathfrak{s}}
\usepackage{array}
\usepackage{tikz-network}
\usepackage{faktor}
\usetikzlibrary{positioning,arrows.meta} 
\usetikzlibrary{shadows.blur}
\usepackage{bm}
\usetikzlibrary{automata, positioning, arrows}
\newtheorem{theorem}{Theorem}[section]
\newtheorem{proposition}[theorem]{Proposition}
\newtheorem{corollary}[theorem]{Corollary}
\newtheorem{lemma}[theorem]{Lemma}
\theoremstyle{remark}

\newtheorem{remark}{Remark}

\theoremstyle{definition}
\newtheorem{definition}{Definition}[section]

\newcommand{\geom}{\operatorname{Geom}}
\newcommand{\diam}{\mathrm{diam}}

\newcommand{\cA}{\ensuremath{\mathcal A}}
\newcommand{\cB}{\ensuremath{\mathcal B}}
\newcommand{\cC}{\ensuremath{\mathcal C}}

\newcommand{\cE}{\ensuremath{\mathcal E}}
\newcommand{\cF}{\ensuremath{\mathcal F}}

\newcommand{\cI}{\ensuremath{\mathcal I}}

\newcommand{\cL}{\ensuremath{\mathcal L}}

\newcommand{\cN}{\ensuremath{\mathcal N}}

\newcommand{\cP}{\ensuremath{\mathcal P}}

\newcommand{\cS}{\ensuremath{\mathcal S}}
\newcommand{\cT}{\ensuremath{\mathcal T}}

\newcommand\numberthis{\addtocounter{equation}{1}\tag{\theequation}}

\newcommand{\bbE}{{\ensuremath{\mathbb E}} }

\newcommand{\bbP}{{\ensuremath{\mathbb P}} }

      \let\e=\varepsilon
        
\let\m=\mu            
\let\r=\rho

\title{The critical density of the Stochastic Sandpile Model}
\author{Concetta Campailla\thanks{Sapienza Universit\`a di Roma, Dipartimento di Matematica, Roma, Italy.}
\and
Nicolas Forien\thanks{CEREMADE, CNRS, Université Paris-Dauphine, Université PSL,
75016 Paris, France}
\and
Lorenzo Taggi\footnotemark[1]}

\begin{document}

\maketitle

\begin{abstract}
We study the stochastic sandpile model on $\mathbb{Z}^d$ and demonstrate that the critical density is strictly less than one in all dimensions. This generalizes a previous result by Hoffman, Hu, Richey, and Rizzolo (2022), which was limited to the one-dimensional case. In addition, we show that the critical density is strictly positive on any vertex-transitive graph, extending the earlier result of Sidoravicius and Teixeira (2018) and providing a simpler proof.
\end{abstract}

\section{Introduction}

Many natural and complex phenomena, such as earthquakes, snow avalanches, forest fires, and others, exhibit sudden large-scale events triggered by small disturbances. Bak, Tang, and Wiesenfeld proposed a general explanation for these phenomena, which they referred to as \textit{self-organized criticality} \cite{BTW}. Their research brought about a major shift in understanding and became one of the most frequently referenced papers in physics throughout the following decade.

Various processes have been suggested as universal models of self-organized criticality \cite{Bak}. Bak, Tang, and Wiesenfeld introduced the abelian sandpile model, which operates through a deterministic process of redistribution \cite{BTW}. The model’s intricate mathematical framework has been studied extensively by various branches of mathematics. However, physicists later observed that the model lacks the universality typically expected from systems displaying self-organized criticality. For example, it is highly dependent on the initial conditions and other details \cite{FLW10, Dhar1, JJ10}.

Manna introduced a probabilistic variant of the abelian sandpile model
\cite{Ma}, later known as the \textit{stochastic sandpile model}
(SSM). This version demonstrates more robust universality properties
compared to the abelian sandpile. It is widely believed that this
model is the appropriate framework for investigating the phenomenon of
self-organized criticality and its universal characteristics
\cite{Dick, Lub, HPC}. However, its rigorous mathematical analysis presents various challenges. For this reason, despite some progress, in recent years, the focus of the mathematics community has shifted to a close variant of SSM, the activated random walk model (ARW), which is more tractable mathematically and is believed to belong to the same universality class as SSM \cite{Lub}.
In particular, progress has been made toward understanding the
properties of the critical curve for ARW. These advancements include
the demonstration that it is strictly between zero and one~\cite{AFG,
FG, Hu, Basu, HRR, RS, ST}, continuity  as a function of the deactivation
parameter \cite{TaggiEssential} and universality \cite{RSZ}.
Furthermore, the equivalence of critical thresholds has been
established in one dimension \cite{Hoffman2024} and explored in wider
generality \cite{LS}. Additionally, some
bounds on the mixing time of the so-called 'driven-dissipative'
dynamics have been provided \cite{Lev21, BS}.

In contrast, our understanding of the stochastic sandpile model
remains more limited. In particular, one of the key open questions for
this class of models --- whether the critical density is below one --- remains
unresolved for any graph beyond \(\mathbb{Z}\) \cite{HHRR}. 
In this
paper, we resolve this conjecture for \(\mathbb{Z}^d\) in any
dimension \(d \geq 1\). Additionally, we provide a new, simplified
proof that the critical density is strictly positive on any
vertex-transitive graph, thereby extending previous results \cite{PR,
RS, ST12}. 
The central tool in deriving the upper bound for the critical threshold is a novel stochastic comparison with a variant of the activated random walk model. This comparison enables us to apply recently developed techniques from the analysis of ARW to the stochastic sandpile model. 
We hope that this new stochastic comparison may help to make further progress in the rigorous understanding of the stochastic sandpile model in the future.

\subsection{Definition and main results}

Let us consider a finite graph $G = (V, E)$. At time zero, each vertex has a number of particles, which is an independent Poisson random variable with mean $\mu > 0$. At any time, each vertex is declared \textit{stable} if it hosts at most one particle and \textit{unstable} otherwise. The dynamics evolve in continuous time. Each vertex is associated with an independent Poisson clock with rate one. When the clock rings, if the vertex is unstable, two particles from that vertex perform an independent simple random walk step, while if the vertex is stable, nothing happens.

On infinite vertex-transitive graphs (like, for example, $\mathbb{Z}^d$), the dynamics can be defined by taking the limit of large finite balls with absorbing boundary conditions \cite{Andjel, RS}. Criticality of the SSM on infinite graphs is defined with respect to whether the system remains active or not. We say that the system \textit{locally fixates} if every vertex is visited only finitely many times by the particles and eventually becomes stable; otherwise, we say that the system stays \textit{active}.

Rolla and Sidoravicius \cite{RS} showed in wide generality the
existence of a critical value $\mu_c \geq 0$ such that if $\mu >
\mu_c$, the system is almost surely active, and if $\mu < \mu_c$, the
system fixates almost surely. They proved that $\mu_c \in
[\frac{1}{4}, 1]$ in $\mathbb{Z}$. The lower bound in $\mathbb{Z}$ was
then improved to $\mu_c \geq \frac{1}{2}$ in \cite{PR}. Sidoravicius
and Teixeira \cite{ST12} employed renormalisation techniques to show that $\mu_c > 0$ in $\mathbb{Z}^d$ for each $d \in \mathbb{N}$.

Our first theorem shows that the critical density of the stochastic sandpile is strictly positive on vertex-transitive graphs.

\begin{theorem}\label{thm:theorem1}
    For any vertex-transitive graph $G$ of degree at least two, the critical density of the stochastic sandpile model satisfies \[ \mu_c > 0.\]
\end{theorem}

The novelty of the result is not only in extending the class of graphs for which it holds but also in presenting a simpler proof.

Our second theorem establishes a new upper bound for the critical density of the SSM.
It is generally not difficult to show that $\mu_c \leq 1$.
Indeed, if the particle density is strictly greater than one, there must necessarily exist a positive density of unstable sites at all times, ensuring that the system remains active.
It is therefore interesting to ask whether $\mu_c < 1$, i.e., whether the system is active for particle densities that are strictly below one (but sufficiently large).
Our next theorem provides a positive answer to this question in $\mathbb{Z}^d$ for all dimensions $d \geq 1$, thereby extending the previous results of \cite{HHRR}, which were limited to one dimension.

\begin{theorem}\label{thmupperbound}
    In any dimension $d \geq 1$, the critical density of the stochastic sandpile model on $\mathbb{Z}^d$ satisfies \[\mu_c < 1.\]
\end{theorem}

Theorem \ref{thmupperbound} is one of the main results of our paper. Together with Proposition \ref{propdom} below, which presents the new coupling argument, these form the core contributions of our work.

\subsection{Proof outline}

The proof of Theorem \ref{thm:theorem1} requires the adaptation to the stochastic sandpile model of the ``weak stabilisation'' method, which was introduced in \cite{ST} for the activated random walk model. The proof is short and is presented in Section \ref{sec:prooflowerbound}, together with a short description of the method.

For the proof of Theorem \ref{thmupperbound}, the new central idea is a stochastic comparison between the SSM and a version of the activated random walk model, which we refer to as the activated random walk with instantaneous deactivation (in short: ARWD), see Proposition \ref{propdom} below.

Similarly to the activated random walk model, in ARWD, particles can
be either active or sleeping. Active particles get activated whenever
they share a vertex with other active particles, and they fall asleep
with probability $\frac{\lambda}{1 + \lambda}$ independently at each
discrete step if they are alone, where $\lambda > 0$. Informally, the
main difference with the activated random walk model is that, in addition, the active particles fall asleep \textit{instantaneously} whenever they jump to a vertex.

The coupling between the two models requires moving particles ‘in parallel’ by using the same jump instructions. Even though in SSM particles can only be of one type, we can view each particle as ‘sleeping’ if it has just jumped to a vertex that has been toppled an even number of times. We think of a particle as ‘activated’ whenever such a vertex is visited by some other particle afterward. The coupling is then constructed in such a way that, whenever a particle becomes stable in SSM, the corresponding particle also falls asleep in ARWD. In some cases, it may happen that we are allowed to move a particle in SSM, but we are not allowed to move the corresponding particle in ARWD since it is sleeping. The opposite, however, never happens. This allows us to \textit{dominate from below} the activity in SSM by the activity in ARWD.

For the coupling to work, we need to restrict the set of vertices where the reactivation of sleeping particles is possible in both models, ignoring the fact that particles outside this set may still be active. Such a restriction is necessary to dominate from below using independent random variables the reactivation of sleeping particles in SSM within this set. The proof method also requires disallowing the deactivation of particles outside a certain arbitrary set of vertices $A$ in both models (for SSM, this is formalized through the introduction of the so-called '$A$-stabilisation', see Section \ref{sectionAstab} below) and then considering all possible choices of $A$. We refer to Section \ref{section-strategy-propdom} for a more formal description of the coupling argument.

Once the stochastic comparison is established, it remains to show that ARWD is active, thus deducing by the stochastic comparison that SSM is also active. For this, we adapt to ARWD the strategy which was developed in \cite{AFG}. We first show that for $\mu$ close enough to $1$ and less than one, the stabilisation time of the SSM on the torus is exponentially large:

\begin{theorem}
    \label{thm-exp-time}
    For every $d \geq 1$, for $\mu \in (0,1)$ close enough to $1$,
the stabilisation time $\mathcal{T}_n$
of the stochastic sandpile model on the torus $\mathbb{Z}_n^d = (\mathbb{Z}/n\mathbb{Z})^d$
started with i.i.d.\ Poisson numbers of particles on each site, with parameter $\mu$, satisfies
    \begin{equation}
    \label{exponential-time}
\exists\,c>0\quad
    \forall\,n\geq 1\qquad
    \mathbb{P}\big(\mathcal{T}_n < e^{c n^d}\big) < e^{-c n^d} \, .
    \end{equation}
\end{theorem}

Since it is already known that $\mu_c < 1$ for the stochastic sandpile model in dimension $d = 1$ \cite{HHRR}, we provide the detailed proof only for $d \geq 2$.
Theorem \ref{thm-exp-time} implies Theorem \ref{thmupperbound} through the following result, which is an analog of Theorem 4 in \cite{FG}:

\begin{theorem}
    \label{thm-torus-Zd}
    For every $d \geq 1$, for every $\mu < \mu_c$, where $\mu_c$ is the critical density of the stochastic sandpile model on $\mathbb{Z}^d$, the property \eqref{exponential-time} does not hold.
\end{theorem}


While the overall strategy to prove that ARWD is active is similar to
that in \cite{AFG}, the adaptation requires addressing some technical
challenges. These include the fact that, unlike in \cite{AFG}, not all
particles are active at time zero, the reactivation of sleeping
particles is restricted to a certain set of sites, and there is an
additional rule in ARWD governing when particles fall asleep. These
features make it harder for the model to be active in comparison to
the activated random walk. We believe these generalizations may have
some relevance on their own. We also remark that ARWD is an example of
a non-abelian particle system, and exploring this model may offer further insights.

 \begin{remark}
 The proofs in this paper can be adapted to a more general model in which an integer $k$ is specified, and a site is considered unstable if it contains at least $k$ particles; otherwise, it is stable. Each unstable site topples at rate $1$, sending $k$ of its particles to neighbours chosen uniformly and independently at random. Using our methods, we can then show that $\mu_c<k-1$.
 \end{remark}

 \paragraph*{Notation.}
 
We let $d_G ( \cdot, \cdot)$ be the graph distance in $G =(V,E)$,
$B(x,r)$ be the set of vertices which have graph distance at most $r$ from $x$.   
 We use the   notation $B(r)$ for $B(o,r)$, where $o \in V$ is the root of $G$.
 We also use the notation $\mathbb{N}_0 = \{0, 1, \ldots\}$ and
 $\mathbb{N} = \{1, 2, \ldots \}$.

\section{Diaconis-Fulton representation}
\label{sectionDiaconis}

In this section we describe the Diaconis-Fulton graphical representation for the dynamics of SSM and introduce several stabilisation procedures.

\subsection{Stabilisation}
 Let $G =(V, E)$ be an undirected,  locally finite, finite or infinite graph.  
 A configuration of the stochastic sandpile model on $G$ can be represented by a vector $\eta : V \to \N_0$,
 with $\N_0 =\N\cup \{0\}$, where $\eta(x)=k \in \N_0$ means that there are $k$ particles on the site $x$.
A site $x \in V$ is \textit{stable} in the configuration $\eta$ if $\eta(x) = 0$ or 1, and it is \textit{unstable} if $\eta(x) \geq 2$.
We say that a configuration $\eta$ is stable in a subset $V'\subset V$
if $\eta(x)\leq 1$ for all $x\in V'$.
 We now define, for every $x \in V$ and every $y$ neighbour of $x$ in the graph $G$, an operator $\tau_{xy}$ which corresponds to one particle jumping from $x$ to $y$: namely, if $\eta \in \N_0^V$ is such that $\eta(x)\geq 1$, we define
 \begin{equation*}
    \tau_{xy}\eta(z)=
    \begin{cases}
\eta(x)-1
&\quad\text{if }z=x\\
\eta(y)+1
&\quad\text{if }z=y\\
\eta(z) 
&\quad\text{otherwise.}
\end{cases}
\end{equation*}
Let us fix for now an array of instructions $\tau =
(\tau^{x,j}\,:\, x \in V,\, j \in \N)$,
where for every $x\in V$ and $j\in\N$,
 $\tau^{x,j} \in  \{ \tau_{xy}\,:\, y \sim x\}$.
Once this array is fixed, the configuration of the model at a certain time may be encoded in a couple $(\eta,\,h)$, where $\eta:V\to\N_0$ is the current configuration and $h:V\to\N_0$ is a function which counts the number of pairs of instructions already used at each site.
We say that we topple $x$ when we act on this couple $(\eta,\,h)$ 
through the operator $\Phi_x$, which is defined as
\begin{equation*}
    \Phi_x(\eta,h)=(\tau^{x, 2h(x)+2} \tau^{x, 2h(x)+1}  \eta,\, h+\mathbbm{1}_x)\,,
\end{equation*}
provided that $x$ is {unstable} in $\eta$.
The toppling operation $\Phi_x$ at $x$ is termed \textit{legal} for a couple $(\eta,\,h)$
 when $x$ is unstable in $\eta$, and \textit{illegal} otherwise.
We let $h = 0$ indicate $h(y) = 0$ for all $y \in V$, and we abbreviate $\Phi_x(\eta,0)$ as  $\Phi_x \eta$.

 Given a
sequence $\alpha = (x_1,\,x_2,\,\dots , \,x_k)$ of sites, we denote by $\Phi_\alpha = \Phi_{x_k} \Phi_{x_{k-1}} \cdots 
\Phi_{x_2} \Phi_{x_1}$
the composition of the topplings at the sites $x_1,\,x_2,\,\dots,\,x_{k-1},\,x_k$, in that order.
We call $\Phi_\alpha$ legal for a configuration $(\eta,\,h)$ if $\Phi_{x_1}$
is legal for $(\eta,\,h)$ and $\Phi_{x_\ell}$
is legal for
$\Phi_{(x_1,\,\dots,\,x_{\ell-1})}(\eta,\,h)$ for each $\ell = 2,\,\dots,\,k$. In this case we call $\alpha$ a legal sequence of
topplings for $(\eta,\,h)$.

For every $x\in V$, we define
\begin{equation*}
m_\alpha(x)=\sum_{\ell=1}^k \mathbbm{1}_{\{x_\ell=x\}}\,.
\end{equation*}
In other words, $m_\alpha(x)$ indicates the number of times that the site $x$ appears in the sequence $\alpha$. 
Given
a toppling sequence $\alpha= (x_1,\,x_2,\,\dots, \,x_k)$ and a
subset~$V'$ of~$V$, we say that
$\alpha$ is contained in $V'$, and we write $\alpha\subset V$, if all its elements are in $V'$,  and we say that $\alpha$ stabilises a given configuration $\eta$
in $V'$ if every site $x \in V'$ is stable in the particle configuration of $\Phi_\alpha\eta$.

For any subset $V'\subset V$, any particle configuration $\eta$  and any array of
 instructions $\tau$, we define the \textit{odometer}
 \begin{equation*}
 m_{V',\eta}
 = 2 
 \sup_{\alpha\subset V',\,\alpha\text{ legal for  $(\eta, 0)$}}\,
 m_\alpha.
 \end{equation*}
 In the particular case $V'=V$, we simply write $m_{\eta}=m_{V,\eta}$.
 It follows from the Abelian property (see \cite[Lemma 2]{RS}) that, if there exists a legal sequence $\alpha\subset V'$ which stabilises $\eta$ in $V'$, then $m_{V',\eta}=2 m_\alpha$, that is to say, $m_{V',\eta}$ counts the number of times that a particle jumps from $x$  in the stabilisation of $V'$ starting from configuration $\eta$ using the instructions in $\tau$, using whatever legal stabilising sequence.
For any $ V'' \subset V' \subset V$, we write
\begin{equation*}
    \|m_{V', \eta}\|_{V''}=\sum_{x \in V''}m_{V', \eta}(x)
\end{equation*}
for the number of topplings on the sites of $V''$ during the stabilisation of $\eta$ in $V'$. In the particular case
where $V'' = V'$, we simply write $\|m_{V', \eta}\|=\|m_{V', \eta}\|_{V'}$ and, if $V''=V'=V$, we write $\|m_{\eta}\|$. 
If
$\|m_{V',\eta}\|<\infty$, then the Abelian property ensures that the configuration reached when stabilising $\eta$ in $V'$
does not depend on the toppling sequence used and thus we can denote by $\eta_{V',\infty}$ the final configuration.  In the case where  $V' = V$, we simply write $\eta_{\infty}$.

We now introduce a probability measure on the space of instructions and of particle configurations. 
Let $\cP$ be a probability measure on the set of all possible fields of instructions,
which is such that the instructions $(\tau^{x,j} )_{x\in V, j \in \N}$, are independent, with, for every $x \in V$ and any $j\in \N$, $\cP^\mu(\tau^{x,j}=\tau_{xy})=\frac{1}{d_x}$  for any $y \in V$ neighboring $x$, where
 $d_x$ is the degree of the vertex $x$ in $G$.

  The initial distribution $\eta$ is i.i.d.\ with a Poisson number of particles with parameter $\mu$ on each site. Let us
 denote by $P_\mu$ the corresponding probability measure on the set $\N_0^V$.
We write $\cP^\mu= \cP \otimes P_\mu$ for the joint law of $\tau$ and $\eta$, and we let $\bbP^\mu$ denote the probability measure induced by the SSM process.
That is to say, compared to $\cP^\mu$, the measure $\bbP^\mu$ also contains the randomness relative to the times at which each site topples.
 Write $E$, $E^\mu$, $\bbE^\mu$ for the expectations with respect to the measures $\cP$, $\cP^\mu$, $\bbP^\mu$.

The following lemma gives a 0-1 law for local fixation and relates it to a property of the stabilisation odometer $m_{\eta}$:
 
\begin{lemma}[Lemma 4 in \cite{RS}]
\label{0-1law}
Let $G =(V,E)$ be an infinite vertex-transitive graph and let $x \in V$ be arbitrary. 
Then,
\begin{equation*}
    \bbP^{\mu}(\text{the system locally fixates})=\mathcal{P}^\mu(m_{\eta}(x)<\infty) \in \{0,1\}
    \,.
\end{equation*}
\end{lemma}
This lemma tells us that, in order to establish local fixation almost
surely, it is enough to show that a site is toppled a finite number of
times with positive probability.

\subsection{Stabilisation with half-topplings}
Given an array of instructions $\tau=(\tau^{x,j}: x \in V,\, j \in \N)$ and a configuration $\eta$, we introduce a stabilisation strategy with the half-toppling operation, which moves the
particles one by one.

A \textit{half-toppling} consists in sending only one particle to a neighboring site chosen
at random (i.e., following the appropriate instruction) and is possible when $\eta(x) \geq 1$. Fix an array of instructions $\tau$. We say we half-topple $x$ when we act on a  particle configuration $\eta$  satisfying $\eta(x) \geq 1$ and on the vector $h$ through the operator $\phi_x$, which is defined as
\begin{equation*}
    \phi_x(\eta,h)=(\tau^{x,2 h(x)+1} \eta,\, h+  \frac{1}{2}\mathbbm{1}_x)\,.
\end{equation*}
As before, we abbreviate $\phi_x(\eta, 0)$ by $\phi_x \eta$.
For a particle configuration~$\eta:V\to\N_0$ and an
odometer~$h:V\to(1/2)\N_0$, we say that a site~$x$ is half-unstable
in~$(\eta,\,h)$ if either~$\eta(x)\geq 2$ or~$\eta(x)=1$
and~$h(x)\notin\N_0$ (and otherwise the site is said half-stable).
For a sequence $\alpha=(x_1,\,\dots,\,x_k)$
we
denote, for any $\ell \in \{1,\dots, k\}$, \smash{$\alpha^{(\ell)}=(x_1,\,\dots,\, x_\ell)$},
and we define the operator $\phi_\alpha$ as the composition
$\phi_{x_k}  \phi_{x_{k-1}} \cdots  \phi_{x_1}$.
We say that the sequence $\alpha=(x_1,\,\dots, \,x_k)$ is
\textit{admissible} for $(\eta,\,h)$ if for any $\ell \in \{1,\,\dots,
\,k\}$, the particle configuration of~$\phi_{\alpha^{(\ell)}}(\eta,\,h)$ contains
at least one particle at~$x_\ell$.
Given $\alpha$ admissible for $\eta$, we say that $\alpha =(x_1,\,\dots, \,x_k)$ is
\textit{legal} for the stabilisation with half-topplings of
$(\eta,\,h)$ (in short: {half-legal}) 
 if for every $\ell \in \{1,\,\dots, \,k\}$, the site~$x_\ell$ is
half-unstable in~$\phi_{\alpha^{(\ell)}}(\eta,\,h)$.
 Moreover, we say that the sequence of topplings $\alpha$ 
 \textit{stabilises} $(\eta,\,h)$ in $V^{\prime}$ via half-topplings (in short: it is half-stabilising) if it is half-legal for $\eta$,
 it is contained in $V^\prime$ and every site $x \in V^{\prime}$ is
half-stable in~$\phi_\alpha(\eta,\,h)$.
We have
\begin{equation*}
 m_{V',\eta}
 =  
 \sup_{\alpha\subset V',\,\alpha\text{ half-legal for $\eta$}}\,
 m_\alpha,
 \,
 \end{equation*}
for any $V^\prime \subset V$
and the Abelian property also holds for the stabilisation with half-topplings
 (see 
\cite[Section 6]{RS} for details).
Thus, 
$m_{V',\eta}$ counts the number of times that a particle jumps from $x$ both  in the
stabilisation and in the half-stabilisation of $\eta$ in $V'$ using the instructions in $\tau$, using whatever stabilising or half-stabilising sequence  $\alpha$.

\subsection{Modified stability rules with respect to a fixed settling set}
\label{sectionAstab}
\begin{definition}
Given a subset $A\subset V$, a site $x\in A$ is said  \textit{$A$-stable} for a couple $(\eta,\,h)$ if $\eta(x)=0$ or if $\eta(x)=1$ and $h(x)\in \N_0$, whereas a site $x\in V\setminus A$ is said $A$-stable for $(\eta,\,h)$ if $\eta(x)=0$.
Otherwise, the site is said~$A$-unstable.
We call an admissible sequence of half-topplings $\alpha=(x_1, \dots, x_k)$ \textit{A-legal} for $\eta$ if for any $\ell \in \{1\, \dots, k\}$, the site~$x_\ell$ is~$A$-unstable in~$\phi_{\alpha^{(\ell)}}\eta$.
  We say that a couple $(\eta,\,h)$ is \textit{A-stable} in a set
$V'\subset V$ if all the sites of $V'$ are $A$-stable for $(\eta,\,h)$.
\end{definition}
\begin{definition}
\label{def-A-stab}
We say
that a sequence of half-topplings $\alpha$ \textit{A-stabilises}
$\eta$ in $V'\subset V$ if $\alpha$ is admissible and $\phi_\alpha
\eta$ is $A$-stable in $V'$.
We also define the odometer of the $A$-stabilisation of $V'$ as
\[
m_{V',\eta}^A
=
\sup_{\alpha\subset V',\ \alpha \text{ is }  A\text{-legal}}
m_\alpha
\,.
\]
In the particular case $V'=V$, we omit $V'$ and simply write
${m}^A_{\eta}$. 
\end{definition}

Similarly to stabilisation and half-stabilisation,   the Abelian
property   also holds for the $A$-stabilisation with minor changes in the proofs (see \cite[Section 3]{RS}).
For any $ V''\subset V'$, denote by
\begin{equation*}
    \|{m}^A_{V', \eta}\|_{V''}=\sum_{x \in V''}{m}^A_{V', \eta}(x)
\end{equation*}
the number of half-topplings on the sites of $V''$ during the
$A$-stabilisation of $\eta$ in $V'$. In the particular case
where $V'' = V'$, we simply write \smash{$\|{m}^A_{V', \eta}\|=\|{m}^A_{V',
\eta}\|_{V'}$} and, if $V''=V'=V$, we write \smash{$\|{m}^A_{\eta}\|$}. 
If \smash{$\|{m}^A_{V', \eta}\|< \infty$}, which means that $\eta$ can be
$A$-stabilised in $V'$ with
a finite number of half-topplings, then the Abelian property ensures
that the configuration reached when stabilising $\eta$ in $V'$
does not depend on the toppling sequence used, which allows us to
write this final configuration $\eta^A_{V', \infty}$.  In the case
where  $V' = V$, we simply write $\eta^A_{\infty}$. Note that if
$V'=V=A$ then $m_\eta^A={m}_{\eta}$ and $\eta_{\infty}^A=\eta_\infty$.

\begin{lemma}
\label{lemma-A-stab}
If $\eta_\infty=\mathbbm{1}_A$ then $m_\eta=m_{\eta}^A$.
\end{lemma}

\begin{proof}
  If $\eta_\infty=\mathbbm{1}_A$ then there exists a half-legal
sequence $\alpha$ such that $\phi_\alpha\eta=(\mathbbm{1}_A,\,h)$ with
$h(x)$ integer for every $x\in A$.
    Then we have $m_\eta=m_\alpha$.
    
    Yet, this final configuration $(\mathbbm{1}_A,\,h)$ is not only stable but also $A$-stable.
    Besides, as any half-legal sequence, this sequence $\alpha$ is also $A$-legal.
    Thus, $\alpha$ is a $A$-legal sequence which $A$-stabilises $\eta$, whence by the Abelian property $m_\eta^A=m_\alpha=m_\eta$.
\end{proof}

\section{Proof of Theorem \ref{thm:theorem1}}
\label{sec:prooflowerbound}
The proof of Theorem \ref{thm:theorem1} requires the introduction of the weak stabilisation technique for the stochastic sandpile model. This method was introduced for the activated random walk model in \cite{ST}.

\subsection{Weak stabilisation}
We introduce the weak stabilisation strategy for the stochastic
sandpile model. Given an array of instructions $\tau=(\tau^{x,j}: x
\in V,\, j \in \N)$, a configuration $\eta$ and $\alpha$ a legal
sequence for $(\eta, 0)$, we use the notation $\eta_\alpha$ for the particle configuration which is obtained by toppling the vertices of the sequence $\alpha$, i.e., the first component of the pair $\Phi_\alpha(\eta)$,
using the convention that $\eta_{\varnothing} = \eta$. 
For any $x \in V$, let $I_x=B(x,1)$ be the set of vertices whose graph distance from $x$ is at most one.

\begin{definition}[Weakly stable configurations]
Given a set $V^\prime \subset V$, we say that the configuration $\eta$
is \textit{weakly stable} in $V^\prime$ with respect to $x \in
V^\prime$ if one of the following two conditions is fulfilled:
\begin{enumerate}
\item[(i)] $\eta$ is stable in $V^\prime$,
\item[(ii)]  $\eta$ is stable in $V^\prime \setminus I_x$, there exists a site $z \in I_x$ such that $\eta(z)= 2$ and the other sites of $I_x$ are empty in $\eta$. 
\end{enumerate}
For conciseness, we simply say that $\eta$ is weakly stable in $(x, V^\prime)$.
\end{definition}

\begin{definition}[Weak stabilisation]
Given a set $V^\prime \subset V$, a vertex $x \in V^\prime$, and an array $\tau$, we say that the sequence of topplings $\alpha =(x_1, \ldots, x_k)$ is \textit{legal for the weak stabilisation} of $\eta$ in $V^\prime$ with respect to $x$ (in short: $x$-weakly legal) if it is legal for $(\eta, 0)$ and if 
\[
x_\ell \in I_x  \quad  \mbox{and} \quad 
 \eta_{\alpha^{(\ell-1)}} (x_\ell) = 2
\implies \exists z \in I_x \,  :\,  z  \neq x_\ell \mbox{ and }  \eta_{\alpha^{(\ell-1)}} (z) \geq 1,
\]
for each $\ell \in \{1, 2, \ldots, k\}$.
We say that the sequence of topplings $\alpha =(x_1, \ldots, x_k)$ \textit{weakly stabilises $\eta$ in $(x, V^\prime)$} if it is contained in $V^\prime$, it is $x$-weakly legal for $(\eta, 0)$, and $\eta_\alpha$ is weakly stable in $(x, V^\prime)$.
\end{definition}
Similarly to stabilisation, half-stabilisation and $A$-stabilisation, the Abelian property holds also for the weak stabilisation
with minor changes in the proofs (see [21, Section 3]).
An $x$-weakly legal sequence ensures that, whenever a vertex in $I_x$ is toppled, there will still be a particle in $I_x$ in the subsequent step. This observation leads to the following lemma.
\begin{lemma}\label{lemma:weaklyproperty}
For arbitrary $V^\prime \subset V$ and $x \in V^\prime$, let $\alpha =(x_1, \ldots, x_k)$ be a sequence which weakly stabilises $\eta$ in $(x, V^\prime)$.
Suppose that there exists $\ell \in \{0, 1, \ldots, k \}$ such that $\eta_{\alpha^{(\ell)}} (z) \geq 1$ for some $z \in I_x$.  Then there exists $z \in I_x$ such that  $\eta_\alpha(z) \geq 1$.
\end{lemma}
\begin{proof}
The lemma follows  from the definition of weakly legal sequence. Indeed, according to the definition of a $(x, V^\prime)$-weakly legal sequence, we can only topple a vertex in $I_x$ if it contains at least three particles, or if it contains exactly two particles and another vertex in $I_x$ hosts at least one particle. These conditions ensure that, if a particle jumps to $I_x$ at any point, $I_x$ will continue to host at least one particle at the end of the weak stabilisation, when we obtain a weakly stable configuration.
\end{proof}

Weak stabilisation is used for the proof of the next proposition, providing a lower bound for the probability that stabilisation ends with at least one particle in $I_x$.

\begin{proposition} \label{propweak}
Let $G = (V,E)$ be a regular graph with degree $d$. Let $A_{x, L}$ be the event that there is at least one particle in $I_x$ at the end of the stabilisation of $B(L)$. Then, for any $L \in \mathbb{N}$ and $x \in B(L)$,
\begin{equation*}
    \mathcal{P}^\mu(A_{x, L})\geq \frac{d-1}{d^3}-\mathcal{P}^\mu(m_{ B(L), \eta}(x)=0).
\end{equation*}
\end{proposition}

The strategy for the proof of Proposition \ref{propweak} is to
stabilise $B(L)$ by first weakly stabilising $(x, B(L))$, which gives
a weakly stable configuration $\eta^w$ for $(x, B(L))$. Then either
$I_x$ was never visited during the weak stabilisation, or~$I_x$ was
visited and $\eta^w$ is stable in $B(L)$, in which case Lemma
\ref{lemma:weaklyproperty} ensures that we finish the stabilisation procedure with one particle in $I_x$, or there exists $z \in I_x$ such that $\eta^w(z) = 2$,
all the other vertices of $I_x$ are empty and $\eta^w$ is stable in $V^\prime \setminus I_x$.  In the latter case, with some positive probability, we stabilise $B(L)$ in at most two steps in such a way that some particle remains in $I_x$.

\begin{proof}[Proof of Proposition \ref{propweak}]
Let $\eta$ be the initial particle configuration. We stabilise $\eta$
in $B(L)$ by first performing a weak stabilisation of $(x, B(L))$,
thus obtaining a weakly stable configuration $\eta^w$, and then, if
$\eta^w$ is not stable in $B(L)$, by proceeding with the
stabilisation.

Let $Q_{x, L}$ be the event that at least one particle has been in  $I_x$  during the weak stabilisation of $(x, B(L))$, i.e., 
there exists $\ell \in \{0, 1, \ldots, k \}$ such that $\eta_{\alpha^{(\ell)}} (z) \geq 1$ for some $z \in I_x$, where $\alpha =(x_1, \ldots, x_k)$ is a sequence which stabilises the initial configuration 
weakly in $(x,B(L))$.
We observe that 
\begin{equation}
\mathcal{P}^\mu(A_{x, L})\geq \mathcal{P}^\mu(A_{x, L} \cap Q_{x, L} \cap \{\eta^w \text{ is stable in }  B(L) \})+\mathcal{P}^\mu(A_{x,L} \cap \{\eta^w \text{ is not stable in } B(L) \}). \label{eq2}
\end{equation}
Consider the second term on the right-hand side of \eqref{eq2}. If
$\eta^w$ is not stable in $B(L)$, by definition of a weakly stable configuration  this implies that there exists a site $z \in I_x$ such that $\eta^w(z)=2$ and $\eta^w(y)=0$ for all $y \in I_x\setminus \{z\}$. If $z \in I_x \setminus \{x\}$, it then happens with probability  
$
\frac{d - 1}{d^3}
$
that such two particles both jump to $x$ in one step and, at the next step, they jump to two distinct neighbours of $x$, thus giving a stable configuration in $B(L)$ with two particles in $I_x$. Instead, if $z = x$, it happens with probability $\frac{d - 1}{d}$ that such two particles jump to two distinct neighbours of $x$, again giving a stable configuration in $B(L)$ with two particles in $I_x$. Hence,
\[
   \mathcal{P}^\mu(A_{x,L} \cap \{\eta^w \text{ is not stable in } B(L)\})
    \notag \geq \frac{d-1}{d^3}\mathcal{P}^\mu( \eta^w \text{ is not stable in } B(L)).
\]

We consider now the first term on the right-hand side of \eqref{eq2}. By Lemma \ref{lemma:weaklyproperty} we have that
\begin{equation*}
    Q_{x,L} \cap \{\eta^w \text{ is stable in } B(L) \} \subset A_{x,L}.
\end{equation*}
Hence,
\begin{align*}
    \mathcal{P}^\mu(A_{x, L } \cap Q_{x, L } \cap \{\eta^w \text{ is
stable in } B(L)\})&=\mathcal{P}^\mu(Q_{x, L } \cap \{\eta^w \text{ is stable in } B(L) \})\notag\\&\geq 1-\mathcal{P}^\mu(Q_{x, L}^c)-\mathcal{P}^\mu(\eta^w \text{ is not stable in } B(L)).
\end{align*}
By replacing the previous expression in \eqref{eq2}, we then obtain that
\begin{align*}
    \mathcal{P}^\mu(A_{x, L })&\geq 1-\mathcal{P}^\mu(Q_{x, L}^c)-\mathcal{P}^\mu(\eta^w \text{ is not stable in } B(L))+\frac{d-1}{d^3}\mathcal{P}^\mu\left(\eta^w \text{ is not stable in } B(L) \right)\\&\geq 1-\mathcal{P}^\mu(Q_{x, L}^c)-\left(1-\frac{d-1}{d^3}\right)\mathcal{P}^\mu\left(\eta^w \text{ is not stable in } B(L) \right)\\&\geq 1-\mathcal{P}^\mu(Q_{x, L}^c)-\left(1-\frac{d-1}{d^3}\right)=\frac{d-1}{d^3}-\mathcal{P}^\mu(Q_{x, L}^c).
\end{align*}
The proof is then concluded by observing that the event $Q_{x, L}^c$ implies that $m_{ B(L), \eta}(x)=0$.
\end{proof}

The next corollary is central in the proof of Theorem
\ref{thm:theorem1}. The corollary implies that, if we assume that the
stochastic sandpile model is active, then the probability that any
ball $I_x$ which is far enough from the boundary of $B(L)$ hosts at
least one particle after stabilisation is bounded from below.

\begin{corollary} \label{corrweak}
Let $G = (V,E)$ be an infinite vertex-transitive graph with degree
$d$, and suppose that the SSM in $G$ is active. Then, for any
$\varepsilon >0$, for $r$ large enough, for any $L \in \mathbb{N}$ and
$x \in B(L)$ with $B(x, r) \subset B(L)$,
\begin{equation}\label{eq:corollaryweak}
    \mathcal{P}^\mu(A_{x, L})\geq \frac{d-1}{d^3}- \varepsilon. 
\end{equation}
\end{corollary}

\begin{proof}
Since the system is active and the graph is vertex-transitive, Lemma
\ref{0-1law} implies that we can find $r=r(\varepsilon)>0$ such that
for any $x \in V$,
    \begin{equation*}
        \mathcal{P}^\mu(m_{ B(x,r),\eta}(x)=0)\leq \varepsilon.
    \end{equation*}
By monotonicity \cite[Lemma 3]{RS}, we have that for every~$x\in V$
such that~$B(x,r)\subset B(L)$,
\begin{align*}
\mathcal{P}^\mu(m_{ B(L),\eta}(x)=0) \leq    \mathcal{P}^\mu(m_{
B(x,r),\eta}(x)=0) \leq  \varepsilon.
\end{align*}
This concludes the proof, using Proposition~\ref{propweak}.
\end{proof}

\subsection{Concluding the proof of Theorem \ref{thm:theorem1}}
Let us briefly present the idea of the proof of Theorem \ref{thm:theorem1} for SSM in $\mathbb{Z}^d$. Corollary \ref{corrweak} implies that, if the system is active, then the density of particles on large balls {after stabilisation} is at least a positive constant $c$, which depends only on $d$ and does not depend on $\mu$
(this argument fails on non-amenable graphs, since
(\ref{eq:corollaryweak}) does not hold at vertices which are ``too
close'e to the boundary of the ball, and such graphs have too many vertices on the boundary). 
Since the density of particles cannot increase during stabilisation,
this implies that~$\mu\geq c$, and this being true for every~$\mu$
such that SSM is active, we deduce that~$\mu_c\geq c$.

The proof on general vertex-transitive graphs (which are not
necessarily amenable) requires an additional tool, namely the method
of ghost particles. We will still look for a lower bound on the
density of particles after stabilisation, using Corollary
\ref{corrweak}, to obtain \textit{in fine} a lower bound on~$\mu$ such
that SSM stays active.

We let $X(t)$ denote a simple random walk in $G$, and $P_x$ denote its law when $X(0) = x \in V$. We let $E_x$ denote the corresponding expectation. Given a set $Z \subset V$, we define $\tau_Z := \inf\{t \geq 0 : X(t) \in Z\}$.  For any $x,y \in V$, we define the Green’s function,
\begin{equation*}
     G_Z(x,y):=E_x\Bigg[ \sum_{t=0}^{\tau_{Z^c}-1}\mathbbm{1}_{\{X(t)=y\}}\bigg].
 \end{equation*}
We state an auxiliary result from \cite{CST} which will be used in the proof.

\begin{lemma}[Lemma 4.1 in  \cite{CST}]\label{lemma4.1CST}
   Given an infinite vertex-transitive graph $G=(V,E)$, for any $r \in
\N$ there exists $L_0=L_0(r)$ such that, for any $L\geq L_0$,
     \begin{equation*}
         \sum_{x \in B(L)}G_{B(L)}(x,o)<10 \sum_{\substack{x \in B(L):
\\B(x,r)\subset B(L)}}G_{B(L)}(x,o).
     \end{equation*}
\end{lemma}

\begin{proof}[Proof of Theorem \ref{thm:theorem1}]
Let~$\mu>0$ such that the SSM is active.
We will show that this implies a lower bound on~$\mu$.

Let~$L\in\N$.
We perform a stabilisation of the set $B(L)$. Once the stabilisation
of $B(L)$ is complete, we let each stable vertex in $B(L)$ with
precisely one particle perform a \textit{ghost} simple independent
random walk starting from that vertex and being killed upon leaving
$B(L)$. Let $\bar m_L(o)$ be the number of particles or ghosts jumping
from the root vertex $o\in V$, let $w_L(o)$ be the number of ghosts jumping
from $o$ and recall that $m_{B(L), \eta}(o)$ denotes the odometer, namely
the number of particles jumping from $o$ during the stabilisation of
$B(L)$. We then have that
\begin{equation}\label{eq:ghostrelation}
    \bar m_L(o)=m_{B(L), \eta}(o)+w_L(o)
\geq w_L(o).
\end{equation}
Denote by $\bar \cP^\mu$ the probability measure in the enlarged
probability space of particles and ghosts, by $\bar E^\mu$ the
expectation with respect to $\bar \cP^\mu$, and recall that
$\eta_{B(L),\infty}$ denotes the particle configuration after the stabilisation of $B(L)$.

It follows from the definition of ghosts that
\begin{align}
\label{eq:visits1}
    \bar E^\mu[\bar m_L(o)] & = \mu \sum_{x \in B(L)}G_{B(L)}(x,o), \\
    \label{eq:visits2}
    \bar E^\mu [w_L(o)]
    &  =  \sum_{x \in B(L)}\bar \cP^\mu\left( \eta_{B(L),\infty}(x) =
1  \right) G_{B(L)}(x,o).
\end{align}

The first identity holds since, in order to sample $\bar m_L(o)$, we
need to start an independent simple random walk trajectory killed upon
leaving $B(L)$ for each initial particle in $B(L)$ and count the
number of visits at $o$. The second identity holds since a ghost starts from a vertex $x$ if and only if a particle is at $x$ after the stabilisation of $B(L)$.

Let~$\varepsilon>0$.
We define the set $K_L := \{  x \in B(L) \, : \, B(x, r) \subset B(L)
\}$,
choosing the constant $r$ large enough so that, by Corollary \ref{corrweak},
 $\mathcal{P}^\mu( A_{x, L}  ) \geq  \frac{d-1}{d^3}  - \varepsilon$
for every~$L\in\N$ and $x \in K_L$.
 Note that we can identify a set of vertices $U \subset K_L$ such that
for each $x, y \in U$, if~$x\neq y$ then~$I_x\cap I_y=\varnothing$,
and 
\begin{equation}\label{eq:Uproperty}
\forall \,x \in K_L\quad
\exists \,y \in U\qquad
d_G(x,y) \leq 2
\end{equation}
(indeed, if there was $x \in K_L$ such that $d_G(x,y) > 2$ for each $y
\in U$, we could simply take $x$ in $U$). Moreover, we observe that
from the Markov property, for $x, y \in V$ such that $d_G(x,y) = k$ we have that
\begin{equation}\label{eq:confronto}
   \frac{1}{d^k}  \,  G_{B(L)}(y, o)  \,   \,  \leq \, G_{B(L)}(x, o)
\,  \leq  \,  \,  d^{k }\, G_{B(L)}(y, o).
\end{equation}
From (\ref{eq:Uproperty}) and (\ref{eq:confronto}) we deduce the crude bound
\begin{align*}
 \sum_{x\in K_L }   \, G_{B(L)}(x,o)  
&\leq 
 \sum_{y \in U }   \sum\limits_{ \substack{x \in K_L  : \\ d(x,y) \leq
2}   } \, G_{B(L)}(x,o)  \\
&\leq   \sum_{y \in U }   \sum\limits_{ \substack{x \in K_L  : \\ d(x,y) \leq 2}   } \, 
 d^2  G_{B(L)}(y,o)  \leq 
 d^2(d^2+1)  \sum_{y \in U }  
 G_{B(L)}(y,o).
\numberthis\label{eq:comparisongreessums}
\end{align*}
Finally, note that, by the Markov property, for any $y \in U$ we have that 
\begin{equation}\label{eq:inequalityweak}
\sum\limits_{x \in I_y}  
\bar \cP^\mu\left( \eta_\infty(x) = 1  \right) G_{B(L)}(x,o)
\, \geq  \, 
\frac{1}{d} 
\bar \cP^\mu\left(  A_{y,  L}  \right) G_{B(L)}(y,o),
\end{equation}
where we used that the expected number of visits at $o$ from ghosts
starting in $I_y$ is at least the expected number of visits from one
of the ghosts starting from $I_y$ assuming that, if such ghost starts
from $x  \in I_y$ with $x \neq y$, then in the first step it jumps to
$y$.

Thus,  we have that, for any large enough~$L$,                          
                        
\begin{align*}
 \bar E^\mu [\bar m_L(o)]
\geq
 \bar E^\mu [w_L(o)]
&\geq  \sum_{y \in U}\sum_{x\in I_y}\bar \cP^\mu\left( \eta_{B(L),\infty}(x) =1  \right) G_{B(L)}(x,o)\\
    &  \geq   \frac{1}{d}  \sum_{y \in U}\bar \cP^\mu\left( A_{y, L}  \right) \,  G_{B(L)}(y,o) \\
   & \geq  \frac{1}{d}  \left( \frac{d-1}{d^3}  - \varepsilon \right)  \sum_{y \in U}     \, G_{B(L)}(y,o) \\
      & \geq  \frac{1}{d^3(d^2+1)}  \left( \frac{d-1}{d^3}  -
\varepsilon \right) \sum_{x \in K_L}   \, G_{B(L)}(x,o) \\
          & \geq \frac{1}{10d^3(d^2+1)}  \left(\frac{d-1}{d^3}  -
\varepsilon \right)   \sum_{x \in B(L)}  \, G_{B(L)}(x,o)
\end{align*}
where the first inequality follows from \eqref{eq:ghostrelation}, 
 for the second inequality we use \eqref{eq:visits2}, for the
third inequality we use \eqref{eq:inequalityweak}, for the fourth
inequality we use that~$r$ was given by Corollary \ref{corrweak}, for
the fifth inequality we use
(\ref{eq:comparisongreessums})
and for the last inequality we use Lemma \ref{lemma4.1CST}.
Then, using \eqref{eq:visits1} we deduce that
\[
 \bar E^\mu [\bar m_L(o)]
\geq
\frac{1}{10d^3(d^2+1)}  \left(\frac{d-1}{d^3}  - \varepsilon \right)
\times\frac 1\mu \bar E^\mu [\bar m_L(o)],
\]
which, given that~$\bar E^\mu [\bar m_L(o)]\neq 0$, implies that
\[
\mu
\geq
\frac{1}{10d^3(d^2+1)}  \left(\frac{d-1}{d^3}  - \varepsilon \right).
\]
This being true for every~$\varepsilon>0$ and for every~$\mu>0$ such
that the SSM is active, we obtain that
\begin{equation}
\label{lowerBoundMuc}
\mu_c
\geq
\frac{d-1}{10d^6(d^2+1)},
\end{equation}
which concludes the proof that~$\mu_c>0$.
\end{proof}

\section{Stochastic comparison}
In this section we introduce the 
activated random walks with instantaneous deactivation (ARWD)
and we present one of our main results, namely 
the stochastic comparison between SSM and ARWD.

\subsection{Definition of the ARWD process}
\label{sec:def-ARWD}

Our strategy to obtain the upper bound is based on a comparison with
another model, which we call activated random walks with instantaneous deactivation (ARWD).
As for activated random walks, a configuration of the ARWD model on~$G$ can be represented by a vector~$\eta : V \to \N_{\s}$,
with~$\N_{\s} = \N_0 \cup \{\s\}$, where~$\eta(x) = k \in \N_0$ means that there are~$k$ active particles on the site~$x$, while~$\eta(x) = \s$ means that
there is one sleeping particle on the site~$x$.
A site~$x \in V$ is said to be stable in the configuration~$\eta$ if~$\eta(x) \in \{0, \s\}$, otherwise it is called unstable. We say that a configuration~$\eta$ is stable in a subset~$U \subset V$ if~$\eta(x) \in\{0, s\}$ for all~$x\in U$.

For~$\eta\in\N_\s^V$, we denote by~$|\eta|=\sum_{x\in V}|\eta(x)|$ the
total number of particles in~$\eta$ (using the convention
that~$|\s|=1$).
We equip the set~$\N_\s$ with the
order~$0<\s<1<2<\cdots$ and, for~$\eta,\,\eta'\in\N_s^V$, we
write~$\eta\leq\eta'$ if~$\eta(x)\leq\eta'(x)$ for every~$x\in V$.

Given~$\eta \in \N_{\s}^V$,~$U\subset V$, we denote by~$\eta^U$ and~$\eta_U$ the configurations respectively obtained by waking up the particles in~$U$ or making them fall asleep, that is to say, for every~$x\in V$,
\begin{equation}
\label{notation-eta-U}
    \eta^U(x)=\begin{cases}
         1
         &\text{if~$x \in U$ and~$\eta(x)=\s$}\,,\\
     \eta(x)
     &\text{otherwise}
    \end{cases}
    \qquad\text{and}\qquad
    \eta_U(x)=\begin{cases}
         \s
         &\text{if~$x \in U$ and~$\eta(x)=1$}\,,\\
     \eta(x)
     &\text{otherwise.}
    \end{cases}
\end{equation}
We now present the ARWD process, first informally, before giving the
precise definition below.
This process is associated with a certain fixed subset~$A\subset V$,
and is a variant of the loop representation of ARW introduced in~\cite{AFG}.
At each time step of this process, an unstable site~$x\in V$ is
chosen, with a
certain deterministic rule which depends on the past history of the
process.
This rule is given by what we call a toppling strategy:

\begin{definition} \label{def-topp-strategy}
   A toppling strategy is an application~$f:\cup_{t\geq 0}(\N_s^V)^{t+1}\to V \cup \{\varnothing\}$ such that for every~$t\geq 0$ and every configurations~$\eta_0,\,\dots,\,\eta_t$, if~$f( \eta_0,\,\dots,\,\eta_t) \neq \varnothing$,  the site~$f( \eta_0,\,\dots,\,\eta_t)$ is unstable in~$\eta_t$.
\end{definition}

Note that we will indeed use a strategy that depends not only on the current configuration but on the whole history of the process (see Section~\ref{section-def-f-g} for the precise definition of the toppling strategy used during the so-called ping-pong rally).

Then, if the chosen unstable site~$x$ belongs to~$A$ and contains one
single particle, this particle falls asleep with
probability~$\lambda/(1+\lambda)$.
Otherwise, if it does not fall asleep, one particle at~$x$ makes a
random walk on~$V$ until it finds an empty site of~$A$ and is left
there, sleeping.
This is a key difference with respect to the process considered
in~\cite{AFG}, where the particle was left active after its walk.
Besides, along this walk made by the particle, the sleeping particles
met are waken up.

We now turn to the formal definition of the ARWD process.

\begin{definition}
\label{def-ARWD}
Let~$G=(V,\,E)$ be a finite connected graph, and let~$A\subset V$ be
fixed.
We consider the state space~$\cS_A=\{\eta\in\N_\s^V\,:\,|\eta|=|A|\}$.
For a given parameter~$\lambda\geq 0$, a toppling strategy~$f$ 
and an initial configuration~$\eta_0\in\cS_A$, we call~$\mathrm{ARWD}(\lambda,\,A,\,f,\,\eta_0)$ the process~$(\eta_t)_{t\geq 0}$ with state space~$\cS_A$ and deterministic initial configuration~$\eta_0$, which is such that for every~$t\geq 0$ and every~$\eta'\in\cS_A$, we have
\[
\bbP\big(\eta_{t+1}=\eta'\ \big|\ \eta_0,\,\dots,\,\eta_t \big)
=Q_A\big(\eta_t,\,f(\eta_0,\,\dots,\,\eta_t),\,\eta'\big),
\]
where~$Q_A:\cS_A\times\big(V\cup\{\varnothing\}\big)\times\cS_A\to[0,1]$ is defined as follows:
\begin{itemize}
\item For any $\eta,\,\eta' \in \cS_A$, $Q_A(\eta, \,\varnothing,\,\eta')=\mathbbm{1}_{\eta'=\eta}$.

    \item Let~$\eta\in\cS_A$ be an unstable configuration and let~$x\in V$ be an unstable site of~$\eta$.
    Let us define
    \[
    p\ =\ 
    \begin{cases}
    \frac\lambda{1+\lambda}
    &\text{ if }x\in A\text{ and }\eta(x)=1\,,\\
    0
    &\text{ otherwise.}
    \end{cases}
    \]
    Then, we have~$Q_A(\eta,\,x,\,\eta')=p\,Q_A^{\text{sleep}}(\eta,\,x,\,\eta')+(1-p)Q_A^{\text{walk}}(\eta,\,x,\,\eta')$, with~$Q_A^{\text{sleep}}$ and~$Q_A^{\text{walk}}$ defined as follows.
    
    On the one hand, recalling the notation~$\eta_{\{x\}}$ defined in~\eqref{notation-eta-U}, we have
\[
Q_A^{\text{sleep}}\big(\eta,\,x,\,\eta_{\{x\}}\big)=1
\,,
\]
that is to say, if the site~$x$ belongs to~$A$ and contains one single particle, this particle falls asleep with probability~$p=\lambda/(1+\lambda)$.

On the other hand, with probability~$1-p$ the particle at~$x$ walks
until it finds an empty site of~$A$ and is left there sleeping, having
woken up the sleeping particles met along its path. Thus, writing~$S=\{z\in V\,:\,\eta(z)=\s\}$, for every~$y\in A$
and~$R\subset S$, we have
\[
Q_A^{\text{walk}}\big(\eta,\,x,\,\eta^{R}-\mathbbm{1}_x+\s\mathbbm{1}_y)
=
P_x\big(X(\tau_B^+)=y\text{ and }\{X(t),\,0\leq t< \tau_B^+\}\cap S=R\big)
\,,
\]
where, under~$P_x$,~$(X(t))_{t\geq 0}$ is a simple random walk on~$G$
started at~$x$ and~$\tau_B^+$ is its first hitting time of the set~$B$
of the empty sites of~$A$, that is to say,
\[
B
=
\big\{z\in A\ :\ (\eta-\mathbbm{1}_x)(z)=0\big\}
\]
and~$\tau_B^+=\inf\{t\geq 1\,:\,X(t)\in B\}$, 
which is almost surely finite because we work with configurations with exactly~$|A|$ particles (and the graph is connected).
\end{itemize}
\end{definition}

Note that this model is not abelian: to our knowledge, there is no
monotonicity property of the number of steps before stabilisation with respect to the initial configuration.

\begin{definition}
\label{def-stab-time-ARWD}
For given parameters~$A,\,\lambda,\,f,\,\eta_0$, we denote by~$T(A,\,f,\,\eta_0)$ the first time that~$f$ is the empty set, that is to say,~$T(A,\,f,\,\eta_0)=\inf\{t\geq 0\,:\, f( \eta_0, \dots,\eta_t)=\varnothing\}$ where~$(\eta_t)_{t\geq 0}$ is an instance of the~$\mathrm{ARWD}(\lambda,\,A,\,f,\,\eta_0)$ process.
The dependence on~$\lambda$ is omitted to shorten the notation.
\end{definition}

When~$f(\eta_0,\,\dots,\,\eta_t)=\varnothing$ but the
configuration~$\eta_t$ is not stable, the intuitive idea is that the
strategy~$f$ ignores these unstables sites, that is to say, it considers
them as stable.
Indeed, in Section~\ref{section-def-f-g} we will introduce the notion of
sleep mask, which allows to ignore some reactivation events, and the
strategy~$f$ will be constructed by choosing particles which are really
active, i.e., whose reactivation was not ignored by the sleep mask.

\subsection{Connection between SSM and ARWD}

Our interest for ARWD lies in the comparison established by Proposition~\ref{propdom} below, between~$A$-stabilisation for the SSM and the stabilisation time for ARWD.
For a technical reason we need to restrict ourselves to a subset~$B\subset A$, and only consider the time for ARWD to stabilise this subset.

\begin{proposition}\label{propdom}
Let~$G=(V,E)$ be a finite connected graph with maximal
degree~$\Delta\geq 2$ and let~$B\subset A\subset V$ such that each
site of~$B$ has degree at least~$2$ (in~$G$) and~$B$ is totally disconnected, i.e., it does not contain any pair of neighbouring sites.
Let~$\eta\in\N_0^V$ be a SSM configuration such that~$|\eta|=|A|$ and let~$\tilde\eta\in\N_\s^V$ be the ARW configuration with one particle on each site of~$A$ defined by~$\tilde\eta=\mathbbm{1}_U+\s\mathbbm{1}_{A\setminus U}$, where~$U=\{x\in A\,:\,\eta(x)\geq 2\}$.

Let~$\lambda=\Delta^3$, and let~$f$ be whatever toppling strategy.
Then we have that \smash{$\|m^A_{\eta}\|$} (the number of
half-topplings to~$A$-stabilise the SSM configuration~$\eta$, 
see Definition~\ref{def-A-stab})
stochastically dominates~$T(B,\,f,\,\tilde\eta\mathbbm{1}_B)$, the
stabilisation time of the ARWD process on~$B$ starting
from~$\tilde\eta\mathbbm{1}_B$ (see Definition~\ref{def-stab-time-ARWD}).
\end{proposition}

The proof of this proposition is the object of Section~\ref{sectionARWD}.
It relies on a coupling between the ARWD process and
the~$A$-stabilisation for the SSM such that the topplings made in ARWD are~$A$-legal for the SSM.
To construct this coupling we show that, conditioned on the evolution of the SSM process, each time that we start a walk at a site~$x\in B$ which contains one single active particle, then with probability at least~$1/(1+\lambda)$ the odometer at this site is odd, which implies that this walk is~$A$-legal for the SSM.

\subsection{Bound on the stabilisation time of ARWD}
\label{sec-prop-ARWD-stab-time}

Once the problem is reduced to the study of the stabilisation time of ARWD, we show the following result:

\begin{proposition}
    \label{prop-ARWD-stab-time}
Let~$d\geq 2$.
    For every~$\lambda\geq 0,\,\mu\in(0,1)$ and~$\rho\in(0,1)$, there
exists~$\kappa>0$ such that, for every~$n\geq 1$, for
every~$A\subset\Z_n^d$ such that~$|A|\geq \mu n^d$, there exists a probability space $(\Sigma, \cF, \nu)$ and a collection of toppling strategies $(f^\mathbf{j})_{\mathbf{j}\in \Sigma}$
such that, if $J$ is a random variable with distribution $\nu$ and we use the strategy~$f^J$ associated with the value of~$J$, then, starting
with a random initial configuration~$\eta_0$ with one particle on each
site of~$A$, each particle being active with probability~$\rho$ and
sleeping with probability~$1-\rho$, independently of other particles (and independently of~$J$), we have that~$1+T(A,\,f^J,\,\eta_0)$ dominates a geometric variable with parameter~$\exp(-\kappa n^d)$.
\end{proposition}

We prove this result in Section~\ref{section-ARWD-stab-time}, adapting the strategy developed in the context of ARW by \cite{AFG}.
This strategy consists in a multiscale argument, constructing a
hierarchical structure on the set~$A$ with ``well connected'' components at the first level of the hierarchy and pairs of clusters which merge at each level of the hierarchy.
We first show that the clusters at the first level have a stabilisation time which is exponentially large in their size, before showing that this property is inherited at further levels of the hierarchy, loosing only a little in the constant in the exponential.
The choice of a random variable~$J$ and of a strategy which depends on
the value of this random variable~$J$ can be understood as considering a
strategy which is not deterministic but random.
As in~\cite{AFG}, this randomness will consist of a choice of a random
color associated to each random walk, which will affect which sites can
be reactivated by this walk (see Section~\ref{section-ARWD-stab-time}
for more details).

Then, in Section~\ref{section-exp-time}, we explain how Theorem~\ref{thm-exp-time} follows from Propositions~\ref{propdom} and~\ref{prop-ARWD-stab-time}.

\section{Connection SSM-ARWD: proof of Proposition~\ref{propdom}} \label{sectionARWD}

This section is devoted to the proof of Proposition~\ref{propdom} which relates the number of toppling for the~$A$-stabilisation of the SSM to the stabilisation time of the ARWD process.

\subsection{Strategy: sleep if the odometer is even}
\label{section-strategy-propdom}

To prove this proposition, we construct a  coupling between the ARWD process and the SSM,   up to a certain stopping time~$T$.
More precisely, we construct the ARWD process using a field of instructions~$\tau$,
which are used to determine the random walk trajectories of the particles, 
and Bernoulli random variables which are used to decide whether the 
particle falls asleep or performs a random walk.

We want to show that, following these dynamics, we obtain at each time step a toppling sequence that is~$A$-legal for the SSM. 
This requires showing that,  each time that we choose a site~$x\in A$ with only one particle, and that the particle at~$x$ makes a random walk, the odometer at this site is odd,  a necessary  condition for the toppling at~$x$ being legal.

The idea is to show that, in this case, knowing the configurations up to this time (but not the field~$\tau$ nor the odometers), the odometer at~$x$ is odd with probability at least~$1/(1+\lambda)$, with $\lambda = \Delta^3$. 
This allows us to couple the Bernoulli variable that determines whether the particle sleeps.  The coupling is such that, when
 the odometer is even, the variable ensures sleep,
 and when the odometer is odd,  there is still a probability of sleep.
 The overall probability to sleep is then~$\lambda/(1+\lambda)$,
 whatever the odometer is.
This bound on the conditional probability that the odometer is odd is the content of Lemma~\ref{lemma-add-instructions} below.

The proof of Lemma~\ref{lemma-add-instructions}  uses the fact that,  if~$x\in A$ is active and~$\eta_t(x)=1\neq \eta_0(x)$, then it entails that the site~$x$ was visited since the last time that we chose~$x$ (because upon its last arrival on~$x$ the particle fell asleep, so it must have been reactivated by the visit of another particle), or since the beginning if~$x$ was not chosen yet.
Then, once we know that the site was visited, we can bound from below the probability that the number of visits was odd, by showing that one visit can be added with a reasonable probabilistic cost.
This idea of studying the parity of the number of visits during a random walk excursion was already used in \cite{HHRR, PR}.

\subsection{Coupled ARWD dynamics}

\begin{definition}
\label{def-coupled-ARWD}
    Let~$G=(V,E)$ be a finite connected graph with maximal degree~$\Delta$, let~$B\subset A\subset V$, let~$f$ be a toppling strategy,
    let~$\eta\in\N_0^V$ be a fixed SSM configuration with~$|\eta|=|A|$ and let~$U=\{x\in A\,:\,\eta(x)\geq 2\}$.

    Then the coupled ARWD process~$(\eta_t,\,\alpha_t)_{t\geq 0}$ associated with these parameters consists of a sequence of configurations~$(\eta_t)_{t\geq 0}$ in~$\cS_A=\{\eta\in\N_s^V\,:\,|\eta|=|A|\}$ and a sequence of
half-toppling sequences~$(\alpha_t)_{t\geq 0}$ which are defined as follows.
We start taking~$\eta_0=\eta_A$, that is to say, the configuration obtained from~$\eta$ by declaring sleeping the particles which stand alone on a site of~$A$.
 Using an arbitrary order on~$V$, we define another toppling strategy~$f'$ by writing, for every~$t\geq 0$, every~$\eta_1,\,\dots,\,\eta_t\in\N_\s^V$,

   \[
     f'( \eta_0,\,\dots,\,\eta_t)
=
\begin{cases}
     \min\big\{x\in V\setminus A\,:\,\eta_t(x)\geq 1\big\}
     &\text{if $\{x\in V\setminus A\,:\,\eta_t(x)\geq 1\big\}\neq
\varnothing$;}\\ 

     \min\big\{x\in A\,:\,\eta_t(x)\geq 2\big\}
     &\text{else if $\big\{x\in A\,:\,\eta_t(x)\geq 2\big\}\neq
\varnothing$;}\\

     f(\eta_S \mathbbm{1}_{B},\,\dots,\,\eta_t
\mathbbm{1}_{B})
     &\text{else if $\eta_t\leq\mathbbm{1}_A$,}\\
     
     \varnothing
     &\text{otherwise,}
\end{cases}
     \]
  
  where 
    \[ \ S=\min\big\{s\leq t\,:\,\eta_s\leq\mathbbm{1}_A\big\}. \]

In other words, as long as there is not one particle on each site of~$A$, sites of $A$ that contain only one particle can't be chosen by $f'$, while once this is the case we follow the strategy~$f$.

    Let~$\lambda=\Delta^3$.
    Let~$\tau=(\tau^{x,j})_{x\in V,\,j\geq 0}$ be a random field of
toppling instructions distributed according to~$\mathcal{P}$, that is
to say, each instruction at~$x$ is a jump to a uniformly chosen
neighbour of~$x$ (as defined in Section~\ref{sectionDiaconis}).
We also consider a sequences of i.i.d.\ random variables~$(Z_t)_{t\geq 0}$, independent of $\tau$, where~$Z_0$ has uniform distribution on~$[0,1]$.

Now, we define recursively~$(\eta_t)_{t\geq 0}$ and~$(\alpha_t)_{t\geq 0}$.
For every~$t\geq 0$ we will denote by~$m_t$ the odometer
of~$\alpha_t$, i.e.,~$m_t=m_{\alpha_t}$.

Recall that~$\eta_0=\eta_A$ and we let~$\alpha_0$ be the empty toppling sequence.
Then, let~$t\geq 0$ be such that~$\eta_0,\,\dots,\,\eta_t$ and~$\alpha_t$ are constructed.

If~$\eta_t\leq \mathbbm{1}_{A\setminus B}+\s\mathbbm{1}_B$, that is to say, if each site of~$A$ already contains one particle with all the particles on~$B$ being asleep, then we do nothing and simply let~$\eta_{t+1}=\eta_t$ and~$\alpha_{t+1}=\alpha_t$.

Assume now that~$\eta_t\not\leq \mathbbm{1}_{A\setminus B}+\s\mathbbm{1}_B$. If $f'( \eta_0,\,\dots,\,\eta_t)=\varnothing$, we set~$\eta_{t+1}=\eta_t$ and~$\alpha_{t+1}=\alpha_t$. Otherwise, if~$X_t=f'(  \eta_0,\,\dots,\,\eta_t)\neq \varnothing$, we write
\[
\pi_t=\mathbb{P}\big(m_t(X_t)\text{ odd}\ \big|\ \eta_1,\,\dots,\,\eta_t\big)\,,
\]
and we consider the variable~$I_t$ which can take two
values,~$\mathrm{sleep}$ and~$\mathrm{walk}$, defined as
\[
I_t
=
\begin{cases}
    \text{sleep}
    &\text{if }
    X_t\in B\,,\ 
    \eta_t(X_t)=1
    \text{ and }m_t(X_t)\text{ is even,}
    \\
    
    \text{sleep}
    &\text{if }
    X_t\in B\,,\ 
    \eta_t(X_t)=1\,,\ 
    m_t(X_t)\text{ is odd and }Z_t>\frac{1}{(1+\lambda)\pi_t}
    \,,
    \\
    
    \text{walk}
    &\text{otherwise.}
\end{cases}
\]

On the one hand, if~$I_t=\text{sleep}$, then the particle at~$X_t$ falls asleep: we let~$\eta_{t+1}=(\eta_t)_{\{X_t\}}$ and~$\alpha_{t+1}=\alpha_t$.

On the other hand, if~$I_t=\text{walk}$, then we let the particle at~$X_t$ walk with successive topplings drawn from~$\tau$ and applied to~$(\eta_t,\,m_t)$, until it reaches a site of~$A$ which is empty (in the configuration~$\eta_t-\mathbbm{1}_{X_t}$).
Let~$Y_t$ be the first such site reached by the particle, let~$\beta_t$ be the sequence of sites toppled along this walk, and let~$R_t$ be the set of sites visited by this walk, namely~$R_t=\{x\in V\,:\,m_{\beta_t}(x)\geq 1\}$. If $\eta_t(X_t)= 1$ and $X_t \in V\backslash A$ or $\eta_t(X_t)\geq 2$, we let~$\eta_{t+1}=\eta_t-\mathbbm{1}_{X_t}+\s\mathbbm{1}_{Y_t}$, that is to say, the particle falls asleep at~$Y_t$ at the end of its walk and sleeping particles met can not be awaken.
If~$\eta_t\leq \mathbbm{1}_A$,  let~$\eta_{t+1}=(\eta_t)^{R_t}-\mathbbm{1}_{X_t}+\s\mathbbm{1}_{Y_t}$, that is to say, sleeping particles met by the walk are waken up.
We define~$\alpha_{t+1}$ as the concatenation of~$\alpha_t$ and~$\beta_t$.

The process~$(\eta_t,\,\alpha_t)_{t\geq 0}$ defined above is called the coupled ARWD process associated with~$G,\,A,\,B,\,f, \eta_0$.
\end{definition}

Note that, with this toppling strategy~$f'$ and with this definition of $(\eta_t, \alpha_t)_{t\geq 0}$,
the evolution of the process consists in the following three successive stages, which correspond to the three first cases in the definition of~$f'$:

\begin{itemize}
\item
During a first stage, we move the particles from~$V\setminus A$ to vacant sites of~$A$ until all particles are in~$A$, and particles cannot wake up during this stage.

\item
Then comes a second stage during which we move particles from sites of~$A$ with two or more particles to empty sites of~$A$, until each site of~$A$ contains exactly one particle.
Reactivation events are also ignored during this stage.

At the end of this second stage, the particles in~$U$ are all active, because a site in~$U$ has no chance to fall asleep during these first two phases, while the particles in~$A\setminus U$ are all sleeping, that is to say, we have reached the configuration~$\tilde\eta=\mathbbm{1}_U+\s\mathbbm{1}_{A\setminus U}$.

\item
Then, during the third and last stage, active particles on sites of~$B$ perform loops and come back to their position, sleeping, until all particles on~$B$ are sleeping.
Then, once all particles on~$B$ are sleeping, the process does not evolve anymore.
\end{itemize}

\subsection{The coupling lemma}

We now show that the coupled process defined above enjoys the following properties:

\begin{lemma}
\label{properties-coupled-ARWD}
    Let~$G=(V,E)$ be a finite connected graph with maximal
degree~$\Delta\geq 2$, let~$B\subset A\subset V$ with~$B$ totally
disconnected and~$d_x\geq 2$ for every~$x\in B$.
Let~$\eta\in\N_0^V$ be a SSM configuration with~$|\eta|=|A|$.
    We also define the ARW configuration~$\tilde\eta=\mathbbm{1}_U+\s\mathbbm{1}_{A\setminus U}$, where~$U=\{x\in A\,:\,\eta(x)\geq 2\}$.
    Let~$\lambda=\Delta^3$ and let~$f$ be a toppling strategy.
    Let~$(\eta_t,\,\alpha_t)_{t\geq 0}$ be the coupled ARWD process defined above, with~$f'$ the modified toppling strategy 
    and~$\tau$ the associated array of instructions.
    Then, defining the two stopping times~$S=\inf\{t\geq
0\,:\,\eta_t\leq\mathbbm{1}_A\}$ and~$T=\inf\{t\geq S\,:\, f'(\eta_0,\,\dots,\,\eta_t)=\varnothing\}$, the following hold:
\begin{enumerate}[(i)]
\item \label{i-finite}
the stopping times~$S$ and~$T$ are almost surely finite;
    \item \label{i-legal}
    for every~$t\geq 0$, the half-toppling sequence~$\alpha_t$ is
$A$-legal for~$\eta$ (with respect to the field of instructions~$\tau$);
    

    \item \label{i-after-S}
    we have~$\eta_S=\tilde\eta$ (if~$S<\infty$) and~$(\eta_{S+t}\mathbbm{1}_B)_{t\geq 0}$ is distributed as the~$\mathrm{ARWD}(\lambda,\,B,\,f,\,\tilde\eta\mathbbm{1}_B)$ process;

    \item \label{i-nb-steps}
    for every~$t\geq 0$ we have~$\|m_t\|\geq \mathbbm{1}_{\{t\geq S\}}\big((t\wedge T)-S+N_t\big)$, where~$N_t=|\{x\in A\,:\,\eta_t(x)=1\}|$.
\end{enumerate}
\end{lemma}

Before turning to the proof of Lemma~\ref{properties-coupled-ARWD}, we first explain how Proposition~\ref{propdom} follows from it.

\begin{proof}[Proof of Proposition~\ref{propdom}]
  Recall that the notation and hypotheses of Proposition~\ref{propdom} are the same as that of 
    Lemma~\ref{properties-coupled-ARWD}.
First, note that point~\itref{i-finite} ensures that the half-toppling
sequence~$\alpha_T$ is almost
surely well defined.
    Point~\itref{i-legal} of Lemma~\ref{properties-coupled-ARWD} tells us that this half-toppling sequence is~$A$-legal for~$\eta$.
    By definition of \smash{$m_{\eta}^A$}, this implies that \smash{$m_{\eta}^A\geq m_{\alpha_T}=m_T$}, whence \smash{$\|m_{\eta}^A\|\geq\|m_T\|$}.
    Yet, by virtue of~\itref{i-nb-steps}, we have~$\|m_T\|\geq T-S$.
    Then, point~\itref{i-after-S} entails that~$T-S$ is distributed as~$T(B,\,f,\,\tilde\eta\mathbbm{1}_B)$.
    Combining all this, we obtain that~$\|m_{\eta}^A\|$ stochastically dominates~$T(B,\,f,\,\tilde\eta\mathbbm{1}_B)$, as claimed.
\end{proof}

\subsection{Proof of the coupling lemma}

As explained in Section~\ref{section-strategy-propdom}, one key point
to obtain Lemma~\ref{properties-coupled-ARWD}, and in particular~\itref{i-after-S}, is to establish the following:

\begin{lemma}
\label{lemma-add-instructions}
    With the notation and hypotheses of Lemma~\ref{properties-coupled-ARWD}, for every~$t\geq 0$ and~$x\in B$ we have the implication
\[
\big\{\eta_t(x)=1\big\}
\subset
\bigg\{
\bbP\big(m_t(x)\text{ odd}\ \big|\ \eta_1,\,\dots,\,\eta_t\big)
\geq\frac 1{1+\lambda}
\bigg\}
\,.
\]
\end{lemma}

Let us postpone the proof of Lemma~\ref{lemma-add-instructions} to Section~\ref{section-proof-add-instructions} and first prove Lemma~\ref{properties-coupled-ARWD} admitting this Lemma.

\begin{proof}[Proof of Lemma~\ref{properties-coupled-ARWD}]

First, point~\itref{i-finite} follows from our assumptions
that~$|\eta|=|A|$ and that~$G$ is connected, along with the fact that~$\lambda>0$.

\medskip

Point~\itref{i-legal} follows from the fact that, when we topple a site of~$B$ which contains only one particle, we never topple it if the odometer at this point is even (because in this case~$I_t=\text{sleep}$ and we do not topple, marking the particle as sleeping) and by definition of~$f'$ we never topple a site of~$A \setminus B$ which contains only one particle.
Also, recall that topplings at sites in~$V\setminus A$ with at least one particle or at sites of~$A$ with at least two particles are always~$A$-legal.
Yet, along a random walk, apart from the very first toppling, we never topple a site of~$A$ which contains only one particle, because we stop as soon as the particle reaches a site of~$A$ where it is alone.
Hence, point~\itref{i-legal} holds.

\medskip

\medskip

Let us now check~\itref{i-after-S}.
First note that, since~$|\eta|=|A|$, the definition of~$S$ implies
that at time~$S$ there is exactly one particle on each site of~$A$.
Then, note that by definition of~$U$, for each~$x\in A\setminus U$
we have~$\eta_0(x)\in\{0,\,\s\}$.  For each of these sites which starts empty, the first particle which visits the site
 settles there, sleeping. Moreover, recall that by definition of the coupled ARWD process, sleeping particles cannot be
 awakened before time $S$. Therefore, we must have~$\eta_S(x) = s$ for every~$x \in A\setminus U$. Besides, note that the sites
 of $U$ start with at least two particles in $\eta_0$ and before time $S$ they are never toppled when they contain only
 one particle (by definition of our strategy~$f'$). Thus, these sites cannot become stable before time $S$, so that
 we have~$\eta_S(x) = 1$ for every~$x \in U$. 
In the end, provided that~$S<\infty$, the only possibility is that~$\eta_S=\tilde\eta$.

We now prove that $(\eta_{S+t}\mathbbm{1}_B)_{t\geq 0}$ is distributed
as the ARWD($\lambda, B, f, \tilde \eta \mathbbm{1}_B$) process, which we recall to have state space~$\cS_B$ and transition kernel~$Q_B$ (see Definition~\ref{def-ARWD}).
That is to say, we have to show that for every~$t\geq 0$ and
every~$\eta'\in\cS_B$, we have
\begin{equation}
\label{Markov-prop}
\bbP\big(\eta_{S+t+1}\mathbbm{1}_B=\eta'\ \big|\
\eta_S\mathbbm{1}_B,\,\ldots,\,\eta_{S+t}\mathbbm{1}_B\big)
\ =\
Q_B\big(\eta_{S+t}\mathbbm{1}_B,\,f(\eta_S\mathbbm{1}_B,\,\ldots,\,\eta_{S+t}\mathbbm{1}_B),\,\eta')
\,.
\end{equation}
Let~$t\geq 0$ and let~$\bar\eta_0,\,\ldots,\,\bar\eta_t\in\cS_A$ with $\eta_0=\bar \eta_0$. Let us assume that the
event~\smash{$\cA=\{\eta_1=\bar\eta_1,\,\dots,\,\eta_t=\bar\eta_t\}$}
has positive probability.
Assume moreover that~$\bar\eta_t\leq\mathbbm{1}_A$ and let~$s\leq t$ be
the first time for which~$\bar\eta_s\leq\mathbbm{1}_A$, so
that~$\cA\subset\{S=s\}$.
Let~$x=f'(\bar\eta_0,\,\ldots,\,\bar\eta_t)$.
By definition of~$f'$, we
have~$x=f(\bar\eta_s\mathbbm{1}_B,\,\ldots,\,\bar\eta_t\mathbbm{1}_B)$.
Let~$\eta'\in\cS_B$.
To prove~\eqref{Markov-prop}, there only remains to show that
\begin{equation}
\label{Markov-prop-bis}
\bbP\big(\eta_{t+1}\mathbbm{1}_B=\eta'\ \big|\ \cA\big)
\ =\
Q_B(\bar\eta_t\mathbbm{1}_B,\,x,\,\eta')
\,.
\end{equation}

First, if~$x=\varnothing$ then both members in~\eqref{Markov-prop-bis}
are by definition equal to~$1$ if~$\eta'=\bar\eta_t\mathbbm{1}_B$ and~$0$ otherwise.
Hence, we assume now that~$x\neq\varnothing$.
Then,
since~$x=f(\bar\eta_s\mathbbm{1}_B,\,\ldots,\,\bar\eta_t\mathbbm{1}_B)$,
by definition of a toppling strategy the site~$x$ must be unstable
in~$\bar\eta_t\mathbbm{1}_B$, that is to say, we must have~$x\in B$
and~$\bar\eta_t(x)=1$.
Thus, defining~$\pi=\bbP(m_t(x)\text{ odd}\,|\,\cA)$,
Lemma~\ref{lemma-add-instructions} ensures that~$\pi\geq 1/(1+\lambda)$.
Therefore, we have
\[
     \bbP\big(I_t=\text{walk}\ \big|\ \cA\big)
     =
     \bbP\bigg(m_t(x)\text{ odd,}\ Z_t\leq\frac 1 {(1+\lambda)\pi_t}\
\bigg|\ \cA\bigg)
     =
     \pi\times \bbP\bigg(Z_t\leq\frac 1 {(1+\lambda)\pi}\bigg)
     =
     \pi\times\frac{1}{(1+\lambda)\pi}
     =
     \frac{1}{1+\lambda}
     \,.
\]
Thus, writing~$p=\lambda/(1+\lambda)$, we have
\[
\bbP\big(\eta_{t+1}\mathbbm{1}_B=\eta'\ \big|\ \cA\big)
=
p\,
\bbP\big(\eta_{t+1}\mathbbm{1}_B=\eta'\ \big|\ \cA,\,I_t=\text{sleep}\big)
+(1-p)\,
\bbP\big(\eta_{t+1}\mathbbm{1}_B=\eta'\ \big|\ \cA,\,I_t=\text{walk}\big)
\,.
\]
Therefore,~\eqref{Markov-prop-bis} follows if we show
\begin{equation}
\label{eq-Qsleep}
\bbP\big(\eta_{t+1}\mathbbm{1}_B=\eta'\ \big|\ \cA,\,I_t=\text{sleep}\big)
=
Q_B^{\text{sleep}}(\bar\eta_t\mathbbm{1}_B,\,x,\,\eta')
\end{equation}
and
\begin{equation}
\label{eq-Qwalk}
\bbP\big(\eta_{t+1}\mathbbm{1}_B=\eta'\ \big|\ \cA,\,I_t=\text{walk}\big)
=
Q_B^{\text{walk}}(\bar\eta_t\mathbbm{1}_B,\,x,\,\eta')
\,.
\end{equation}
First, in~\eqref{eq-Qsleep} both members are equal to~$1$ if~$\eta'=(\bar\eta_t\mathbbm{1}_B)_{\{x\}}$ and~$0$ for all other values of~$\eta'$.
Equation~\eqref{eq-Qwalk} follows similarly, using  that the
instructions in~$\tau$ used to construct the random walk at step~$t$ are
independent of~$\cA$.
Thus, we proved~\eqref{Markov-prop}, which concludes the proof
of~\itref{i-after-S}.

\medskip

We now turn to point~\itref{i-nb-steps}, that we prove by induction on~$t$.
Recall that by construction no site of~$A$ starts with one single active particle, whence~$N_0=0$.
Thus, for~$t=0$ there is nothing to prove, both sides being equal to~$0$.
Let now~$t\geq 0$ be such that~\itref{i-nb-steps} holds at time~$t$
and let us check that~\itref{i-nb-steps} also holds at time~$t+1$.

If~$t\geq T$ then~\itref{i-nb-steps} at time~$t+1$ is identical to \itref{i-nb-steps} at time~$t$.

If~$t+1<S$ then there is nothing to prove, the right-hand side being equal to zero.

If~$t+1=S$ then we have to show that~$\|m_S\|\geq N_S$.
Recall that~\itref{i-after-S} ensures that~$\eta_S=\tilde\eta$, which
entails that~$N_S=|\{x\in A\,:\,\tilde\eta(x)=1\}|=|U|$.
Yet, for every~$x\in U$ we have~$\eta_S(x)=1<\eta_0(x)$, whence~$m_S(x)\geq 1$: each site of~$U$ must have been toppled at least once before time~$S$ because it has lost at least one particle.
Hence, we have~$\|m_S\|\geq|U|=N_S$.

Assume now that~$S\leq t<T$.
Thus, we have~$X_t\in B$ and~$\eta_t(X_t)=1$.
If moreover~$I_t=\text{sleep}$, then we have~$m_{t+1}=m_t$ and~$N_{t+1}=N_t-1$, so that both members of the inequality do not change between rank~$t$ and rank~$t+1$.
Otherwise if~$I_t=\text{walk}$, we have
\begin{multline*}
N_{t+1}
=
N_t
-1
+|R_t\cap\{y\in A\,:\,\eta_t(y)=\s\}|\\
\leq
N_t-1+|R_t\setminus\{X_t\}|
\leq 
N_t-1+|\beta_t|-1
=
N_t+\|m_{t+1}\|-\|m_t\|-2
\,,
\end{multline*}
which, combined with~\itref{i-nb-steps} at rank~$t$, yields \itref{i-nb-steps} at rank~$t+1$.
This concludes the proof of Lemma~\ref{properties-coupled-ARWD}, given Lemma~\ref{lemma-add-instructions} that we prove in the next section.
\end{proof}

\subsection{Proof of Lemma~\ref{lemma-add-instructions}: adding jump instructions to make the odometer odd}
\label{section-proof-add-instructions}

We now prove Lemma~\ref{lemma-add-instructions}.
Note that the proof crucially relies on the facts that
the coupling is restricted to the totally disconnected set~$B$ and
that in the dynamics considered, the particles first move to occupy
all the sites of~$A$.

\begin{proof}[Proof of Lemma~\ref{lemma-add-instructions}]
Let~$t\geq 0$, let~$\bar\eta_0=\eta_0$ and 
let~$\bar\eta_1,\,\dots,\,\bar\eta_t\in\cS_A$.
Let us write~$\cA=\{\eta_1=\bar\eta_1,\,\dots,\,\eta_t=\bar\eta_t\}$ and assume that~$\bbP(\cA\big)>0$.
Let~$x\in B$ and assume that~$\bar\eta_t(x)=1$.
We want to show that
\begin{equation}
\label{goal-ht-odd}
\mathbb{P}\big(m_t(x)\text{ odd}\ \big|\ \cA\big)
\geq
\frac 1 {\Delta^3+1}
=
\frac 1{\lambda+1}
\,.
\end{equation}
Note that, in the process, there are two random ingredients: the field of instructions~$\tau$ and the uniform variables~$(Z_t)$, recalling that the initial configuration is fixed and equal to~$\eta_0$.
Thus, we may express the process until time~$t$
with deterministic functions of~$\tau$
and~$\mathbf{Z}=(Z_0,\,\dots,\,Z_{t-1})$, writing
\smash{$\eta_1^{\tau,\mathbf{Z}},\,\dots,\,\eta_t^{\tau,\mathbf{Z}}$},
\smash{$m_0^{\tau,\mathbf{Z}},\,\dots,\,m_t^{\tau,\mathbf{Z}}$} and
\smash{$R_0^{\tau,\mathbf{Z}},\,\dots,\,R_{t-1}^{\tau,\mathbf{Z}}$},
with the convention~$R_s=\varnothing$ when~$I_s=\text{sleep}$.

Let us denote the set of all possible fields of instructions by
\[
\cI
=
\big\{\bar\tau=(\bar\tau^{z,j})_{z\in V,\,j\geq 1}
\ :\ 
\forall z\in V,\ 
\forall j\geq 1,\ 
\bar\tau^{z,j}\in\{\tau_{zz'},\ z'\sim z\}
\big\}
\]
and, for every~$\mathbf{z}=(z_0,\,\dots,\,z_{t-1})\in[0,1]^t$, let us consider
\[
\cI^{\mathbf{z}}
=
\big\{\bar\tau\in\cI
\ :\ 
\eta_1^{\bar\tau,\mathbf{z}}=\bar\eta_1,\,\dots,\,\eta_t^{\bar\tau,\mathbf{z}}=\bar\eta_t
\big\}
\,,
\]
so that
\[
\mathbb{P}\big(\cA\big)
=
\int_{[0,1]^t}
\cP(\cI^{\mathbf{z}})\,\mathrm{d}^t\mathbf{z}
\,.
\]
Similarly, defining
\[
\cI^{\mathbf{z}}_\text{odd}
=
\big\{\bar\tau\in\cI
\ :\ 
\eta_1^{\bar\tau,\mathbf{z}}=\bar\eta_1,\,\dots,\,\eta_t^{\bar\tau,\mathbf{z}}=\bar\eta_t
\text{ and }m_t^{\bar\tau,\mathbf{z}}(x)\text{ odd}
\big\}
\,,
\]
we have
\[
\mathbb{P}\big(\cA,\ m_t(x)\text{ odd}\big)
=
\int_{[0,1]^t}
\cP(\cI^{\mathbf{z}}_\text{odd})\,\mathrm{d}^t\mathbf{z}
\,.
\]
Hence, to obtain~\eqref{goal-ht-odd} it is enough to show that for every~$\mathbf{z}=(z_0,\,\dots,\,z_{t-1})\in[0,1]^t$, we have
\begin{equation}
    \label{goal-ht-odd-bis}
\cP(\cI^{\mathbf{z}}_\text{odd})
\geq
\frac 1{\Delta^3+1}
\,
\cP(\cI^{\mathbf{z}})
\,.
\end{equation}
Henceforth, we fix~$\mathbf{z}=(z_0,\,\dots,\,z_{t-1})\in[0,1]^t$.
We want to use a counting argument, but a small difficulty lies in the fact that the probability~$\cP$ has no atoms: the probability of every single field~$\bar\tau$ is zero.
Thus, we reason instead with truncated fields: for every~$\bar\tau\in\cI$ and~$m\in\N_0^V$, let us denote the~$m$-truncation of the array~$\bar\tau$ by
\[
\bar\tau^{\leq m}
=
(\bar\tau^{z,j})_{z\in V,\,j\leq m(z)}
\,.
\]
For every~$m\in\N_0^V$, we call~$\cI_m^{\mathbf{z}}$ the subset of truncated arrays of~$\cI^{\mathbf{z}}$ with odometer~$m$, namely,
\[
\cI_m^{\mathbf{z}}
=
\big\{\bar\tau=(\bar\tau^{z,j})_{z\in V,\,j\leq m(x)}
\ :\ 
\exists\, \bar\tau_0\in\cI^{\mathbf{z}},\ 
\bar\tau_0^{\leq m}=\bar\tau
\text{ and }
m_t^{\bar\tau_0,\mathbf{z}}=m
\big\}
\,,
\]
so that, writing~$w(m)=
\cP\big(\tau^{\leq
m}\in\cI^{\mathbf{z}}_m\big)
=\cP\big(\tau\in\cI^{\mathbf{z}},\,m_t^{\tau,\mathbf{z}}=m\big)$, we have
\begin{equation}
\label{proba-Iu}
\cP(\cI^{\mathbf{z}})
=
\sum_{m\in\N_0^V}
w(m)
\qquad\text{and}\qquad
\cP(\cI^{\mathbf{z}}_\text{odd})
=
\sum_{\substack{m\in\N_0^V\,:\\m(x)\text{ odd}}}
w(m)
\,.
\end{equation}
We now claim that, for every~$m\in\N_0^V$ such that~$m(x)$ is even, we have
\begin{equation}
    \label{claim-Iu}
    w(m)
    \leq
    \Delta^2
    \sum_{y\sim x}\,
    w(m+\mathbbm{1}_x+\mathbbm{1}_{y})
    \,.
\end{equation}
Admitting this claim for now and plugging it into~\eqref{proba-Iu}, we can write
\begin{align*}
\cP(\cI^{\mathbf{z}}\setminus\cI^{\mathbf{z}}_\text{odd})
&=
\sum_{\substack{m\in\N_0^V\,:\\m(x)\text{ even}}}
w(m)
    \leq
    \sum_{\substack{m\in\N_0^V\,:\\m(x)\text{ even}}}
    \Delta^2
    \sum_{y\sim x}\,
    w(m+\mathbbm{1}_x+\mathbbm{1}_{y})
\\
    &=
    \Delta^2
    \sum_{y\sim x}
    \sum_{\substack{m'\in\N_0^V\,:\\m'(x)\text{ odd},\,m'(y)\geq 1}}
w(m')
\leq
 \Delta^3\,
  \cP(\cI^{\mathbf{z}}_\text{odd})
 \,,
\end{align*}
which implies the desired inequality~\eqref{goal-ht-odd-bis}.

Thus, to conclude the proof of the lemma, there only remains to prove the claim~\eqref{claim-Iu}.
Note that, for every~$m\in\N_0^V$, we have
\[
w(m)
=
\sum_{\bar\tau\in\cI^{\mathbf{z}}_m}
\cP\big(\tau^{\leq m}=\bar\tau\big)
=
|\cI^{\mathbf{z}}_m|
\prod_{z\in V}\left(\frac 1{d_z}\right)^{m(z)}
\,,
\]
where we recall that~$d_z$ is the degree of~$z$ in the graph~$G$,
so that the claim~\eqref{claim-Iu} will follow if we show that for every~$m\in\N_0^V$ such that~$m(x)$ is even,
\begin{equation}
\label{claim-Iu-bis}
|\cI^{\mathbf{z}}_m|
    \leq
    \sum_{y\sim x}\,
|\cI^{\mathbf{z}}_{m+\mathbbm{1}_x+\mathbbm{1}_y}|
\,.
\end{equation}
We now fix~$m\in\N_0^V$ with~$m(x)$ even and we construct an injective mapping from~$\cI^{\mathbf{z}}_m$ to~$\cup_{y\sim x}\cI^{\mathbf{z}}_{m+\mathbbm{1}_x+\mathbbm{1}_y}$.
The idea is to modify the array of instructions by adding two jump instructions at~$x$ and at a well chosen neighbouring site~$y$, at a well chosen position in the array.

Note that our assumption that~$\bar\eta_t(x)=1$ implies that~$x$ was visited at least once before time~$t$, because by definition~$\bar\eta_0(x)\neq 1$.
Thus, for every~$\bar\tau\in \cI^{\mathbf{z}}_m$ we can
define~$\sigma(\bar\tau)=\max\{s<t\ :\ R_s^{\bar\tau,\mathbf{z}}\ni x\}$.

Now, for every~$\bar\tau\in \cI^{\mathbf{z}}_m$,
writing~$\sigma=\sigma(\bar\tau)$
and~$u=u(\bar\tau)=f'(\bar\eta_0,\,\dots,\,\bar\eta_\sigma)$ we define, using an arbitrary order on the sites of~$V$,
\[
y(\bar\tau)
=
\min\Big(\big\{y\sim x\,:\,y\notin A\big\}
    \cup    
    \big\{y\sim x\,:\,\bar\eta_\sigma(y)-\mathbbm{1}_u(y)\neq 0\big\}\Big)
    \,.
\]
Let us explain why the above set cannot be empty, distinguishing between two cases.

One the one hand, if the above set is empty and~$u\sim x$, then we
have~$u\in A$ and~$\bar\eta_\sigma(u)=1$, while the other neighbours
of~$x$ are empty in~$\bar\eta_\sigma$, but also belong to~$A$.
This situation is impossible due to the priority rules of our toppling
strategy~$f'$ and to the fact that~$|\eta_0|=|A|$, because a site of~$A$ with one single particle can be chosen by~$f'$ only once that each site of~$A$ is occupied by one particle (in other words, only after time~$S$).
Note that we used here that~$x$ has at least two neighbours, following
our assumption that each site of~$B$ has degree at least~$2$ in the
graph~$G$.

One the other hand, let us now assume that~$u$ is not a neighbour
of~$x$ and that the set in the definition of~$y(\bar\tau)$
is empty. Then for every~$y$ neighbour of~$x$ we have~$y\in A$ and~$\bar\eta_\sigma(y)=0$.
Note that it also follows from the priority rules of~$f'$ (and from the fact that~$|\eta_0|=|A|$) that an occupied site of~$A$ cannot become empty in a later configuration of the ARWD process.
Therefore, our assumptions imply that no particle can have
visited~$x$ before time~$\sigma$,
i.e.,~$m_\sigma^{\bar\tau,\mathbf{z}}(x)=0$ (otherwise this particle
would have visited one of the neighbours of~$x$, and would have settled
there).
Moreover, since~$\bar\eta_t(x)=1$ and~$\sigma$ is the last visit of~$x$
before~$t$, we have~$\bar\eta_{\sigma+1}(x)=\bar\eta_t(x)=1$.
Yet, by definition of~$\sigma$ the site~$x$ must be visited at
step~$\sigma$, so
that we must have~$u=x$ and~$\bar\eta_\sigma(x)=2$,
so that in the next step one particle at~$x$ jumps to one of the
neighbouring sites and then terminates its walk there (because we
assumed that all the neighbouring sites are in~$A$ and empty
in~$\bar\eta_\sigma$).
This entails that~$m_{\sigma+1}^{\bar\tau,\mathbf{z}}(x)=1$.
Yet, by definition of~$\sigma$ we
have~$m(x)=m_t^{\bar\tau,\mathbf{z}}(x)=m_{\sigma+1}^{\bar\tau,\mathbf{z}}(x)$, whence~$m(x)=1$, which contradicts our assumption that~$m(x)$ is even.
Therefore, we checked that the set in the minimum defining~$y(\bar\tau)$ is never empty.

Now, once~$y(\bar\tau)$ is well defined, we are ready to construct our application~$\varphi:\cI_{m}^{\mathbf{z}}\to \cup_{y\sim x}\cI_{m+\mathbbm{1}_x+\mathbbm{1}_y}^{\mathbf{z}}$.
For every~$\bar\tau\in \cI_{m}^{\mathbf{z}}$, every~$z\in V$ and
every~$1\leq j\leq (m+\mathbbm{1}_x+\mathbbm{1}_y)(z)$,
writing~$\sigma=\sigma(\bar\tau)$ and~$y=y(\bar\tau)$, we set
\[
\varphi(\bar\tau)^{z,j}
=
\begin{cases}
    \bar\tau^{z,j}
    &\text{if }z\notin\{x,y\}\,,\\
    
    \bar\tau^{z,j}
    &\text{if }z\in\{x,y\} \text{ and }j\leq m_\sigma^{\bar\tau,\mathbf{z}}(z)\,,\\

    \tau_{xy}
    &\text{if }z=x \text{ and }j=m_\sigma^{\bar\tau,\mathbf{z}}(x)+1\,,\\
    
    \tau_{yx}
    &\text{if }z=y \text{ and }j=m_\sigma^{\bar\tau,\mathbf{z}}(y)+1\,,\\

    \bar\tau^{z,j-1}
    &\text{if }z\in\{x,y\} \text{ and }j\geq m_\sigma^{\bar\tau,\mathbf{z}}(z)+2\,.
\end{cases}
\]

The application~$\varphi$ consists in adding two jump instructions
from~$x$ to~$y=y(\bar\tau)$ and from~$y$ to~$x$ so that these
instructions add one round trip between these two vertices in the path
of the random walk at step~$\sigma$, thus changing the parity of the odometer~$m_t$ at~$x$ and~$y$, but leaving the resulting configuration unchanged.

Note that, using the field of instructions~$\bar\tau$ (i.e., before
applying~$\varphi$), the walk at
step~$\sigma$ did not terminate at~$x$ because, if it
was the case, then the particle would have settled at~$x$, sleeping,
i.e., we would have~$\bar\eta_{\sigma+1}(x)=\s$, which would imply
that~$\bar\eta_t(x)=\s$, since~$x$ is no
longer visited in the steps between~$\sigma+1$ and~$t$ (by definition
of~$\sigma$).
This would contradict our assumption that~$\bar\eta_t(x)=1$.
Hence, since the walk at step~$\sigma$ visits~$x$ but does not
terminate there, although~$x\in A$, we must
have~$\bar\eta_\sigma(x)-\mathbbm{1}_u(x)\neq 0$.
As a consequence, also after adding the two jump instructions, the
walk at step~$\sigma$ still cannot stop at~$x$.

Moreover, by definition of~$y$, we have either~$y\notin A$ or~$\bar\eta_\sigma(y)-\mathbbm{1}_u(y)\neq 0$.
Thus, the walk cannot terminate neither at~$y$.
Hence, the two added instructions are indeed used during the
step~$\sigma$: if before applying~$\varphi$, the walk visited~$x$ before~$y$, then~$\varphi$ adds a round trip~$x\to y\to x$, while if it visited~$y$ first, then it adds a round trip~$y\to x\to y$.

Moreover, this added round trip does not change anything else than
this increase in the odometer at~$x$ and~$y$, so that for every~$s\leq
t$ we have
\smash{$\eta_s^{\varphi(\bar\tau),\mathbf{z}}=\eta_s^{\bar\tau,\mathbf{z}}=\bar\eta_s$}.
Indeed, by definition of~$\sigma$, the site~$x$ is never visited again.
The site~$y$ may be visited again between time~$\sigma$ and time~$t$ but, since~$y\notin B$ (recall that we assumed that~$B$ does not contain neighbouring points), the change of odometer at~$y$ has no impact on the evolution of the coupled ARWD process (which looks at the parity of the odometer only when we topple a site of~$B$ with one particle).
All this shows that~$\varphi(\bar\tau)\in\cI_{m+\mathbbm{1}_x+\mathbbm{1}_{y}}^{\mathbf{z}}$.

Lastly, note that applying~$\varphi$ does not change the value
of~$\sigma$,
i.e.,~$\sigma(\varphi(\bar\tau))=\sigma(\bar\tau)$, and neither does
it change the value of~$y$ or~$m_\sigma$.
Hence, the initial array~$\bar\tau$ may be recovered from the
knowledge of~$\varphi(\bar\tau)$ by simply removing the two
instructions on~$x$ at index~$m_\sigma(x)+1$ and on~$y$ at
index~$m_\sigma(y)+1$, and renumbering the other instructions.
This shows that~$\varphi$ is injective, which concludes the proof of the claim~\eqref{claim-Iu-bis}, and therefore of the lemma.
\end{proof}

\section{The stabilisation time of ARWD: proof of Proposition~\ref{prop-ARWD-stab-time}}
\label{section-ARWD-stab-time}

We now prove Proposition~\ref{prop-ARWD-stab-time} which states that the stabilisation time of the ARWD model is exponentially large.
The proof is an adaptation of the technique developed in \cite{AFG} in
the context of the activated random walk model.
Yet, since there are some important differences and the proof is technical, we present it in details, stressing the differences with respect to \cite{AFG}.

\subsection{Preliminaries}

As explained in Section~\ref{sec-prop-ARWD-stab-time}, we will consider a random toppling strategy which consists in a choice at each time step of a random set of sites where reactivations are ignored.

\subsubsection{Ignoring some reactivations: variant of ARWD with sleep mask}
\label{sec:def-ARWD-bis}

We now define a variant of the ARWD process, which is obtained by ignoring some reactivations at each step.
The rule indicating which reactivations should be ignored is encoded into a sleep mask, which is defined as follows:
\begin{definition}
    A sleep mask is an application~$g:\cup_{t\geq 0}(\N_\s^V)^{t+1}\to\cP(V)$.
\end{definition}
The idea is simply that the sleep mask~$g(\eta_0,\,\ldots,\,\eta_t)$ will indicate a set of sites which are allowed to be reactivated at step~$t$, and the reactivation of particles outside this set will be ignored.
We now define the variant of ARWD with sleep mask, overriding the previous notation:

\begin{definition}
Let~$G=(V,E)$ be a finite connected graph, let~$\lambda\geq 0$, let~$A\subset V$, let~$f$ be a toppling strategy and~$g$ a sleep mask, let~$\cS_A$ be, as in Definition~\ref{def-ARWD}, the set of configurations with exactly~$|A|$ particles and let~$\eta_0\in\cS_A$ be a fixed initial configuration.
Then, we call~$\mathrm{ARWD}(\lambda,\,A,\,f,\,g,\,\eta_0)$ the process~$(\eta_t)_{t\geq 0}$ with state space~$\cS_A$ and initial configuration~$\eta_0$ which is such that, for every~$t\geq 0$ and~$\eta'\in\cS_A$, we have
\[
\bbP\big(\eta_{t+1}=\eta'\ \big|\ \eta_0,\,\dots,\,\eta_t \big)
= Q_A\big(\eta_t,\,f(\eta_0,\,\dots,\,\eta_t),\,g(\eta_0,\,\dots,\,\eta_t),\,\eta'\big),
\]
where~$Q_A:\cS_A\times\big(V\cup\{\varnothing\}\big)\times\cP(V)\times\cS_A\to[0,1]$ is defined as follows:
\begin{itemize}
\item For any $\eta,\,\eta' \in \cS_A$ and any~$W\subset V$, we have~$Q_A(\eta, \,\varnothing,\,W,\,\eta')=\mathbbm{1}_{\eta'=\eta}$.

    \item If~$\eta\in\cS_A$ is an unstable configuration and~$\eta'\in\cS_A$, if~$x\in V$ is an unstable site of~$\eta$ and~$W\subset V$, then~$Q(\eta,\,x,\,W,\,\eta')=p\,Q_A^{\text{sleep}}(\eta,\,x,\,\eta')+(1-p)Q_A^{\text{walk}}(\eta,\,x,\,W,\,\eta')$, with~$p$ and~$Q_A^{\text{sleep}}$ as in Definition~\ref{def-ARWD} and, defining~$S=\{z\in W\,:\,\eta(z)=\s\}$, which is the set of sites which may be reactivated, for every~$y\in A$
and~$R\subset S$,
\[
Q_A^{\text{walk}}\big(\eta,\,x,\,W,\,\eta^{R}-\mathbbm{1}_x+\s\mathbbm{1}_y)
\ =\ 
P_x\big(X(\tau_B^+)=y\text{ and }\{X(t),\,0\leq t< \tau_B^+\}\cap S=R\big)
\,,
\]
using the same notation as in Definition~\ref{def-ARWD}, with~\smash{$\tau_B^+=\inf\{t\geq 1\,:\,X(t)\in B\}$} and the set~$B$ given by~$B=\{z\in A\,:\,(\eta-\mathbbm{1}_x)(z)=0\}$.
\end{itemize}
\end{definition}

Note that, compared to Definition~\ref{def-ARWD}, the only difference is in the definition of the set~$S$ of sites which may be reactivated, which was the set of sites with a sleeping particle and is now the set of sites of~$W$ with a sleeping particle.
Hence, this definition generalizes the previous one, because if the sleep mask is the constant function always equal to~$V$, we recover the ARWD process as defined before.

Then, similar to Definition~\ref{def-stab-time-ARWD}, we define~$T(A,\,f,\,g,\,\eta_0)=\inf\{t\geq 0\,:\,f(\eta_0,\,\ldots,\,\eta_t)=\varnothing\}$.

Note that it is not clear whether two instances of this process with different sleep masks can be compared, but we have the following property:

\begin{lemma}
\label{lemma-comparison-sleep-mask}
    For every finite connected graph~$G$, every~$\lambda\geq 0$ and~$A\subset V$, for any initial configuration~\smash{$\eta_0\in\cS_A$}, any toppling strategy~$f$ and any sleep mask~$g$, there exists a toppling strategy~$f'$ such that~$T(A,\,f',\,\eta_0)$ (as defined in Section~\ref{sec:def-ARWD}, i.e., without sleep mask) stochastically dominates~$T(A,\,f,\,g,\,\eta_0)$.
\end{lemma}

\begin{proof}
    The idea is simply to couple the~$\mathrm{ARWD}(\lambda,\,A,\,f,\,g,\,\eta_0)$ process~$(\eta_t)_{t\geq 0}$ with an~$\mathrm{ARWD}(\lambda,\,A,\,f',\,\eta_0)$ process~$(\eta'_t)_{t\geq 0}$ by taking~$\eta'_0=\eta_0$ and, at each time~$t<T(A,\,f,\,g,\,\eta_0)$, considering the toppling strategy given by~$f'(\eta'_0,\,\ldots,\,\eta'_t)=f(\eta_0,\,\ldots,\,\eta_t)$, and using the same random walks and the same randomness to decide between falling asleep and walking.
    That is to say, until time~$T(A,\,f,\,g,\,\eta_0)$ we always choose the same site in the two processes.
    Then, for every time~$t\leq T(A,\,f,\,g,\,\eta_0)$ the two configurations~$\eta_t$ and~$\eta'_t$ only differ by the fact that some sleeping particles in~$\eta_t$ may be active in~$\eta'_t$.
    
    Yet, for~$f'$ to be a well-defined toppling strategy, there remains to explain that for every~$t\geq 0$,~$(\eta_0,\,\ldots,\,\eta_t)$ is a deterministic function of~$(\eta'_0,\,\ldots,\,\eta'_t)$.
    This is the case because, at each time step, the evolution of the process with the sleep mask can be deduced from the observation of the process without sleep mask, by simply ``cancelling'' the reactivations that the sleep mask tells us to ignore.

    Lastly, by construction we have~$T(A,\,f',\,\eta_0)\geq T(A,\,f,\,g,\,\eta_0)$.
\end{proof}

\subsubsection{Reformulated goal: bound on ARWD with sleep mask}
\label{sec:reformulated-goal}

Thanks to the comparison result given by Lemma~\ref{lemma-comparison-sleep-mask}, to obtain Proposition~\ref{prop-ARWD-stab-time} it is enough to show a similar result for the variant of the ARWD process with a sleep mask.

Besides, since Proposition~\ref{prop-ARWD-stab-time} is restricted to the setting in which there is always one particle on each site of the settling set~$A$, we can choose to represent configurations with the subset~$U\subset A$ of the unstable sites, that is to say, a configuration~$\eta:V\to\N_s$ with one particle on each site of~$A$ is represented by the set~$U=\{x\in A\,:\,\eta(x)=1\}$.

From now on we use this new notation, so that the ARWD process can be denoted by~$(U_t)_{t\geq 0}$.
A toppling strategy becomes an application~$f:\cup_{t\geq 0}\cP(A)^{t+1}\to A\cup\{\varnothing\}$, still with the constraint that the value of~$f(U_0,\,\ldots,\,U_t)$ should always be either a site of~$U_t$ (i.e., an site which is unstable at time~$t$) or~$\varnothing$.
Similarly, a sleep mask can be written~\smash{$g:\cup_{t\geq 0}\cP(A)^{t+1}\to \cP(A)$}.

To sum up, at each time step~$t$, if~$f(U_0,\,\ldots,\,U_t)=\varnothing$ nothing happens, and otherwise, one particle at the site~$X_t=f(U_0,\,\ldots,\,U_t)$ falls asleep with probability $\lambda/ (1+\lambda)$ if it is alone and, with probability~$1/(1+\lambda)$, it
makes a random walk until it comes back to~$X_t$.
 Along this loop, the sleeping particles met are waken up, but only those which belong to the set~$g(U_0,\,\ldots,\,U_t)$ indicated by the sleep mask.

For~$\r\in(0,1)$ and~$C\subset\Z_n^d$, we denote by~$\cP_{C,\,\r}$ the
Bernoulli site percolation measure on~$C$ with parameter~$\r$, that is
to say, the probability measure on~$\cP(C)$ defined
by~$\cP_{C,\,\r}(U)=\r^{|U|}(1-\r)^{|C\setminus U|}$ for each~$U\subset C$.
This corresponds to the initial configuration considered in Proposition~\ref{prop-ARWD-stab-time}, with this new notation in terms of subsets.

Lastly, for the probability space~$(\Sigma,\,\cF,\,\nu)$ mentioned in Proposition~\ref{prop-ARWD-stab-time} we choose to consider~$\Sigma=\N_0^{\N_0}$ equipped with the product sigma-field and the product measure~$\nu$ which is such that, if~$J=(J_t)_{t \geq 0}$ has distribution~$\nu$, the variables~$(J_t)_{t \geq 0}$ are i.i.d.\ and~$1+J_0$ is a geometric variable with parameter $1/2$.
Then, to each possible realization~$\mathbf{j}=(j_t)_{t\geq 0}\in\Sigma$ of this sequence~$J$ we will associate a sleep mask~$g^\mathbf{j}$, which will be precised in Section~\ref{sec-def-g}.
The toppling strategy~$f$, in turn, will be fixed and will not depend on the value of~$J$.

To sum up, in the remainder of this section we aim to show the following result:

\begin{proposition}
    \label{prop-ARWD-stab-time-bis}
Let~$d\geq 2$.
    For every~$\lambda\geq 0,\,\mu\in(0,1)$ and~$\rho\in(0,1)$, there
exists~$\kappa>0$ such that, for every~$n\geq 1$, for
every~$A\subset\Z_n^d$ such that~$|A|\geq \mu n^d$, there exists a toppling strategy~$f$ and a collection of sleep masks~$(g^\mathbf{j})_{\mathbf{j} \in \Sigma}$
such that, if~$J$ is a random variable with distribution~$\nu$, then, starting
with a random initial configuration~$U_0$ with distribution~$\cP_{A,\,\r}$, independent of~$J$, we have that~$1+T(A,\,f,\,g^J,\,U_0)$ dominates a geometric variable with parameter~$\exp(-\kappa n^d)$.
\end{proposition}

Through Lemma~\ref{lemma-comparison-sleep-mask}, this proposition implies Proposition~\ref{prop-ARWD-stab-time}.
Indeed, for each fixed value~$\mathbf{j}\in\Sigma$, Lemma~\ref{lemma-comparison-sleep-mask} ensures the existence of a toppling strategy~$f^\mathbf{j}$ such that~$T(A,\,f^\mathbf{j},\,U_0)$ dominates~$T(A,\,f,\,g^\mathbf{j},\,U_0)$, so that the deterministic toppling strategy~$f$ and the random sleep mask~$g^J$ result in a random strategy~$f^J$.

\subsubsection{Geometric notation and hitting probabilities}
\label{section-notation-ARWD-stab-time}

For convenience, in this section we use the following notion of distance on the torus, which does not correspond to the graph distance.
Let~$d,\,n \geq 1$ and~$G=\Z_n^d$. Denoting by~$\pi_n : \Z^d \to \Z_n^d$
the standard projection from~$\Z^d$ onto the torus, we define the distance between two points~$x, y \in \Z^d_n$ as
\begin{equation}
\label{def-distance}
    d(x,y)=\inf\big\{\|a-b \|_\infty\ :\ a, b \in \Z^d,\,\pi_n(a)=x,\,\pi_n(b)=y\big\}.
\end{equation}
We consider this distance because we rely on some results
of~\cite{AFG} which use this distance (the reason behind being that
then the volume of balls has a simple expression), but
things would work equally with the graph distance.
For every non-empty set~$C\subset  \Z_n^d$, we define its diameter
\begin{equation*}
    \diam\,C=\max_{x,y \in C}d(x,y).
\end{equation*}
If~$r \in \N$, a set~$C \subset \Z_n^d$ is said to be~$r$-connected if, for any two points~$x, y \in C$, there exists~$k \in \N$ and a sequence~$x_0, \dots , x_k \in C$ such that~$x_0 = x$,~$x_k = y$ and~$d(x_j , x_{j+1}) \leq r$ every~$j < k$.

Besides, let us introduce the following function, which measures the chance to wake up a distant site in a loop
 on the torus~$\Z_n^d$, for~$d \geq 1$:
 \begin{equation*}
     \Upsilon_d\,:\, r \in \N \ \longmapsto\ \inf \big\{P_x(\tau_y<\tau_x^+),\,n \in \N,\,x,y \in \Z_n^d\,:\,d(x,y)\leq r\big\}.
 \end{equation*}
 We use the following estimates on this function (this lemma is identical to Lemma 2.5 of \cite{AFG}):
 
 \begin{lemma}
 \label{lemmahittingprob}
 We have the following estimates:
     \begin{itemize}
         \item In dimension~$d = 1,$ we have~$\Upsilon_1(r) = 1/(2r)$ for every~$r \geq 1$;
         \item In dimension~$d = 2,$ there exists~$K > 0$ such that~$\Upsilon_2(r) \geq  K/\ln r$ for every~$r \geq 2$;
         \item In dimension~$d \geq 3,$ there exists~$K = K(d) > 0$ such that~$\Upsilon_d(r) \geq K$ for every~$r \geq 1$.
     \end{itemize}
 \end{lemma}

\subsection{The dormitory hierarchy}
\label{section-hierarchy}

As explained in \cite{AFG}, given a settling set~$A \subset \Z_n^d$, 
we introduce a hierarchical structure called the dormitory hierarchy which is
deterministically associated to the set~$A$ and also depends on some parameters~$v$ and~$(D_j )_{j\in \N_0}$.

\subsubsection{Definition}

A dormitory hierarchy is defined as follows, exactly as in \cite{AFG}:

\begin{definition}
Let~$d,\,v,\,n\geq 1$ and let~$D=(D_j)_{j \in \N_0} \in \N^{\N_0}$. For every~$A\subset \Z_n^d$ we call we call a~$(v, D)$- dormitory hierarchy on~$A$ a non-increasing sequence of subsets~$A \supset A_0 \supset \dots \supset A_{L}$, with~$L \in \N_0$ and,
for every~$j \leq L$, a partition~$\cC_j$ of~$A_j$ such that:

\begin{enumerate}[(i)]
    \item
    For~$0 \leq j \leq L$, for every~$C \in \cC_j~$, we have~$|C|\geq 2^{\lfloor j/2 \rfloor}v$;
    
    \item\label{conditionMerge}
    For~$0\leq j\leq L-1$, for every~$C \in \cC_{j+1}\setminus \cC_j$, we have~$\diam\,C\leq D_j$ and there exist two sets~$C_0, C_1\in \cC_j$ such that~$C=C_0\cup C_1$;
    
    \item 
    The last partition~$\cC_L$ contains one single set.
\end{enumerate}
\end{definition}

Given a dormitory hierarchy~$(A_j , \cC_j )_{j\leq L}$, for every~$j \leq L$ and every~$x \in A_j$, we define~$\cC_j (x)$ to
be the set~$C \in \cC_j$ such that~$x \in C$. When~$x \in \Z_n^d\setminus A_j$ or~$j > L$, we set~$C_j (x) = \varnothing$. The sets~$C \in \cC_j$ are called clusters at the level~$j$. The parameter~$v \in \N$ controls the volume of the clusters at each
level of the hierarchy, while the sequence of diameters~$D_j$ ensures that we only merge clusters which
are not too far apart.

\subsubsection{Construction of the hierarchy for the two-dimensional case}

We now build the dormitory hierarchy used for our proof in dimension~$2$.
Compared with~\cite{AFG}, we can afford to use a simpler hierarchy
because we only aim to show that the critical density is less than
one, whereas in~\cite{AFG} the parameters were optimized to obtain
good bounds on the critical density.
Note that in the simpler case of dimension at least~$3$ we use the trivial hierarchy with only one level.

\begin{lemma}\label{lemma3.3}
Let~$d = 2$, let~$\mu\in(0,1)$ and~$v\in\N$.
Let us write \smash{$r=2\big\lfloor\sqrt{2 v/\mu}\big\rfloor$} and~$D_j = 6^j\times 12 vr$
for every~$j \in \N_0$.
For every~$n\geq r+1$ and
 every~$A \subset \Z^2_n$  with~$|A|\geq \mu n^2$, there exists~$L \in \N_0$ and a~$(v, D)$-dormitory hierarchy~$(A_j, \cC_j)_{j\leq L}$
 on~$A$, with~$|A_{L}|\geq |A|/4$ and where every set~$C \in \cC_0$ is~$r$-connected.
\end{lemma}

\begin{proof}
Let~$d,\,\mu,\,v,\,r,\,D_j,\,n$ and~$A$ be as in the statement.
First, we simply let~$\cC_0$ be the set of the~$r$-connected components of~$A$ of cardinality at least~$v$, and~$A_0=\cup_{C\in\cC_0}C$.
Recall that we consider the distance coming from the infinite norm, as defined in~\eqref{def-distance}.
Thus, each square of side~$r+1$ can intersect at most one~$r$-connected component of~$A$.
Since each of the~$r$-connected components of~$A\setminus A_0$ contains at most~$v-1$ vertices, we have
\[
|A\setminus A_0|
\leq
\bigg\lceil\frac{n}{r+1}\bigg\rceil^2 (v-1)
\leq
\bigg(\frac n{r+1}+1\bigg)^2 v
\leq
\frac {4 n^2v}{(r+1)^2}
\leq
\frac {\mu n^2}{2}
\leq
\frac {|A|}{2}
\,,
\]
whence~$|A_0|\geq|A|/2$.
Then, we apply Lemma 5.1 of \cite{AFG} to obtain the remainder of the hierarchy.
The conditions required for this lemma are satisfied because~$D_0/(12v)=r$ and 
\[
|A_0|
\geq\frac{|A|}2
\geq \frac{\mu n^2}2
\geq \frac{2n^2}{r^2v}
=8v\bigg(\frac{6n}{D_0}\bigg)^2
\,.
\]
This lemma completes~$(A_0,\,\cC_0)$ into a~$(v,D)$-dormitory hierarchy~$(A_j, \cC_j)_{j\leq L}$
 on~$A$, with the final set satisfying 
 \[
 |A_{L}|
 \geq |A_0|-4v\bigg(\frac{6n}{D_0}\bigg)^2
 \geq |A_0|-\frac{\mu n^2}4
 \geq \frac{|A|}4
 \,,
 \]
 as announced.
\end{proof}

\subsubsection{Distinguished vertices}

 Let~$A \subset \Z^d_n$, and let~$(A_j, \cC_j)_{j\leq L}$ be a dormitory hierarchy on~$A$. We define recursively a distinguished vertex in each set of the partitions.
 The distinguished point of a set~$C$ is written~$x_C^\star$
 and the particle sitting in~$x^\star_C$
 is called the distinguished
 particle of the cluster~$C$. For every~$C \in \cC_0$, we simply set~$x^\star_C=\min C$, for an arbitrary order on
 the vertices of the torus. Then, for~$1 \leq j \leq L$, if~$C \in \cC_j \setminus \cC_{j-1}$, the property~\itref{conditionMerge} of the hierarchy
 tells us that~$C$ is the union of two clusters of~$\cC_{j-1}$. In this case, we let~$x^\star_C$ be the distinguished
 vertex of the biggest of these two clusters (in terms of number of vertices, and with an arbitrary
 rule to break ties).
 We say that a vertex~$x$ is distinguished at the level~$j \leq L$ if there exists~$C \in  \cC_j$ such that~$x = x_C^\star$,
 that is to say, if \smash{$x = x^\star_{C_j(x)}$}.
 If~$j > L$, we say that no vertex is distinguished at level~$j$. Note that
 if~$x$ is distinguished at a certain level~$j$, then it is also distinguished at all levels~$j'$ for~$j' < j$.

\subsection{Construction of the toppling strategy and sleep mask}
\label{section-def-f-g}

We now explain which toppling strategy and sleep mask 
we choose to consider.
Let~$A \subset \Z^d_n$, and let~$(A_j, \cC_j)_{j\leq L}$ be a fixed dormitory hierarchy on~$A$.

\subsubsection{Toppling strategy}

At the first level of the hierarchy,  
we use the notion of~$C$-toppling procedure, which was already used in~\cite{AFG}:

\begin{definition}
    For every level~$j\leq L$ and every cluster~$C\in\cC_j$,
a~$C$-toppling procedure is an
application~\smash{$p:\cP(C)\setminus\{\varnothing\}\to C$} which is such that for every~$U\in\cP(C)\setminus\{\varnothing\}$, we have~$p(U)\in U$ and, if~$x_C^\star\in U$ then~$p(U)=x_C^\star$.

\end{definition}

 In words, a~$C$-toppling procedure is a rule to choose an active site of~$C$ in a configuration which is unstable on~$C$, with the rule that the distinguished particle is chosen in priority, as soon as it is unstable.

 We construct in Lemma~\ref{lemma3.4} below a collection
of~$C$-toppling procedures~$p_C$ for each~$C\in\cC_0$.
 Let us assume for the moment that this collection of~$C$-toppling procedures is constructed, and let us define recursively for every~$j\leq L$, for every~$C\in\cC_j$, a toppling strategy~$f_C$ to stabilise the set~$C$.

First, if~$C\in\cC_0$, then for every~$t\geq 0$
and~$U_0,\,\dots,\,U_t\subset A$ we simply
define
\[
f_C(U_0,\,\dots,\,U_t)
\ =\ 
\begin{cases}
p_C(U_t\cap C)
&\text{if }U_t\cap
C\neq\varnothing\,,\\
\varnothing
&\text{otherwise.}
\end{cases}
\]

Now, let~$0\leq j<L$, assume that we defined a toppling strategy~$f_C$ for each~$C\in\cC_j$, and let~$C\in\cC_{j+1}$.
If~$C$ also belongs to~$\cC_{j}$ then~\smash{$f_C$} is already constructed.
Otherwise, if~$C=C_0\cup C_1$ with~$C_0,\,C_1\in\cC_j$, then for
every~$t\geq 0$ and every~$U_0,\,\dots,\,U_t\subset A$,
for~$k\in\{0,1\}$ we denote by~$T_k$ the last instant before~$t$ when 
 the cluster~$C_k$ was stable, that is to say,
\[
T_k
=
\sup\big\{s\in\{0,\,\dots,\,t\}\,:\,  U_s\cap C_k=\varnothing\big\}
\,,
\]
with the convention that~$T_k=-\infty$ if the above set is empty, and we let
\[
f_C(U_0,\,\ldots,\,\,U_t)
=
\begin{cases}
    f_{C_0}( U_0,\,\dots,\,U_t)
    &\text{if }T_0\leq T_1\,,\\
    f_{C_1}( U_0,\,\dots,\,U_t)
    &\text{otherwise.}
\end{cases}
\]

The above definition encodes the so-called ping-pong rally used in \cite{AFG}: we first stabilise~$C_0$ using the strategy~$f_{C_0}$, until~$C_0$ is stable, and then stabilise~$C_1$ using the strategy~$f_{C_1}$.
Then, if some sites of~$C_0$ have been reactivated during the stabilisation of~$C_1$, we stabilise again~$C_0$, and then again~$C_1$ if necessary, and so on until both sets are stabilised.

This recursive procedure yields a toppling strategy~$f_C$ for every cluster~$C\in\cC_j$ at every level~$0\leq j\leq L$.
Then, the global toppling strategy that we consider is~\smash{$f_{A_L}$}, the strategy associated to the top-level cluster~$A_{L}$ with the above recursive construction.
We simply denote it by~$f$.

\subsubsection{Sleep mask}
\label{sec-def-g}

To define the sleep mask, let us fix~$\mathbf{j}=(j_t)_{t\geq 0}\in\Sigma$ a given realization of the random variable~$J$.
For every~$t\geq 0$ and every~$U_0,\,\ldots,\,U_t\subset A $, writing~$x=f(U_0,\,\ldots,\,U_t)$ and~$j=j_t$,
we define
\[
g^\mathbf{j}(U_0,\,\ldots,\,U_t)
\ =\ 
\begin{cases}
\varnothing
&\text{if }x\text{ is distinguished at level }j+1\text{ and }U_s\cap C_j(x)\neq\varnothing\text{ for all }s\leq t\,,\\
    C_{j+1}(x)\setminus C_{j}(x)
&\text{otherwise if }x\text{ is distinguished at level }j\,,\\
\varnothing
&\text{if }x \text{ is distinguished at level }0\text{ but not at level }j\,,\\
C_0(x)
&\text{if }x\text{ is not distinguished at any level.}
\end{cases}
\]

Thus, at each time step~$t\geq 0$, the behaviour of the sleep mask depends on~$j_t$, the value taken by the variable~$J_t$.
As in~\cite{AFG}, this can be understood as, at each time step, attributing to the loop a random color in $\N_0$ and the set of sites that
 the loop can reactivate depends on this color.
 The idea is that this
 sleep mask enables us to gain some independence at certain key steps of the process.  It is used in particular
 in the proofs of Proposition \ref{propdom} and Lemma \ref{lemma3.1}.

If we compare the definition of $g^\mathbf{j}$ to the definition of the sleep mask~$w(x,\,j)$ in \cite{AFG}, we added the exception of the first case, which is meant to prevent the distinguished vertex of~$C_0$ to wake up sites in~$C_1$ during the first stabilisation of~$C_0$ in the ping-pong rally, if~$C_0$ and~$C_1$ are two clusters which merge in the hierarchy.
 This will be useful for the proof of Lemma~\ref{lemma3.1}, to avoid a dependence issue which comes from the fact that we do not start with a fully active configuration.
For a diagram summarizing the effect of the sleep mask, we refer to Figure~1 in~\cite{AFG}.

 Given this toppling strategy~$f$ and this sleep mask~$g^\mathbf{j}$, for
every~$j\leq L$ and every~$C\in\cC_j$, we denote the total number of
steps (sleeps or loops) performed by the distinguished vertex of~$C$
during the stabilisation of~$C$, starting with the configuration~$U_0$, by
 
 \[
 H^*(C,\,f,\,g^\mathbf{j},\,U_0)
 \ =\ 
 \big|\big\{t<T(C,\,f,\,g^\mathbf{j},\,U_0)\,:\,  f(U_0,\,\dots,\,U_t)=x_C^\star\big\}\big|
 \,,
 \]
where~$(U_t)_{t\geq 0}$ is
the~$\mathrm{ARWD}(\lambda,\,A,\,f,\,g^\mathbf{j},\,U_0)$ process
and~$T(C,\,f,\,g^\mathbf{j},\,U_0)$ is  as defined in Section~\ref{sec:def-ARWD-bis}.
Note that, with the toppling strategy that we constructed,~$T(C,\,f,\,g^\mathbf{j},\,U_0)$ is the first time when the configuration is stable in the set~$C$, so that it is really the ``stabilisation time'' of~$C$.

\subsection{Initialization: stabilisation of an elementary cluster}

We now show that, for every cluster of the first level of the hierarchy constructed in Section~\ref{section-hierarchy}, the number of steps made at the distinguished vertex is exponentially large in the size of this cluster.

\begin{lemma}\label{lemma3.4}
    Assume that~$d=2$ and let~$\lambda> 0,\,\mu\in(0,1)$ and~$\r_0\in(0,1]$.
     There exists~$v_0\in\N$ such that for every integer~$v\geq v_0$
and every sequence~$D \in
\N^{\N_0}$, writing \smash{$r=2\lfloor\sqrt{2 v/\mu}\rfloor$}, for
every~$n \geq 1$, if~$A \subset \Z_n^2$ and~$(A_j , \cC_j )_{j\leq L}$
is a~$(v,D)$-dormitory hierarchy on~$A$ such that every set~$C \in
\cC_0$ is~$r$-connected, then there
exists a collection of~$C$-toppling procedures~$p_C$ for
each~$C\in\cC_0$ such that, for any fixed~$\mathbf{j}\in\Sigma$, with the toppling strategy~$f$ and the
sleep mask~$g^\mathbf{j}$ defined in Section~\ref{section-def-f-g}, for
every~$C\in\cC_0$ and~$\rho\in[\rho_0,1]$, starting from a random initial configuration~$U_0$ with distribution~$\cP_{C,\rho}$, the variable~$1+H^*(C,\,f,\,g^\mathbf{j},\,U_0)$ dominates a geometric variable with parameter \smash{$\exp\big(-|C|/\sqrt{v}\big)$}. 
\end{lemma}

This lemma is an analog of Lemma 3.4 of \cite{AFG}, and our proof follows the same strategy, but with some differences.
In particular, we start with a random initial configuration instead of starting fully active, and after each loop, the particle falls asleep.
Also, the stochastic domination concerns the total number of steps
performed by the distinguished vertex, whereas in~\cite{AFG} we needed
an estimate on the number of loops of color~$0$.

Note that, in fact, for~$C\in\cC_0$ the variable~$H^*(C,\,f,\,g^J,\,U_0)$ is independent of~$J$, because the sleep mask that we defined has no influence on what happens inside a cluster at the first level of the hierarchy during the stabilization of this cluster.
This is why Lemma~\ref{lemma3.4} is formulated in a ``quenched'' way, for any fixed realization~$\mathbf{j}=(j_t)_{t\geq 0}\in\Sigma$ of the random sequence~$J$.

\subsubsection{The metastability phenomenon}

We first show the following result, which is the analog of Lemma 6.2 in \cite{AFG}.
The difference with respect to~\cite{AFG} lies in the fact that, after
each loop, the particle which performed the loop falls asleep.
This results in a slight change in the
assumption~\eqref{assumption-drift} below, but is in fact harmless.
For completeness, we include nevertheless the proof.
As in~\cite{AFG}, we use the convention
that~$\Upsilon_d(\infty)=\lim_{r\to\infty}\Upsilon_d(r)$ (we also
use this lemma with~$r=\infty$ to deal with the transient case).

\begin{lemma}\label{lemma6.2}
 Let~$d\geq 1,\,\lambda>0,\,\alpha>0,\,\beta\in(0,1),\,v\in\N,\,r\in\N\cup\{\infty\},\,D\in\N^{\N_0}$ and assume that
 \begin{equation}
\label{assumption-drift}
 (1+\lambda)(1-e^{-\alpha})
 \leq
 \Upsilon_d(16vr)
 \big(1-e^{-\alpha(1-\beta)v}\big)
 \,.
 \end{equation}
 Let~$n\in\N$ and~$A \subset \Z_n^d$, let~$(A_j , \cC_j )_{j\leq L}$
be a~$(v,D)$-dormitory hierarchy on~$A$ such that every set~$C \in
\cC_0$ is~$r$-connected, then there
exists a collection of~$C$-toppling procedures~$p_C$ for
each~$C\in\cC_0$ such that, for any~$\mathbf{j}\in\Sigma$, with the toppling strategy~$f$ and the sleep mask~$g^\mathbf{j}$ defined in Section~\ref{section-def-f-g}, for every~$C\in\cC_0$ and every fixed subset~$U_0\subset C$ with~$|U_0|=\lfloor \beta|C|\rfloor$,
 denoting by~$(U_t)_{t\geq 0}$
the~$\mathrm{ARWD}(\lambda,\,C,\,f,\,g^\mathbf{j},\,U_0)$ process and defining
the stopping time~$\tau=\inf\{t\geq 0\,:\,|U_t|>\beta|C|\}$, it holds
that 
\[
\bbP(\tau=\infty)
\leq
e^{-\alpha(\beta|C|-2)}
\,,
\]
that is to say, with probability at least~$1-e^{-\alpha(\beta|C|-2)}$,
the proportion of active sites in~$C$ exceeds~$\beta$ before
hitting~$0$, i.e., before the
stabilisation of~$C$.
\end{lemma}

\begin{proof}
Let~$d,\,\lambda,\,\alpha,\,\beta,\,v,\,r,\,D,\,n,\,A, \mathbf{j}$
and~$(A_j,\,\cC_j)_{j\leq L}$ be as in the statement.
For every~$C\in\cC_0$, Lemma 6.1 of \cite{AFG} (applied with~$vr$ instead of~$r$) tells us that there exists a function~$\pi_C : \cP(C) \setminus \{\varnothing\}\ \to C$ such that, for every~$U \subset C$ with~$0 < |U| \leq \beta|C|$, we have~$\pi_C(U) \in U$ and 
\begin{equation}\label{eq6.4}
\big|(C \setminus U) \cap B\big(\pi_C(U),\,16vr\big)\big|
\geq
(1-\beta)\,v
\,.
\end{equation}
We turn these functions~$\pi_C$ into~$C$-toppling procedures~$p_C$ by writing, for every~$U\in\cP(C)\setminus\{\varnothing\}$,
\begin{equation*}
  p_C(U)
  =
  \begin{cases}
  x_C^\star
  &\text{if } x_C^\star \in U\,,\\
  
  \pi_C(U)
  &\text{otherwise.}
  \end{cases}
\end{equation*}
Then, we let~$f$ and~$g^\mathbf{j}$ be the toppling strategy and sleep mask derived from these toppling procedures, as explained in Section~\ref{section-def-f-g}.

We now fix a particular cluster~$C\in\cC_0$, we let~$U_0\subset C$ be
such that~$|U_0|=\lfloor \beta|C|\rfloor$, we denote by~$(U_t)_{t\geq
0}$ the~$\mathrm{ARWD}(\lambda,\,C,\,f,\,g^\mathbf{j},\,U_0)$ process and we
let~$\tau$ be the stopping time defined in the statement.
Then, for every~$t\geq 0$, we consider
\begin{equation*}
    M_t=e^{-\alpha N_t}
    \qquad\text{where}\qquad
    N_t=\big|U_t\setminus\{x^\star_C\}\big|
\end{equation*}
and we show that the stopped process~$(M_{t\wedge \tau})_{t\geq 0}$ is a supermartingale with respect to the natural filtration~$(\cF_t)_{t\geq 0}$ of~$(U_t)_{t\geq 0}$.
On the one hand, on the event~$\{t\geq\tau\}\cup\{x_C^\star \in U_t \cap j_t\}$, we have~$M_{(t+1)\wedge\tau}=M_{t\wedge\tau}$, because our toppling procedure~$p_C$ chooses the distinguished particle in priority, and this particle is not allowed to wake up anyone in~$C$, by definition of the sleep mask~$g$.

On the other hand, on the event~$\{t<\tau\}\cap\{x_C^\star\notin U_t \cap j_t\}$, recalling that at each step the chosen particle falls asleep with probability~$\lambda/(1+\lambda)$ and makes a loop otherwise, falling asleep at the end of its loop, and noting that, on this event, our assumption~\eqref{eq6.4} ensures that there are at least~$k=\lceil(1-\beta)v\rceil$ sleeping particles in the ball of radius~$16vr$ around the chosen particle,
we may write, on this event,
\begin{multline*}
\bbE\big[M_{(t+1)\wedge \tau}\,\big|\,\cF_t\big]
=
\bbE\big[M_{t+1}\,\big|\,\cF_t\big]
=
M_t\,
\bbE\Big[e^{-\alpha(N_{t+1}-N_t)}\,\Big|\,\cF_t\Big]
\\
\leq
M_t\,e^\alpha
\bigg(
\frac{\lambda}{1+\lambda}
+\frac{1}{1+\lambda}\,
\bbE\Big[e^{-\alpha(Y_1+\cdots+Y_k)}\Big]
\bigg)
\,,
\numberthis\label{computation-martingale}
\end{multline*}
where the variables~$Y_1,\,\dots,\,Y_k$ are Bernoulli with parameter~$\Upsilon_d(16vr)$, but non necessarily independent.
Then, using Lemma 6.3 of \cite{AFG}, we have
\[
\bbE\Big[e^{-\alpha(Y_1+\cdots+Y_k)}\Big]
\leq
1-\Upsilon_d(16vr)\big(1-e^{-\alpha k}\big)
\leq
1-\Upsilon_d(16vr)\big(1-e^{-\alpha(1-\beta)v}\big)
\,.
\]
Using our assumption~\eqref{assumption-drift}, this becomes
\[
\bbE\big[e^{-\alpha(Y_1+\cdots+Y_k)}\big]
\leq
1-(1+\lambda)(1-e^{-\alpha})
\,.
\]
Plugging this into our computation~\eqref{computation-martingale}, we deduce that, still on the event~$\{t<\tau\}\cap\{x_C^\star\notin U_t\}$,
\[
\bbE\big[M_{(t+1)\wedge \tau}\,\big|\,\cF_t\big]
\leq
M_t\,e^\alpha
\bigg(
\frac{\lambda}{1+\lambda}
+\frac{1}{1+\lambda}
-(1-e^{-\alpha})
\bigg)
=
M_t
=
M_{t\wedge\tau}
\,,
\]
which concludes the proof that~$(M_{t\wedge \tau})_{t\geq 0}$ is a supermartingale.

Hence, using Doob's optional stopping theorem, we deduce that
\[
\bbP(\tau=\infty)
=
\bbP\big(T(C,\,f,\,g^\mathbf{j},\,U_0)<\tau\big)
\leq
\bbE[M_{\tau\wedge T(C,\,f,\,g^\mathbf{j},\,U_0)}]
\leq
M_0
\leq
e^{-\alpha(\beta|C|-2)}
\,,
\]
which concludes the proof of the lemma.
\end{proof}

\subsubsection{Consequence on the number of steps performed by the distinguished particle}

We are now ready to show Lemma~\ref{lemma3.4}.

\begin{proof}[Proof of Lemma~\ref{lemma3.4}]
Let~$d=2,\,\lambda>0,\,\mu\in(0,1),\, \r_0\in(0,1]$ and~$\mathbf{j} \in \Sigma$ as in the statement.
Let~$\beta=\r_0/2$ and let~$K$ be the constant given by Lemma~\ref{lemmahittingprob}. 
We claim that, for~$v$ large enough, we have
\begin{equation}
    \label{drift-for-v}
    (1+\lambda)\bigg[1-\exp\bigg(-\frac {4}{\beta\sqrt{v}}\bigg)\bigg]
    \leq
\frac{K}{\ln\big(64v^{3/2}/\sqrt{\mu}\big)}
    \bigg[1-\exp\bigg(-\frac{4(1-\beta)\sqrt{v}}{\beta}\bigg)\bigg]
    \,.
\end{equation}
Indeed, when~$v\to\infty$ the left-hand side of~\eqref{drift-for-v} is equivalent to~$4(1+\lambda)/(\beta\sqrt{v})$, while the right-hand side is equivalent to~$2K/(3\ln v)$.
Hence, we can fix~$v_1\in\N$ such that~\eqref{drift-for-v} holds for all~$v\geq v_1$.

Let also~$v_2\in\N$ be such that, for every~$v\geq v_2$,
\begin{equation}
    \label{inequality-v2}
    8v^2\,e^{-\sqrt{v}}\leq 1
    \qquad\text{and}\qquad
\bigg(1-\frac K{(1+\lambda)\ln\big(2v^{3/2}\sqrt{2/\mu}\big)}\bigg)^{v^2}
\leq
\frac{e^{-v}}{4}
\,.
\end{equation}

We then let~$v_0=v_1\vee v_2\vee \lceil 1/\beta^4\rceil$, we
take~$v\geq v_0,\,D\in\N^{\N_0}$ and \smash{$r=2\lfloor\sqrt{2
v/\mu}\rfloor$} and we let~$n \geq 1$ and we consider a set~$A \subset \Z_n^2$ equipped with a~$(v,D)$-dormitory hierarchy~$(A_j,\cC_j)_{j\leq L}$ such that every set~$C \in \cC_0$ is~$r$-connected.

Then, the inequality~\eqref{drift-for-v} allows us to use Lemma~\ref{lemma6.2} with~$\alpha=4/(\beta\sqrt{v})$.
Let~$(p_C)_{C\in\cC_0}$ be the collection of toppling procedures given by this lemma, and let~$f$ and~$g^{\mathbf{j}}$ be the toppling strategy and sleep mask derived from this collection, as detailed in Section~\ref{section-def-f-g}.

Let~$C\in\cC_0$.
As in \cite{AFG} we consider the set
\[
\cB
=
\big\{U\subset C\,:\,|U|>\beta|C|\big\}
\]
and its so-called boundary, defined as
\[
\partial\cB
=
\Big\{U\subset C\setminus\{x^\star_C\}\,:\,|U|=\big\lfloor\beta|C|\big\rfloor\Big\}
\,.
\]
Let us fix a deterministic initial configuration~$U_0\in \partial\cB$, and denote by~$(U_t)_{t\geq 0}$ the~$\mathrm{ARWD}(\lambda,\,C,\,f,\,g^\mathbf{j},\,U_0)$ process.
By Lemma~\ref{lemma6.2}, starting from this configuration~$U_0$, the
probability that the process stabilises without ever visiting~$\cB$
is~$\bbP(\tau=\infty)
\leq
e^{-\alpha(\beta|C|-2)}$.
Using that~$\beta|C|\geq\beta v\geq
1/\beta^3\geq 8\geq 4$, we may deduce that
\begin{equation}
\label{proba-tau-infty}
\bbP(\tau=\infty)
\leq
e^{-\alpha\beta|C|/2}
\leq
e^{-2|C|/\sqrt{v}}
\,,
\end{equation}
Furthermore, note that, since a~$C$-toppling procedure only depends on the current configuration, the process~$(U_t)_{t\geq 0}$ is in fact a Markov chain.
Thus, since~\eqref{proba-tau-infty} holds for all~$U_0\in\partial\cB$, denoting by~$\cN_\cB$ the number of times that the process enters~$\cB$ (from a configuration outside this set) before the complete stabilisation of~$C$, the variable~$1+\cN_{\cB}$ dominates a geometric variable with parameter~$e^{-2|C|/\sqrt{v}}$.

Still starting from a fixed~$U_0\in \partial\cB$, let us introduce two random times: first,
\begin{equation*}
    T_\star=\inf\big\{t \in \N_0\,:\,x_C^\star \in U_t \big\}
\end{equation*}
recalling that~$x_C^\star \notin U_0$ by definition of~$\partial \cB$ and, for~$t \in \N_0$, 
\begin{equation*}
    T^t_{\cB} =\inf\big\{t'\geq t\,:\,U_{t'} \in \cB\big\}
    \,.
\end{equation*}
We also define a deterministic time~$t_0=|C|^2$, and we consider the event
\begin{equation*}
    \cE
    =\big\{T_\star\leq t_0\big\}
    \cap 
    \big\{T_{\cB}^{T_\star}<\infty\big\}
    \,.
\end{equation*}
We then control the probability of this event~$\cE$ as in Section 6.3 of \cite{AFG}, obtaining
\begin{align*}
\bbP\big(\cE^c\big)
&\leq
2t_0\,e^{-2|C|/\sqrt{v}}
+
\bigg(1-\frac{\Upsilon_2\big(r|C|\big)}{1+\lambda}\bigg)^{t_0}
\\
&\leq
2|C|^2\,e^{-2|C|/\sqrt{v}}
+
\bigg(1-\frac{K}{(1+\lambda)\ln\big(2|C|\sqrt{2v/\mu}\big)}\bigg)^{|C|^2}
\\
&\leq
2|C|^2\,e^{-\sqrt{|C|}}\,
e^{-|C|/\sqrt{v}}
+
\bigg(1-\frac{K}{(1+\lambda)\ln\big(2|C|^{3/2}\sqrt{2/\mu}\big)}\bigg)^{|C|^2}
\,,
\end{align*}
where in the last inequality we used that~$|C|\geq v$.
Then, since~$|C|\geq v\geq v_2$ we may use the two inequalities in~\eqref{inequality-v2} to get
\[
\bbP\big(\cE^c\big)
\leq
\frac {e^{-|C|/\sqrt{v}}} 4
+
\frac {e^{-|C|}} 4
\leq
\frac {e^{-|C|/\sqrt{v}}} 2
\,.
\]

If this event~$\cE$ is realized, it means that, starting
from~$U_0 \in \partial \cB$, the distinguished particle is awakened within time~$t_0$, after which we reach the set~$\cB$.
As explained above, when we exit~$\cB$, we necessarily end up on a configuration in~$\partial \cB$.

Thus, since the bound on the probability of~$\cE^c$ is uniform over all~$U_0\in\partial\cB$, we deduce that, starting from any configuration in~$\partial\cB$, one plus the number of times that the distinguished vertex is waken up dominates a geometric variable with parameter~$(e^{-|C|/\sqrt{v}})/2$.

Since each time that the distinguished vertex is waken up, it is toppled at the next step, we deduce that, still starting from any fixed configuration~$U_0\in\partial\cB$, the variable~$1+H^\star(C,\,f,\,g^\mathbf{j},\,U_0)$ dominates a geometric variable with parameter~$(e^{-|C|/\sqrt{v}})/2$.

There now remains to go back to the case of a random initial configuration~$U_0$ with distribution~$\cP_{C,\rho}$, that is to say, with i.i.d.\ particles on~$C$, each unstable with probability~$\rho$, where~$\rho\geq \rho_0$.
In this setting, Hoeffding's inequality entails that
\[
\cP_{C,\rho}
\big(U_0\notin\cB\big)
=
\cP_{C,\rho}
\bigg(|U_0|\leq\frac{\rho}{2}\,|C|\bigg)
\leq
e^{-\rho^2|C|/2}
\leq
e^{-2|C|/\sqrt{v}}
\leq
\frac{e^{-|C|/\sqrt{v}}}{2}
\,,
\]
where in the second inequality we used our assumption that~$v\geq
1/\beta^4\geq 16/\rho^4$ and in the last inequality we used that~$\sqrt{v}\geq\ln 2$.

Now recall that if~$U_0\in\cB$, then the process will visit at some time the boundary set~$\partial\cB$.
Thus, what we showed above implies that, conditioned on the event that~$U_0\in\cB$, the variable~$1+H^\star(C,\,f,\,g^\mathbf{j},\,U_0)$ dominates a geometric variable with parameter~$(e^{-|C|/\sqrt{v}})/2$.

Therefore, without this conditioning, writing
\[
H^\star(C,\,f,\,g^\mathbf{j},\,U_0)
\geq
\mathbbm{1}_{\{U_0\in\cB\}}
H^\star(C,\,f,\, g^\mathbf{j},\,U_0)
\]
and using that if~$X$ is a Bernoulli variable with parameter~$1-a$
and~$G$ is a geometric variable with parameter~$b$, independent
of~$X$, then~$1+X(G-1)$ dominates a geometric
variable with parameter~$a+b-ab\leq a+b$, we conclude that the variable~$1+H^\star(C,\,f,\,g^\mathbf{j},\,U_0)$ dominates a geometric with parameter
\[
\cP_{C,\rho}
\big(U_0\notin\cB\big)
+\frac{e^{-|C|/\sqrt{v}}} 2
\leq
e^{-|C|/\sqrt{v}}
\,,
\]
which concludes the proof.
\end{proof}

\subsection{The induction step: two clusters playing ping-pong}

We now go back to considering a random sequence~$J$ and no longer reason with a fixed realisation~$\mathbf{j}$ of this variable.
Recall that for every cluster~$C$ of our hierarchy we denote by~$H^\star(C,\,f,\,g^J,\,U_0)$ the number of steps performed by the distinguished vertex~$x^\star_C$ during the stabilisation of~$C$ starting from an initial configuration~$U_0\subset C$.

\textbf{The induction hypothesis~$\cP(j)$:}  
Let~$\r_0\in(0,1]$ and let~$(\alpha_j)_{j\in \N_0}$ be a sequence of positive real numbers, to be
 chosen later.
 Our induction hypothesis, written~$\cP(j)$, is the following:
 for every~$j \in \{0,\,\dots,\,J\}$, we define
 \begin{equation}
 \label{induction-hyp}
 \cP(j)\ :\qquad
     \forall \,C \in \cC_j 
     \quad
     \forall\rho\in[\rho_0,1]
     \qquad
     U_0\sim\cP_{C,\,\r}
     \quad\Longrightarrow\quad
     1+H^\star(C,\,f,\,g^J,\,U_0)\succeq \geom\big(e^{-\alpha_j|C|}\big)
     \,.
 \end{equation}

Note that, the ARWD model being a priori non-abelian, we have no monotonicity property with respect to the initial distribution at our disposal.
This is why the induction hypothesis is stated for every~$\rho\geq\rho_0$ instead of just taking~$U_0$ with distribution~$\cP_{C,\,\r_0}$.
In practice, at each level of the hierarchy, we will only use the cases~$\rho=\rho_0$ (during the first ping-pong round) and~$\rho=1$ (for the successive ping-pong rallies).

Note that Lemma~\ref{lemma3.4} showed~$\cP(0)$, under certain hypotheses.

Here is now the induction step of our proof, which is adapted from Lemma 3.1 of \cite{AFG}.
The constant~$\bar p$ is a universal constant which comes from an elementary property of geometric sums of Bernoulli variables, given by Lemma~\ref{lemmabernoulligeom} below.

\begin{lemma}\label{lemma3.1}
    Let~$\bar p\in(0,1)$ be given by Lemma~\ref{lemmabernoulligeom}.
    Let~$d\geq 1,\,\r_0\in(0,1],\,v \geq 1,\,(D_j)_{j\geq 0} \in
\N^{\N_0}$ and let~$(\alpha_j)_{j \geq 0}$ be a non-increasing
sequence of positive real numbers and~$(p_j)_{j\geq 0}\in(0,1)^{\N_0}$
be such that for every~$j \geq 0$,
    \begin{equation}
        \label{assumption-pj}
        p_j
        \leq
        \bar p\wedge \frac 1 {2^{j+1}(1+\lambda)} \wedge \big(32\alpha_j(D_j+1)^d\big)^{-1/2}
    \end{equation}
    and
\begin{equation}\label{eqassumption3.1}
    2^{j/2+2}v
    \leq
    p_j^{6}\,\Upsilon_d(D_j)
    \,\exp\big((\alpha_j -\alpha_{j+1})2^{j/2}v\big)
    \,.
\end{equation}

   For every~$n \geq 1$ and every~$A \subset \Z^d_n$ equipped with a~$(v, D)$-dormitory hierarchy~$(A_j, \cC_j)_{j\leq L}$ and
 with a collection of toppling procedures~$(p_C)_{C\in \cC_0}$, with the toppling strategy~$f$ and sleep mask~$g^J$ defined in Section~\ref{section-def-f-g}, if the property~$\cP(0)$ holds, then~$\cP(j)$ also holds
 for every~$j \leq L$.
\end{lemma}

The proof being similar to the proof of Lemma 3.1 in \cite{AFG}, but with some important changes, we concentrate on the novelties in our setting and refer to \cite{AFG} for more detailed explanations.

\subsubsection{General setting for the proof of Lemma~\ref{lemma3.1}}

Let~$d,\,\r_0,\,v,\,(D_j)_{j\geq 0},\,(\alpha_j)_{j\geq 0},\,(p_j)_{j\geq 0},\,n,\,A,\,(A_j, \cC_j)_{j\leq L}$,~$(p_C)_{C\in C_0}$,~$f$ and~$g^J$ be as in the statement.
To lighten the notation, from now on we will write~$H^\star(C,\,U_0)$ instead of~$H^\star(C,\,f,\,g^J,\,U_0)$.
We assume that~$0\leq j< L$ is such that the property~$\cP(j)$ holds and we wish to establish the property~$\cP(j+1)$, i.e., that for every~$C \in C_{j+1}$ and~$\rho\in[\rho_0,1]$, if~$U_0\sim\cP_{C,\,\r}$ then
    \begin{equation}
    \label{goal-induction}
1+H^\star(C,\,U_0)\succeq \geom\big(e^{-\alpha_{j+1} |C|}\big)
\,. 
   \end{equation}
We start by distinguishing between two cases. First, if~$C \in C_{j+1} \cap \cC_j$, then~\eqref{goal-induction} directly follows from induction hypothesis~$\cP(j)$, using our assumption that~$\alpha_{j+1}\leq\alpha_j$.

Suppose now that~$C\in \cC_{j+1}\setminus \cC_j$. 
Following property~\itref{conditionMerge} of the dormitory hierarchy,
we have~$\diam \, C \leq  D_j$ and we can write~$C = C_0 \cup C_1$ with~$C_0,\,C_1 \in \cC_j$. To simplify the notation,
we write \smash{$x_0^\star=x_{C_0}^\star$}
and \smash{$x_1^\star=x_{C_1}^\star$}.
Exchanging the names~$C_0$ and~$C_1$ if necessary, we may assume
that~$|C_0| \geq |C_1|$ and~$x_C^\star=x_0^\star$.

To stabilise~$C$ we perform a ping-pong rally, namely, we first stabilise~$C_0$ then stabilise~$C_1$. Then, some sites of~$C_0$ may have been reactivated during the stabilisation of~$C_1$, so we stabilise again~$C_0$, and then again~$C_1$ if necessary, going on until both sets are stabilised.

As explained before, the dynamics contains the following exception, which is encoded into the sleep mask~$g^J$ defined in Section~\ref{section-def-f-g}: the loops emitted during the first stabilisation of~$C_0$ are not allowed to wake up the sites of~$C_1$, so that the first stabilisation of~$C_1$ starts from the initial configuration inside~$C_1$.

As in \cite{AFG}, denoting by~$U_i$ the set of unstable sites of~$C$ before the stabilisation number~$i$ (starting the numbering at~$0$, with~$U_0$ the random configuration defined above), we define
\begin{equation*}
\cN
=
\inf\big\{\,i\geq 1\ :\ 
C_0\not\subset  U_{2i}
\quad\text{or}\quad
C_1\not\subset U_{2i+1}
\,\big\}
\,,
\end{equation*}
which indicates the first stabilisation which does not wake up all the
sites of the other cluster (with the exception that we do not require
anything on the first stabilisation of~$C_0$, which in any case is not
allowed to wake up anyone in~$C_1$).
Our notation~$U_i$ corresponds to the notation~$R_i$ in \cite{AFG}.

\subsubsection{Key differences with respect to \cite{AFG}}

We now explain the main differences with respect to the proof of Lemma 3.1 in \cite{AFG}.

The first difference is that, while in \cite{AFG} we were interested in the stabilisation time starting with all the particles active (i.e.,~$U_0=C$), here we start instead with an i.i.d.\ initial configuration, where each site of~$A$ is unstable with probability~$\r$.
This difference comes from the fact that, in ARW, we may let the
particles walk until each site of~$A$ contains one particle, and then
obtain the configuration with one active particle sitting on each site of~$A$, whereas in ARWD, due to the fact that a particle falls asleep at the end of its loop, the configuration reached once all the sites of~$A$ are occupied is not necessarily fully active.

Another difference is that, during the first stabilisation of~$C_0$, the configuration in~$C_1$ is frozen with no reactivation possible, due to the exception inserted in the definition of the sleep mask~$g$ (see Section~\ref{section-def-f-g}).
Thus, in our setting the first stabilisation of~$C_0$ and the first
stabilisation of~$C_1$ both start with an i.i.d.\ configuration with parameter~$\r$.
Then, as in \cite{AFG}, each of the subsequent stabilisation rounds of~$C_0$ or~$C_1$, until the~$\cN$-th round, starts with the corresponding set fully active.

As a consequence, the variable~$\cN$ is no longer exactly geometric.
Nevertheless, this difference is harmless.
Indeed, we will show that~$\cN$ dominates a geometric variable, using
that at each stabilisation round we can use the induction hypothesis
(with initial fraction of active particles~$\rho$ for the first round
and~$1$ for the successive rounds).

Another crucial difference is that, in our setting, we do not have an equivalent of Lemma 2.3 of \cite{AFG}.
This lemma states that, in the ARW case, knowing the total number~$\cT$ of sleeps and loops of colour~$0,\,\dots,\,j-1$ produced by the distinguished site of a cluster~$C\in \cC_j$ during the stabilisation of this cluster (starting from all the sites active), then the number of loops of colour~$j$ produced by~$x_C^\star$ is distributed as a sum of~$\cT$ geometric variables (minus one) with an explicit parameter, independent of~$\cT$.
This comes from the key observation that, during the stabilisation of
a cluster~$C\in\cC_j$, the loops of colour at least~$j$ have no effect
on~$\cT$, because these loops are not allowed to wake up anyone
inside~$C$.
Thus, in ARW, to determine~$\cT$ one does not care about how many loops of colour~$j$ are inserted between any two loops of colour less than~$j$.

In \cite{AFG}, the induction hypothesis was formulated in terms of the number of loops of colour~$j$ produced by the distinguished site of a cluster~$C\in\cC_j$.
Thus, we used this lemma to go from one colour to the next one, translating an information about the loops of colour~$j$ into an information about the loops of colour~$j+1$.

Now recall that ARWD differs from ARW in that, after each loop, the particle is forced to fall asleep.
Thus, all the loops have an effect on the stabilisation of~$C\in\cC_j$ because, even if a loop is not allowed to wake up anyone, it has the effect of making the distinguished particle fall asleep.
This dependency makes it more difficult to translate a statement on the number of loops of colour~$j$ into a statement of colour~$j+1$.

To circumvent this, our induction hypothesis is simply formulated with the total number of steps and, at each step, we translate the lower bound on  the total number of steps into a lower bound on the loops of colour~$j$.
Knowing that~$H^\star(C,\,U_0)=h$, the number of loops of colour~$j$ is a sum of~$h$ independent Bernoulli variables with parameter~$1/((1+\lambda)2^{j+1})$, but these Bernoulli variables are not independent of~$H^\star(C,\,U_0)$.
Yet, loosing on the parameter of the geometric variables, we are able to obtain a lower bound, using Lemma~\ref{lemmabernoulligeom} below, whose elementary proof is deferred to Section~\ref{section-proof-lemmabernoulligeom}.

\begin{lemma}
\label{lemmabernoulligeom}
There exists~$\bar p\in(0,\,1)$ such that, for every~$p\in(0,\,\bar p)$, for
every~$\varepsilon$ such
that
\begin{equation}
\label{condition_epsilon}
\exp\left(-\frac 1 {32p^2}\right)
\leq\varepsilon
<p^3
\,,
\end{equation}
if~$T+1$ is a
geometric variable with parameter~$\varepsilon$
and~$(X_n)_{n\geq 1}$ are i.i.d.\ Bernoulli variables with
parameter~$p$, not necessarily independent of~$T$, then the
variable~$1+X_1+\cdots+X_T$ stochastically dominates a geometric
variable with parameter~$\varepsilon/p^3$.
\end{lemma}

Then, to obtain the induction hypothesis at rank~$j+1$ we simply bound from below the total number of steps by the number of loops of colour~$j$.

\subsubsection{The loop representation}

As explained above, we want to study the number of loops of colour~$j$ produced by the distinguished vertex during the successive rounds of the ping-pong rally.
Yet, as ARWD is defined, the number of loops of colour~$j$ cannot be expressed as a function of the ARWD process.
Thus, we use the loop representation presented in~\cite{AFG}, adapting it to our setting, which enables us to study these coloured loops.

This representation is based on independent random variables
\begin{equation*}
    \big(I(x,h)\big)_{x \in A, h \in \N_0} \in \{0,1\}^{A \times \N_0},
    \quad
    J=(J_t)_{t \in \N_0} \in \N^{\N_0}
    \quad\text{and}\quad
    \big(R(x,\ell, j)\big)_{x \in A,\ell \in \N_0, j \in \N_0} \in \cP(A)^{A \times \N_0^2},
 \end{equation*}
 where the variables~$I(x,h)$ are Bernoulli with parameter~$\lambda/(1+\lambda)$, the variables~$1+J_t$ are geometric with parameter~$1/2$ and~$R(x, \ell, j)$ is distributed as the support of a symmetric random
 walk on the torus started and killed at~$x$.
 Probabilities and expectations are simply denoted by~$\bbP$ and~$\bbE$, which depend implicitly on the set~$A$.

 We now describe the update rules of the model. Recall that a
configuration is a subset~$U \subset A$ indicating
which sites are active.  
 To update the configuration we need to recall the numbers of visits
and loops already used at each vertex.
This is the role of what we call the odometer function~$m :  A \to \N_0$ and the loop odometer function~$\ell : A \to \N_0$.

 Let~$k\in\{0,1\}$ and let us describe the procedure to stabilise the cluster~$C_k$ using these variables.
Let~$f_{C_k}$ be the toppling strategy associated with this cluster~$C_k$ and let~$g^J$ be the sleep mask, both defined in Section~\ref{section-def-f-g}.

Given a fixed configuration~$U_0 \subset A$, an odometer~$m_0 : A \to \N_0$ and a loop odometer~$\ell_0: A \to \N_0$, we define a process~$(U_t,\,m_t,\,\ell_t)_{t\geq 0}$ as follows.
  Let $t\geq 0$ and assume that $(U_s,\,h_s,\,\ell_s)$ is constructed for every~$s\leq t$.
  If $U_t\cap C_k=\varnothing$ then we define~$(U_{t+1},\,m_{t+1},\,\ell_{t+1})=(U_t,\,m_t,\,\ell_t)$.
  Otherwise, if $U_t\cap C_k\neq\varnothing$ then, writing \smash{$x_t=f_{C_k}( U_0,\,\dots,\,U_t)$}, $i_t=I(x,\,m_t(x))$, $j_t=J_t$, $R_t=R(x,\,\ell_t(x),j_t)$ and $W_t=g^J(U_0,\,\dots,\,U_t)$, we define the new odometer $m_{t+1}=m_t+\mathbbm{1}_x$, the new loop odometer~$\ell_{t+1}=\ell_t+\mathbbm{1}_{\{i_t=0\}}\mathbbm{1}_x$, and the next configuration

  \begin{equation*}
       U_{t+1}
      =
      \begin{cases}
           U_t\setminus\{x_t\}
          &\text{if~$i=1$,}
          \\
          (U_t\cup (R_t\cap W_t))\setminus \{x_t\}
          &\text{otherwise.}
      \end{cases}
  \end{equation*}
Then, we define the stabilisation operator
 \begin{equation*}
     \text{Stab}_{k}:\begin{cases}
         \cP(A) \times (\N_0^{A})^2 
         &\ \longrightarrow\ \cP(A) \times (\N_0^{A})^2\\
         (U_0,\,m_0,\,\ell_0) 
         &\ \longmapsto\ ( U_\tau,\,m_\tau,\,\ell_\tau)\,,
     \end{cases}
 \end{equation*}
where
\begin{equation*}
    \tau=\inf\{t\in \N_0:  U_t \cap C_k=\varnothing\}\,,
\end{equation*}
which is almost surely finite (we restrict ourselves to the event on
which it is always finite).
This notation allows us to define the number of loops of colour~$j$ produced by the distinguished vertex of~$C_k$ during the stabilisation of~$C_k$, which is
 \begin{equation*}
     \cL(C_k,\,U_0,\,j)=\big|\big\{t<\tau\,:\,x_t=x^\star_{C_k},\, j_t=j\big\}\big|
     \,.
 \end{equation*}

\subsubsection{Translating the induction hypothesis}

Going back to our cluster~$C=C_0\cup C_1$ with~$C_0,\,C_1\in\cC_j$, our induction hypothesis~$\cP(j)$ tells us that for~$k \in \{0, 1\}$, for any random subset~$U_0\sim\cP_{C_k,\,\r}$ with~$\rho\geq\rho_0$, the variable~$1 + H^\star(C_k,U_0)$ dominates a geometric random variable with parameter~$\exp(-\alpha_j |C_k|)$.

We wish to translate this information in terms of the number of loops of colour~$j$.
Recall that each step made by the distinguished particle is a loop of colour~$j$ with probability~$1/((1+\lambda)2^{j+1})$.
Hence, we may write
 \begin{equation*}
\cL(C_k,\,U_0,\,j)
\ =\ 
\sum_{i=1}^{H^\star(C_k,\,U_0)} X_i
\ \succeq\ 
\sum_{i=1}^{H^\star(C_k,\,U_0)} X'_i,
\end{equation*}
where 
$(X_i)_{i\geq 1}$ are i.i.d.\ Bernoulli random variables with parameter~$1/((1+\lambda)2^{j+1})$ and~$(X'_i)_{i\geq 1}$ are i.i.d.\ Bernoulli random variables with parameter~$p_j$ (recall that~$p_j\leq 1/((1+\lambda)2^{j+1})$),
but not necessarily independent of~$H^\star(C_k,\,U_0)$.

We now check that the conditions required to apply Lemma~\ref{lemmabernoulligeom} are satisfied. 
On the one hand, since~$C_k$ is included into~$C$ which has diameter
at most~$D_j$, we have~$|C_k|\leq (D_j+1)^d$.
Thus, using our assumption~\eqref{assumption-pj}, we have
\begin{equation}
\label{condition1-lemmabernouilligeom}
\alpha_{j}|C_k|
\leq
\alpha_j(D_j+1)^d
\leq
\frac{1}{32p_j^2}
\,.
\end{equation}
On the other hand, it follows from our assumption~\eqref{eqassumption3.1} that
\begin{equation}
\label{condition2-lemmabernouilligeom}
\exp\big(-\alpha_{j}|C_k|\big)
\leq  
\exp\big(-\alpha_j 2^{j/2}v\big)
\leq
\frac{p_j^6}{4}
\leq
\frac{p_j^3}{2}
\,.
\end{equation}
These two bounds~\eqref{condition1-lemmabernouilligeom}
and~\eqref{condition2-lemmabernouilligeom} allow us to apply
Lemma~\ref{lemmabernoulligeom} with~$p_j$
and~$\e=e^{-\alpha_{j}|C_k|}$ to deduce that, for an initial
configuration~$U_0\sim\cP_{C_k,\,\r}$, the variable~$1+\cL(C_k,U_0,j)$ dominates a geometric variable with parameter  
 \begin{equation}
     \label{def-q}
q_k= \frac 1{p_j^3}
\exp\big(-\alpha_{j}|C_k|\big)
\,.
 \end{equation}

\subsubsection{The ping-pong rally and the infinite Sisyphus sequence}

For every~$i\in\N_0$, we write~$\e(i)=0$ if~$i$ is even and~$\e(i)=1$ if~$i$ is odd.
Then, starting with~$U_0$ distributed as~$\cP_{C,\,\r}$
and~$m_0=\ell_0=0$, we define, for every~$i\geq 0$,
\[
(U_{i+1},\,m_{i+1},\,\ell_{i+1})
=
\mathrm{Stab}_{\varepsilon(i)}(U_i,\,m_i,\,\ell_i)
\,.
\]
For every~$i\geq 0$, we denote by~$\cL_i$ the number of loops of colour~$j$ performed by the distinguished site of~$C_{\e(i)}$ during the~$i$-th stabilisation (the first stabilisation of~$C_0$ having number~$0$).
We denote by~$\cL(C,\,U_0,\,j)$ the total number of loops of colour~$j$ made by~$x_C^\star$ during the whole ping-pong-rally.

Besides, we also consider another sequence which is built by instead assuming that each stabilisation (except the first stabilisation of~$C_0$) wakes up all the sites of the other cluster.
Namely, for~$i\leq 2$ we let~$(U'_i,\,h'_i,\,\ell'_i)=(U_i,\,m_i,\,\ell_i)$, and for~$i\geq 2$ we let
\[
(U'_{i+1},\,h'_{i+1},\,\ell'_{i+1})
=
\mathrm{Stab}_{\varepsilon(i)}(C_{\varepsilon(i)},\,m_i,\,\ell_i)
\,.
\]
We define~$\cL'_i$ similarly with this new sequence and, as in \cite{AFG} we note that~$\cL'_i=\cL_i$ for every~$i<2\cN$, allowing us to write
\begin{equation*}
H^\star(C,\,U_0)
\geq
\cL(C,\,U_0,\,j)
\geq
\sum_{i=0}^{\cN-1}
\cL_{2i}
=\sum_{i=0}^{\cN-1}
\cL'_{2i}\,.
\end{equation*}
Note that the sequence~$(\cL'_{i})_{i\geq 0}$ is independent, with~$\cL'_0$ distributed as~$\cL(C_0,\,U_0,\,j)$,~$\cL'_1$ distributed as \smash{$\cL(C_1,\,U_0,\,j)$} and, for every~$i\geq 2$,~$\cL'_i$ distributed as~$\cL(C_{\e(i)},\,C,\,j)$.
Thus, as explained above, for every~$i\geq 0$, the variable~$1+\cL'_i$ dominates a geometric variable with parameter~$q_{\e(i)}$ given by~\eqref{def-q}.
And, as in \cite{AFG}, we have the key property that this sequence is independent of \smash{$\big(R(x^\star_k,\,\ell,\,j)\big)_{k\in\{0,1\},\,\ell\geq 0}$}, i.e., the sets visited by the loops of colour~$j$ emitted by the two distinguished vertices.
Note that this independence property would not hold if we had not
introduced the exception that during the first stabilisation of~$C_0$,
the distinguished particle of~$C_0$ is not allowed to wake up anyone
in~$C_1$: because then the first stabilisation of~$C_1$ would start
with a configuration which depends on the first stabilisation
of~$C_0$.
Note also that we use here the assumption that the initial
configuration~$U_0$ is i.i.d., so that~$U_0\cap C_0$ and~$U_0\cap C_1$
are independent.

\subsubsection{Resulting stochastic domination}

Using this independence property and proceeding as in \cite{AFG}, we have a positive correlation inequality which allows us to write, for every~$m\geq 1$,
\[
\bbP\big(\,
H^\star(C,\,U_0)
\,\geq\,m\,\big)
\geq
\sum_{k=1}^{+\infty}\,
\bbP\Bigg(\,\Bigg\{\sum_{i=0}^{k-1}\cL'_{2i}\,\geq\,m\Bigg\}\cap\big\{\cN=k\big\}\,\Bigg)
\geq
\sum_{k=1}^{+\infty}\,
\bbP\Bigg(\,\sum_{i=0}^{k-1}\cL'_{2i}\,\geq\,m\,\Bigg)
\,\bbP\big(\cN= k\big)
\,.
\]
Thus, taking~$\cN'$ a copy of~$\cN$ which is independent of~$\cL'$, we
get the stochastic domination
\[
H^\star(C,\,U_0)
\ \succeq\ 
\sum_{i=0}^{\cN'-1}\cL'_{2i}
\,.
\]
Then, as explained above, this variable~$\cN'$ is not exactly a geometric variable, but it dominates a geometric variable with parameter~$1-\pi_0 \pi_1$, with
\[
\pi_0
=\bbE\big[\psi_0(\cL'_0)\big]
\wedge
\bbE\big[\psi_0(\cL'_2)\big]
\qquad\text{and}\qquad
\pi_1
=\bbE\big[\psi_1(\cL'_1)\big]
\wedge
\bbE\big[\psi_1(\cL'_3)\big]
\,,
\]
where the function~$\psi_0$ is defined by
\[
\psi_0\,:\,x\in\N\ \longmapsto\ \bbE\Bigg(\,
C_1\subset
\bigcup_{0\,\leq\,\ell\,<\,x}
R\big(x^\star_0,\,\ell,\,j\big)
\,\Bigg)
\]
and~$\psi_1$ is defined similarly, replacing~$C_1$ with~$C_0$ and~$x^\star_0$ with~$x^\star_1$.
Now, we use the following elementary result (Lemma 2.4 in \cite{AFG}):

\begin{lemma}
Let~$N$ be a geometric random variable with parameter~$a\in(0,1)$, and let~$(X_n)_{n\in\N}$ be i.i.d.\ geometric variables with parameter~$b\in(0,1)$, independent of~$N$.
Then, the variable
\[
S=1+\sum_{n=1}^N \big(X_n-1\big)\qquad
\text{is geometric with parameter}\qquad
\frac{ab}{1-b+ab}\,.\]
\end{lemma}

Using this, we deduce that~$1+H^\star(C,\,U_0)$ dominates a geometric variable with parameter
\[
q'
=
\frac{(1-\pi_0\pi_1)q_0}{1-\pi_0\pi_1q_0}
\,.
\]
Then, the same computations as in the end of Section 7 of \cite{AFG} yield
\begin{equation}
\label{computation-qprime}
q'
\leq 
\frac{|C|}{(1-q_0)\Upsilon_d(D_j)}\,q_0q_1
=
\frac{|C|\exp\big(-\alpha_j|C|\big)}
{(1-q_0)\Upsilon_d(D_j)\,p_j^{6}}
\leq
\frac{2\,|C|\exp\big(-\alpha_j|C|\big)}
{\Upsilon_d(D_j)\,p_j^{6}}
\,,
\end{equation}
using the expression~\eqref{def-q} of~$q_0$ and~$q_1$ and then the
crude bound~$q_0\leq 1/2$ which follows from~\eqref{condition2-lemmabernouilligeom}.
Then, using our assumption~\eqref{eqassumption3.1}, we have
\[
1
\leq
(\alpha_j-\alpha_{j+1})
2^{j/2}v
\leq
(\alpha_j-\alpha_{j+1})
|C|
\,,
\]
which, using that~$x\mapsto x e^{-x}$ decreases on~$[1,\,\infty)$, implies that
\[
|C|\,\exp\big(-(\alpha_j-\alpha_{j+1})|C|\big)
\leq
2^{j/2}v\,\exp\big(-(\alpha_j-\alpha_{j+1})2^{j/2}v\big)
\,.
\]
Plugging this into~\eqref{computation-qprime}, we obtain
\[
q'\leq
\frac{2^{j/2+1}v\exp\big(-(\alpha_j-\alpha_{j+1})2^{j/2}v\big)}
{\Upsilon_d(D_j)\,p_j^6}
\,\exp\big(-\alpha_{j+1}|C|\big)
\leq 
\exp\big(-\alpha_{j+1}|C|\big)
\]
where in the last inequality we used again our assumption~\eqref{eqassumption3.1}.
This concludes the proof of Lemma~\ref{lemma3.1}.\hfill\qed

\subsection{Proof of Lemma~\ref{lemmabernoulligeom}}
\label{section-proof-lemmabernoulligeom}

\begin{proof}[Proof of Lemma~\ref{lemmabernoulligeom}]
Let us consider the function
\[
f\,:\,p\in(0,\,1)\ \longmapsto\ \ln p-\frac{\ln(1-p+p^2)}{8p^3}
\,.
\]
When~$p\to 0$, we have~$f(p)\sim 1/(8p^2)$.
Hence, we can find~$\bar p\in(0,\,1)$ such that~$f(p)\geq 1/(16p^2)$ for
every~$p\in(0,\,\bar p)$.
Upon decreasing~$\bar p$ if necessary, we can assume that~$\bar p\leq
1/8$.

Let now~$p\in(0,\,\bar p)$ and~$\varepsilon$
satisfying~(\ref{condition_epsilon}), and let~$T$
and~$(X_n)_{n\geq 1}$ be as in the statement.
For every~$n\geq 1$, let us write~$S_n=X_1+\cdots+X_n$.
To prove the result, we have to show that
\begin{equation}
\label{domination_geom_epsilon_p}
\forall k\geq 1\,,\qquad
\mathbb{P}\big(S_T\geq k\big)
\geq
\left(1-\frac\varepsilon{p^3}\right)^k
\,.
\end{equation}
Let~$k\geq 1$.
Defining~$n=\lceil k/(8p^3)\rceil$, we simply write
\begin{equation}
\label{dumb_bound_ST}
\mathbb{P}\big(S_T\geq k\big)
\geq
\mathbb{P}\Big(\big\{T\geq n\big\}\cap\big\{S_n\geqslant k\big\}\Big)
\geq
\mathbb{P}\big(T\geq n\big)
\,-\,
\mathbb{P}\big(S_n<k\big)
\,.
\end{equation}
We then bound separately these two terms.

On the one hand, we have
\[
\mathbb{P}\big(T\geq n\big)
=
(1-\varepsilon)^{n}
=
\exp\left(\left\lceil\frac{k}{8p^3}\right\rceil\ln(1-\varepsilon)\right)
\,.
\]
First, note that
\[
\left\lceil\frac k{8p^3}\right\rceil
\leq
\frac k{8p^3}+1
\leq
\frac k{4p^3}
\,,
\]
using that~$p\leq 1/2$.
Besides, since~$\varepsilon\leq p^3\leqslant
1/2$, we have, by concavity of the logarithm function,
\[
\ln(1-\varepsilon)
\geq
2\varepsilon\,\ln\left(\frac 1 2\right)
\geq
-2\varepsilon
\,.
\]
Hence, we obtain
\begin{multline*}
\mathbb{P}\big(T\geq n\big)
\geq
\exp\left(-\frac{k\varepsilon}{2p^3}\right)
=
\exp\left(\frac{k\varepsilon}{2p^3}\right)\times
\exp\left(-\frac{k\varepsilon}{p^3}\right)
\\
\geq
\left(1+\frac{k\varepsilon}{2p^3}\right)
\exp\left(-\frac{k\varepsilon}{p^3}\right)
\geq
(1+\varepsilon)
\exp\left(-\frac{k\varepsilon}{p^3}\right)
\,.
\numberthis\label{dumb_bound_part1}
\end{multline*}

On the other hand, it follows from Chebychev's inequality that
\begin{multline*}
\mathbb{P}\big(S_n<k\big)
\leq
\mathbb{E}\big[p^{S_n}\big]p^{-k}
=
(1-p+p^2)^np^{-k}
\leq
\exp\big(-kf(p)\big)
\\
\leq
\exp\left(-\frac k{16p^2}\right)
\leq
\exp\left(-\frac 1{32p^2}\right)
\exp\left(-\frac k{32p^2}\right)
\leq
\varepsilon\,
\exp\left(-\frac{k\varepsilon}{p^3}\right)
\numberthis\label{dumb_bound_part2}
\end{multline*}
where, in the last inequality, we used our hypotheses~$\varepsilon\geq
\exp(-1/(32p^2))$ and~\smash{$\varepsilon\leq p^3\leqslant
 p/32$}.

Now, plugging~(\ref{dumb_bound_part1}) and~(\ref{dumb_bound_part2})
into~(\ref{dumb_bound_ST}), we get
\[
\mathbb{P}\big(S_T\geq k\big)
\geq
\exp\left(-\frac{k\varepsilon}{p^3}\right)
\geq
\left(1-\frac\varepsilon{p^3}\right)^k
\,,
\]
using the inequality~$\ln(1-\varepsilon/p^3)\leq -\varepsilon/p^3$.
This concludes the proof of~(\ref{domination_geom_epsilon_p}), which
is the claim of the Lemma.
\end{proof}

\subsection{Concluding proof of Proposition~\ref{prop-ARWD-stab-time-bis} in two dimensions}

We now put the pieces together to obtain the lower bound on the stabilisation time of ARWD on the two-dimensional torus.

\begin{proof}[Proof of Proposition~\ref{prop-ARWD-stab-time-bis}, case~$d=2$]

Let~$d=2,\,\lambda>0,\,\mu\in(0,1)$ and~$\rho\in(0,1)$.
We start by tuning the various parameters.
Let~$K$ be the constant given by Lemma~\ref{lemmahittingprob}, let~$\bar p$ be the constant given by Lemma~\ref{lemmabernoulligeom} and let
\[
p
=
\bar p\wedge\frac{\sqrt{\mu}}{384}\wedge\frac{1}{2(1+\lambda)}
\,.
\]
Let~$v_0$ be given by Lemma~\ref{lemma3.4} (applied
with~$\rho_0=\rho$), and let~$v_1\in\N$ be such that for all~$v\geq v_1$, we have
\begin{equation}
\label{def-v1}
2v^{10}
\ln\bigg(\frac{1152 v^3}{\mu}\bigg)
\leq
p^6 K\,\exp\bigg(\frac{1-2^{-1/4}}{2}\,v^{1/8}\bigg)
\,.
\end{equation}
Let now~$v=v_0\vee v_1$ and
\smash{$r=2\big\lfloor\sqrt{2v/\mu}\big\rfloor$}.
Then, for~$j \in \N_0$ we define
\begin{equation*}
    \alpha_j
    =
    \frac{1+2^{-j/4}}{2\sqrt{v}}
    \,,\qquad
    D_j
    =6^j\times 12vr
    \qquad\text{and}\qquad
    p_j
    =\frac p{6^{j}v^{3/2}}
    \,.
\end{equation*}
Lastly, we take~$\kappa=\alpha_0\mu/8$.

Let now~$n\geq r+1$ and~$A\subset\Z_n^2$ such that~$|A|\geq\mu n^2$.
Using Lemma~\ref{lemma3.3}, we can consider a~$(v,\,D)$-dormitory hierarchy~$(A_j,\,\cC_j)_{j\leq L}$ on~$A$ such that~$|A_L|\geq|A|/4$ and every set~$C\in\cC_0$ is~$r$-connected.

Then, for every~$C\in\cC_0$ Lemma~\ref{lemma3.4} provides us with a~$C$-toppling procedure~$p_C$ such that the induction hypothesis~$\cP(0)$ defined in~\eqref{induction-hyp} holds, with~$\rho_0=\rho$. Let us consider the variable~$J$ as defined in Section~\ref{sec:reformulated-goal}, along with the toppling strategy~$f$ and the sleep mask~$g^J$
derived from this hierarchy and these toppling procedures, as
explained in Section~\ref{section-def-f-g}.

Let us now check that the assumptions of Lemma~\ref{lemma3.1} are satisfied.
For every~$j\geq 0$ have~$p_j\leq\bar p$ and~$p_j\leq
1/(2^{j+1}(1+\lambda))$ and, moreover, using the definitions of~$p$,
of~$r$ and of~$D_j$, along with the fact
that~$\alpha_j\leq 1$, we can write
\[
p_j
\leq
\frac{\sqrt{\mu}}{384\times 6^{j}v^{3/2}}
\leq
\frac{1}{96\sqrt{2}\times 6^{j}vr}
=
\frac{1}{8\sqrt{2}\,D_j}
\leq
\frac{1}{4\sqrt{2}\,(D_j+1)}
\leq
\frac{1}{\sqrt{32\alpha_j}(D_j+1)}
\,,
\]
which shows that the condition~\eqref{assumption-pj} is satisfied.
To check the condition~\eqref{eqassumption3.1}, we write
\begin{align*}
    \frac{p_j^6\,
    \Upsilon_2(D_j)\exp\big((\alpha_j-\alpha_{j+1})2^{j/2}v\big)}
    {2^{j/2+2}v}
    &\geq
    \frac{p^6K}{6^{6j+j/2+2}v^{10}
    \ln\big(6^j\times 12vr\big)}
    \exp\bigg(\frac{1-2^{-1/4}}{2}\,2^{j/4}\sqrt{v}\bigg)
    \\
    &\geq
    \frac{p^6K}
    {4(4^jv)^{10}
    \big[\ln\big(24\sqrt{2/\mu}\big)+(3/2)\ln(4^jv)\big]}
    \exp\bigg(\frac{1-2^{-1/4}}{2}\,(4^jv)^{1/8}\bigg)
    \geq 1
\end{align*}
following~\eqref{def-v1} applied with~$4^jv$ instead of~$v$.
This shows that the condition~\eqref{eqassumption3.1} is also satisfied, allowing us to apply Lemma~\ref{lemma3.1}, to deduce that~$\cP(j)$ holds for all~$j\leq L$.
In particular we deduce that, starting with an initial
configuration~$U_0\sim\cP(A,\,\rho)$, the variable~$1+H^\star(A_L,\,f,\,g^J,\,U_0\cap A_L)$ dominates a geometric variable with parameter
\[
\exp\big(-\alpha_L|A_L|\big)
\leq
\exp\bigg(-\frac{\alpha_0}{2}\times\frac{|A|}{4}\bigg)
\leq
\exp\big(-\kappa n^2\big)\,.
\]
Then, recalling that our toppling strategy~$f$ chooses sites in~$A_L$ in priority
and that both the toppling strategy and the sleep mask
do not depend on the configuration outside~$A_L$, we
have
\[
T(A,\,f,\,g^J,\,U_0)
\ \geq\ 
T(A_L,\,f,\,g^J,\,U_0\cap A_L)
\ \geq\ 
H^\star(A_L,\,f,\,g^J,\,U_0\cap A_L)
\,,
\]
so that the stochastic domination also
holds for~$T(A,\,f,\,g^J,\,U_0)$.
Decreasing the constant~$\kappa$ if necessary so that the result also
holds when~$n\leq r$, the proof of
Proposition~\ref{prop-ARWD-stab-time-bis} in the two-dimensional case is
complete.
\end{proof}

\subsection{Concluding proof of Proposition~\ref{prop-ARWD-stab-time-bis} in the transient case}

We now study the case of dimension at least~$3$.

\begin{proof}[Proof of Proposition~\ref{prop-ARWD-stab-time-bis}, case~$d\geq 3$]
Let~$d\geq 3,\,\lambda>0,\,\mu\in(0,1)$ and~$\rho\in(0,1)$.
Let~$\beta=\rho/2$, let~$K$ be the constant given by Lemma~\ref{lemmahittingprob} applied in dimension~$d$
and let
\[
v
=
\left\lfloor
\frac{(1+\lambda)}{(1-\beta)K}
\right\rfloor
+1
\,.
\]
By definition of~$v$, for~$\alpha$ small enough we have
\[
(1+\lambda)(1-e^{-\alpha})
\leq
K\big(1-e^{-\alpha(1-\beta)v}\big)
\,.
\]
Let us fix such an~$\alpha$, and
let~$\kappa<(\alpha\beta\mu)\wedge(\rho^2\mu/2)$.

Then, given \smash{$n\geq\lceil(v/\mu)^{1/d}\rceil$} and~$A \subset \Z^d_n$ such that~$|A|\geq\mu n^d$, we consider
 the trivial hierarchy with only one level, defined by~$L = 0,\,A_0 = A$ and~$\cC_0 = \{A\}$
 and we apply Lemma~\ref{lemma6.2} with~$r=\infty$ (note that~$|A|\geq\mu n^d\geq v$).
 Thus, we obtain an~$A$-toppling procedure~$p_A$ such that, for every fixed~$U_0\subset A$ with~$|U_0|=\big\lfloor\beta|A|\big\rfloor$, the number~$\cN_{\cB}$ of times that the number of unstable sites passes above the threshold~$\beta|A|$ is such that~$1+\cN_{\cB}$ dominates a geometric random variable with parameter \smash{$e^{-\alpha(\beta |A|-2)}$} (applying a reasoning similar to that used in the proof of Lemma~\ref{lemma3.4}). Let us consider the variable~$J$ as defined in Section~\ref{sec:reformulated-goal}, along with the toppling strategy~$f$ and the sleep mask~$g^J$
derived from this hierarchy and these toppling procedures, as
explained in Section~\ref{section-def-f-g}. Note that $A \in \cC_0$ and thus $T(A,\,f,\,g^J,\,U_0)$ is independent of $J$, because $g^J$ has no influence on what happens at first level of the hierarchy during the stabilization of a cluster at level $0$.
Since~$T(A,\,f,\,g^J,\,U_0)\geq\cN_\beta$, we deduce that this stochastic domination also holds for~$T(A,\,f,\,g^J,\,U_0)$.

Then, to obtain the claimed result starting from a random initial
configuration~$U_0$ distributed as~$\cP_{A,\,\rho}$, we proceed as in
the end of the proof of Lemma~\ref{lemma3.4} using Hoeffding's
inequality to estimate~$\cP_{A,\,\rho}(|U_0|<\beta|A|)$.
We then obtain that, starting from~$U_0\sim \cP_{A,\,\rho}$, the variable~$T(A,\,f,\,g^J,\,U_0)$ dominates a geometric with parameter
\[
e^{-\alpha(\beta |A|-2)}
+
e^{\rho^2\mu n^d/2}
\leq
e^{-\kappa n^d}
\,,
\]
provided that~$n$ is taken large enough.
Decreasing~$\kappa$ so that the estimate holds for any~$n\in\N$, we obtain Proposition~\ref{prop-ARWD-stab-time-bis} in the transient case.
 \end{proof}

\section{Exponential stabilisation time on the torus: proof of Theorem~\ref{thm-exp-time} given Propositions~\ref{propdom} and~\ref{prop-ARWD-stab-time}}
\label{section-exp-time}

\begin{proof}[Proof of Theorem~\ref{thm-exp-time}]
Let~$d\geq 2$, let~$\lambda=(2d)^3$ and let~$\kappa$ be given by Proposition~\ref{prop-ARWD-stab-time} applied with~$\mu_0=1/5$ and~$\rho_0$ being the probability that a Poisson variable with parameter~$4/5$ is at least~$2$, i.e.,
\begin{equation}
\label{def-rho0}
\rho_0
=
1-\frac{1+4/5}{e^{4/5}}
\,.
\end{equation}
Let us consider the function
\[
\psi\,:\,\mu\in(0,1)
\ \mapsto\ 
-\mu\ln\mu-(1-\mu)\ln(1-\mu)
\,.
\]
Since~$\psi(\mu)\to 0$ when~$\mu\to 1$, we can find~$\mu_1\in(4/5,1)$ such that~$\psi(\mu_1)<\kappa/8$.

Let now~$\mu\in(\mu_1,1)$.
Let~$u>0$ be such that~$\varphi(u)>0$, where the function~$\varphi$ is given by
\[
\varphi\,:\,u\ \mapsto\ \mu (1-e^{-u})-\mu_1 u
\,.
\]
The existence of such a~$u$ follows from the fact that~$\varphi(0)=0$
and~$\varphi'(0)=\mu-\mu_1>0$.

 Let~$n\in\N$ be large enough so that
 \begin{equation}
 \label{assumption-n}
     \frac{(n-2)^d}{2}
     \geq
     \frac{2n^d}{5}
 \end{equation}
 and let~$\eta$ be a random SSM configuration on the torus~$\Z_n^d$ with i.i.d.\ Poisson numbers of particles on each site, with parameter~$\mu$.
Let~$M\in\N$.
Let us start by writing
\begin{equation}
\label{less-or-more}
    \cP^\mu\big(\|m_{\eta}\|<M\big)
    \leq
    P_\mu\big(|\eta|<\mu_1 n^d\big)
    +\sum_{k=\lceil \mu_1 n^d\rceil}^{n^d}
    \cP^\mu\big(\|m_{\eta}\|<M,\,|\eta|=k\big)
    \,.
\end{equation}
Note that if~$|\eta|> n^d$ then~$\|m_\eta\|=\infty$: this is why we
consider the sum only until~$k=n^d$.

To handle the first term in~\eqref{less-or-more}, recalling that~$|\eta|$ is a Poisson variable with parameter~$\mu n^d$, we use the following Chernoff bound:
\begin{equation}
    \label{Chernoff-bound}
    P_\mu\big(|\eta|<\mu_1 n^d\big)
    \leq
    E_\mu\big(e^{-u|\eta|}\big)
    e^{u \mu_1 n^d}
    =
    e^{-\varphi(u)n^d}
    \,.
\end{equation}
     We now deal with the second term of~\eqref{less-or-more}.
     Let us take~$\lceil \mu_1 n^d\rceil\leq k\leq n^d$.
     We first write
     \begin{equation}
     \label{sum-A}
   \cP^\mu\big(\|m_{\eta}\|<M,\,|\eta|=k\big)
   =
   \sum_{A\subset\Z_n^d\,:\,|A|=k}
   \cP^\mu\big(\|m_{\eta}\|<M,\,|\eta|=k,\,\eta_\infty=\mathbbm{1}_A\big)
   \,.
   \end{equation}
   Let us fix a set~$A\subset\Z_n^d$ with~$|A|=k$.
   Then, recall that Lemma~\ref{lemma-A-stab} tells us that if~$\eta_\infty=\mathbbm{1}_A$ then~$m_{\eta}=m_{\eta}^A$, whence
\begin{equation}
    \label{after-lemma-A-stab}
   \cP^\mu\big(\|m_{\eta}\|<M,\,|\eta|=k,\,\eta_\infty=\mathbbm{1}_A\big)
   \leq
   \cP^\mu\big(\|m_{\eta}^A\|<M,\,|\eta|=k\big)
   \,.
\end{equation}
   Then, conditioning on the initial configuration, we get
\begin{equation}
\label{conditioning-eta0}
\cP^\mu\big(\|m_{\eta}^A\|<M,\,|\eta|=k\big)
=
\sum_{\eta_0\in\N_0^V\,:\,|\eta_0|=k}
P_\mu\big(\eta=\eta_0\big)
\,\cP\big(\|m_{\eta_0}^A\|<M\big)
\end{equation}
We wish to apply Proposition~\ref{propdom}.
To this end, we need a totally disconnected subset of~$A$.
Considering the standard projection application~$\pi_n:\Z^d\to\Z_n^d$,
we consider the subset~$B_0$ of the ``even sites'' of the torus, removing a
line to avoid problems if~$n$ is odd, namely
\[
B_0
=
\pi_n\Big(\big\{(x_1,\,\dots,\,x_d)\in\{0,\,\dots,\,n-2\}^d\ :\ x_1+\cdots+x_d\text{ even}\big\}\Big)
\,.
\]
This set is totally disconnected and we have~$|B_0|\geq (n-2)^d/2\geq 2n^d/5$, following our assumption~\eqref{assumption-n} on~$n$.
Then, we simply consider the set~$B=A\cap B_0$, so that~$B$ is also
totally disconnected and satisfies~\smash{$|B|\geq |A|-3n^d/5\geq
n^d/5$}.
Let~$J$ and~$f^J$ be as in
Proposition~\ref{prop-ARWD-stab-time} applied with~$\lambda,\,\mu_0$
and~$\rho_0$ defined above, with~$\kappa$ already fixed above and
with this set~$B$, so that if~$U$ is distributed according
to~$\cP_{B,\,\rho_0}$
then~$1+T(B,\,f^J,\,U)$ dominates a
geometric with parameter~\smash{$\exp\big(-\kappa n^d\big)$}.

At this point, Proposition~\ref{propdom} entails that, for every fixed~$\eta_0\in\N_0^V$ with~$|\eta_0|=k$, we have
\[
\cP\big(\|m_{\eta_0}^A\|<M\big)
\leq
\bbP\big(T(B,\,f^J,\,U)<M\big)
\,,
\]
where~$U=U(\eta_0)=\{x\in B\,:\,\eta_0(x)\geq 2\}$ and~$T(B,\,f^J,\,U)$ is the stabilisation time defined in Definition \ref{def-stab-time-ARWD} for the~$\mathrm{ARWD}(\lambda,\,B,\,f^J,\,U)$ process (with the notation with subsets to represent the configurations, as presented in Section~\ref{sec:reformulated-goal}).
Plugging this into~\eqref{conditioning-eta0} leads to
\begin{align*}
\cP^\mu\big(\|m_{\eta}^A\|<M,\,|\eta|=k\big)
&\leq
\sum_{\eta_0\in\N_0^V\,:\,|\eta_0|=k}
P_\mu\big(\eta=\eta_0\big)
\,
\bbP\big(T(B,\,f^J,\,U(\eta_0))<M\big)
\\
&\leq
\sum_{\eta_0\in\N_0^V}
P_\mu\big(\eta=\eta_0\big)
\,
\bbP\big(T(B,\,f^J,\,U(\eta_0))<M\big)
\\
&=
\sum_{U\subset B}
\rho^{|U|}(1-\rho)^{|B\setminus U|}
\,
\bbP\big(T(B,\,f^J,\,U)<M\big)
\,,
\end{align*}
where~$\rho$ is the probability that a Poisson variable with parameter~$\mu$ is at least~$2$.
Note that since~$\mu\geq 4/5$ we have~$\rho\geq \rho_0$, with~$\rho_0$ defined by~\eqref{def-rho0}.
Now, since the toppling strategy~$f^J$ was provided by Proposition~\ref{prop-ARWD-stab-time}, we have
\[
\sum_{U\subset B}
\rho^{|U|}(1-\rho)^{|B\setminus U|}
\,
\bbP\big(T(B,\,f^J,\,U)<M\big)
\leq
\bbP\big(G_n\leq M\big)
\,,
\]
where~$G_n$ is a geometric variable with parameter~$\exp(-\kappa n^d)$.
Now, taking~$M=\lfloor e^{an^d}\rfloor$ with~$a=\kappa/2$, we have
\[
\bbP\big(G_n\leq e^{a n^d}\big)
=
1-\exp\big(\big\lfloor e^{\kappa n^d/2}\big\rfloor
\ln\big(1-e^{-\kappa n^d}\big)\big)
\leq
e^{-\kappa n^d/4}
\,,
\]
provided that~$n$ is large enough.
Combining all this and going back to~\eqref{sum-A} and~\eqref{after-lemma-A-stab}, we obtain
\[
\cP^\mu\big(\|m_{\eta}\|<e^{a n^d},\,|\eta|=k\big)
\leq
\binom{n^d}{k} e^{-\kappa n^d/4}
\,.
\]
Plugging this and~\eqref{Chernoff-bound} into~\eqref{less-or-more} then yields
\[
\cP^\mu\big(\|m_{\eta}\|<e^{a n^d}\big)
\leq
e^{-\varphi(u)n^d}
+
\sum_{k=\lceil \mu_1 n^d\rceil}^{n^d}
\binom{n^d}{k} e^{-\kappa n^d/4}
\leq
e^{-\varphi(u)n^d}
+
n^d\binom{n^d}{\lceil \mu_1 n^d\rceil}
e^{-\kappa n^d/4}
\,.
\]
Using now that
 \begin{equation*}
     \binom{n^d}{\lceil \mu_1 n^d \rceil}
=O\big(e^{\psi(\mu_1)n^d}\big)
     \,,
 \end{equation*}
along with our assumption that~$\psi(\mu_1)<\kappa/8$, we deduce that, taking~$b=\min(\varphi(u)/2,\,\kappa/16)$, provided that~$n$ is large enough,
\[
\cP^\mu\big(\|m_{\eta}\|<e^{a n^d}\big)
\leq
e^{-bn^d}
\,.
\]
Then, adapting  Section 3.10 of \cite{FG} to deal with the
continuous-time model with the exponential clocks, we obtain the
claimed property~\eqref{exponential-time}.
\end{proof}

 \section{Proof of Theorem~\ref{thm-torus-Zd}}

Theorem~\ref{thm-torus-Zd} is an analog of Theorem~4 of~\cite{FG},
which concerns activated random walks.
The proof extends to the SSM with only one minor change that we
explain below.

The proof of Theorem~4 of~\cite{FG} relies on
a criterion for non-fixation for ARW established by Rolla and Tournier
(Proposition~3 in~\cite{Rolla2}).
As explained by Rolla in~\cite{R}, this result extends to the SSM with
no substantial
change in the proof.
The proof of Theorem~4 of~\cite{FG} also uses a rough upper
bound for the stabilisation time of a box in ARW (Lemma~13
of~\cite{FG}), which also holds for the SSM, with the same
proof (apart from the point dealing with sleep instructions, but this
is harmless).

Then, the proof of Theorem~4 of~\cite{FG} extends to the SSM without
problem, apart from the fourth step of the procedure, which uses the
property that, in activated random walks, a configuration with~$k$
sites each containing one active particle, all other sites being stable,
can be stabilised with no particle moving, with
probability at least~$(\lambda/(1+\lambda))^k$: this simply happens if at each
of these~$k$ sites, the particle falls asleep before moving.

Thus, to adapt the proof to the SSM, it is enough to show
the following Lemma, which is an analog of this property for the SSM:

\begin{lemma}
For every graph~$G=(V,\,E)$ with maximal degree~$\Delta$, 
for every~$A\subset V$ finite and connected, for every
SSM configuration~$\eta:V\to\N_0$ with at most one particle on each site
of~$A$ and at least one empty site in~$A$, for every
odometer~$h:V\to(1/2)\N_0$, the probability that no particle exits~$A$
during the stabilisation of~$(\eta,\,h)$ in~$A$ with half-legal topplings
is at least~$\Delta^{-2|A|}$.
\end{lemma}

\begin{proof}
Let~$G=(V,\,E)$ be a graph with maximal degree~$\Delta$.
We prove the result by induction on~$|A|$.
If~$|A|=1$ then there is nothing to show, since we assume that at
least one site of~$A$ is empty.

Assume that~$k\in\N$ is such that the result holds for every~$A\subset
V$ with cardinality~$|A|=k$.

Then, let~$A\subset V$ connected with~$|A|=k+1$,
let~$\eta:V\to\N_0$ with at most one particle on each site
of~$A$ and at least one empty site in~$A$, and let~$h:V\to(1/2)\N_0$.
Let~$x$ and~$y$ be two distinct sites of~$A$ such
that both~$A\setminus\{x\}$ and~$A\setminus\{y\}$ remain connected
(take for example two leaves of a rooted spanning tree of~$A$, or the root
and the leaf if the tree has only one leaf).

If~$\eta(x)=\eta(y)=0$, then the result follows from the induction
hypothesis applied to~$A\setminus\{x\}$.

Assume now that at least one of these two points contains a particle:
assume for example that~$\eta(x)=1$.
If~$h(x)$ is integer, then the result also follows from the induction
hypothesis applied to~$A\setminus\{x\}$, because~$x$ is already
half-stable in~$(\eta,\,h)$.
Assume now that~$h(x)$ is not integer.
Let~$z$ be a neighbour of~$x$ in~$A$.
Then, we perform one half-legal half-toppling at~$x$.
With probability at least~$1/\Delta$, a particle jumps from~$x$ to~$z$.
Then, we perform an admissible half-toppling at~$z$.
With probability at least~$1/\Delta$, a particle jumps back from~$z$ to~$x$, and we
obtain the configuration~\smash{$(\eta,\,h+\frac 1 2 \mathbbm{1}_x+\frac 1
2 \mathbbm{1}_z)$}, in which the
site~$x$ is now half-stable.
Applying the induction hypothesis to this configuration
in~$A\setminus\{x\}$, we know that with probability at
least~$\Delta^{-2k}$,
the configuration stabilises in~$A\setminus\{x\}$ with no particle
jumping out of~$A\setminus\{x\}$.
Thus, with probability at least~$\Delta^{-2k-2}$, there exists an
admissible
half-toppling sequence which half-stabilises~$(\eta,\,h)$ in~$A$ with no
particle jumping out of~$A$.
This implies the result by \cite[Lemma 6]{RS}, noting that the number of
particles which jump out of~$A$ is a non-decreasing function of the
odometer.
\end{proof}

With the notation of the proof of Theorem~4 in~\cite{FG}, this lemma
enables to lower bound the
success probability of the fourth step of the procedure described
in~\cite{FG}, applying the lemma to~$A=B_0\setminus B_3$ (or to its
two connected components if~$d=1$), and adjusting the constants
accordingly.

\bibliographystyle{plain}
\bibliography{ssm.bib}

\end{document}